\documentclass[a4paper,11pt,reqno]{amsart}

\usepackage[utf8]{inputenc}
\usepackage[T1]{fontenc}
\usepackage{latexsym,amsmath,amsfonts,amscd,amssymb,amsthm,mathrsfs}
\usepackage{enumerate}
\usepackage[all]{xy}
\usepackage{color}
\usepackage{tikz}
\usepackage{graphicx}
\usepackage{caption}
\usepackage{subcaption}
\usepackage{esint}
	
\definecolor{cobalt}{RGB}{61,89,171}
\usepackage[colorlinks,citecolor=cobalt,linkcolor=cobalt,urlcolor=cobalt,pdfpagemode=UseNone,backref = page]{hyperref}

\usepackage{setspace}
\usepackage[top=1.5in, bottom=1.5in, left=0.9in, right=0.9in]{geometry}

\newcommand{\Z}{\mathbb{Z}}
\newcommand{\R}{\mathbb{R}}
\newcommand{\C}{\mathbb{C}}

\newcommand\mapsfrom{\mathrel{\reflectbox{\ensuremath{\mapsto}}}}

\renewcommand{\Re}{\mathsf{Re}\,}
\renewcommand{\Im}{\mathsf{Im}\,}

\theoremstyle{plain}
\newtheorem{theorem}{Theorem}[section]
\newtheorem{proposition}[theorem]{Proposition}
\newtheorem{lemma}[theorem]{Lemma}
\newtheorem{corollary}[theorem]{Corollary}

\theoremstyle{definition}
\newtheorem{definition}[theorem]{Definition}
\newtheorem{example}[theorem]{Example}
\theoremstyle{remark}
\newtheorem{remark}[theorem]{Remark}

\def\XXint#1#2#3{{\setbox0=\hbox{$#1{#2#3}{\int}$ }
\vcenter{\hbox{$#2#3$ }}\kern-.6\wd0}}

\def\dashint{\Xint-}

\renewcommand{\dashint}{\fint}

\allowdisplaybreaks[1]

\begin{document}

\title{K\"ahler-Einstein metrics: Old and New}

\author[D. Angella]{Daniele Angella}
\address[D. Angella]{
	Dipartimento di Matematica e Informatica "Ulisse Dini"\\
	Università di Firenze\\
	viale Morgagni 67/a\\
	50134 Firenze\\
	Italy}
\email{daniele.angella@gmail.com}
\email{daniele.angella@unifi.it}
	
\author[C. Spotti]{Cristiano Spotti}
\address[C. Spotti]{
QGM, Centre  for Quantum Geometry of Moduli Spaces\\
Aarhus University\\
Ny Munkegade 118\\
8000 Aarhus C\\
Denmark
}
\email{cristiano.spotti@gmail.com}
\email{c.spotti@qgm.au.dk}

\thanks{D.A. is supported by the Project SIR2014 AnHyC "Analytic aspects in complex and hypercomplex geometry" (code RBSI14DYEB), by Projecti PRIN2015 "Varietà reali e complesse: geometria, topologia e analisi armonica", and by GNSAGA of INdAM.\\
C.S. is supported by AUFF Starting Grant 24285}

\date{\today}

\begin{abstract}
We present classical and recent results on K\"ahler-Einstein metrics on compact complex manifolds, focusing on existence, obstructions and relations to algebraic geometric notions of stability (K-stability). These are the notes for the SMI course "K\"ahler-Einstein metrics" given by C.S. in Cortona (Italy), May 2017. The material is not intended to be original.
\end{abstract}

\maketitle

\section*{Introduction}

Let us begin by presenting our main actors. The most "geometric" way to introduce them is as follow. Connected $2n$-dimensional Riemannian manifolds $(M^{2n},g)$ are called \emph{K\"ahler-Einstein} (KE)  if, as the name suggests, they are: 
\begin{description}
\item[K\"ahler] the  holonomy  $\mathrm{Hol}(M^{2n},g)$ is contained  in the unitary group $ \mathrm{U}(n);$
\item[Einstein] the Ricci curvature satisfies $\mathrm{Ric}(g)=\lambda g$ for $\lambda\in\mathbb{R}$.
\end{description}
 In these lectures we are going to focus on  \emph{compact} KE  manifolds. There are several possible motivations for the study such spaces:
\begin{enumerate}
\item For $n=1$, KE metrics are precisely metrics with constant Gauss curvature on compact oriented surfaces. By the classical {\em Uniformization Theorem}, a jewel of XIX century mathematics due to works of  Riemann, Poincar\'e and Koebe, three cases occur: 
  \begin{itemize}
  \item if $\lambda > 0$, then the manifold is the round  sphere $(\mathbb{S}^2, g_{\text{round}})$;
  \item if $\lambda = 0$, then the manifold is a flat torus $(\mathbb{T}^2, g_{\text{flat}})$; 
  \item if $\lambda < 0$, then the manifold is a compact hyperbolic surface $C_g \simeq \mathbb{H}^2 \slash \Gamma$ where $\Gamma \subseteq \mathrm{PSL}(2;\mathbb R)$ is a discrete subgroup that acts freely on $\mathbb H^2$.
  \end{itemize}
  It is important to remark that, despite the \emph{local} isometry class of such metrics depends only on the value of the constant Gauss curvature, there are, up to scalings, a real $2$-dimensional continuous family  of distinct ({\itshape i.e.}, non \emph{globally} isometric) flat tori, and a $6g-6$ family  (the famous Riemann's \emph{moduli} parameters) of distinct hyperbolic metrics on a surface of genus $g\geq2$. KE manifolds are the most natural higher dimensional generalizations of such important geometric spaces.
  
\item In a curvature hierarchy, Einstein metrics lie between the better understood cases of full constant curvature and  constant scalar curvature (such metrics exist in any conformal class, by the solution of the Yamabe problem). In real dimension bigger than three, Einstein metrics no longer  have   a unique local isometry model but they still  form \emph{finite} dimensional moduli spaces. They are also good candidates for being "canonical" metrics  on a space and they are critical points of the natural {\em Einstein-Hilbert functional} $g\mapsto \int \mathrm{Sc}_g dV_g$, over $\int dV_g=1$.  Thus KE metrics provide an important example  of such quite mysterious class of  Einstein spaces to study.

\item The K\"ahler condition corresponds to the  holonomy group $\mathrm{U}(n)$ in the Berger's classification of reduced holonomy group of irreducible non-locally symmetric spaces. The other are the "generic" $\mathrm{SO}(2n)$, the Calabi-Yau $\mathrm{SU}(n)$, the hyper-K\"ahler $\mathrm{Sp}(n)$, the quaternionic-K\"ahler $\mathrm{Sp}(1)\mathrm{Sp}(n)$ and the special groups $\mathrm{G}_2$ and $\mathrm{Spin}(7)$. Except than $\mathrm{SO}(2n)$ and $\mathrm{U}(n)$, the other cases are Einstein. Thus it is natural to look what happens when we combine the relatively soft  $\mathrm{U}(n)$ holonomy condition with the Einstein property: as we will see, such holonomy condition actually provides a great simplification in studying the Einstein equation. 
\end{enumerate}

The K\"ahler condition implies that there is a natural structure of complex  manifold $X^n (\cong M^{2n})$  which is compatible with the metric. The K\"ahler-Einstein property implies even more:
\begin{itemize}
\item if $\lambda>0$, then $c_1(X)>0$; that is $X^n$ is {\em Fano};
\item if  $\lambda=0$, then $c_1(X)=0$; that is $X^n$ is {\em Calabi-Yau};
\item if  $\lambda<0$, then $c_1(X)<0$; that is  $X^n$ is {\em with ample canonical bundle}, in particular, it is {\em of general type}.
\end{itemize}

These mutually exclusive conditions are very strong and most of the complex manifolds are not of the above types (anyway, such special manifolds are, in a certain sense, supposed to be fundamental  "building blocks" for complex (K\"ahler) manifolds).

In analogy with the uniformization of surfaces/complex curves  mentioned above (note that in such case $c_1=2-2g$), it is very natural to ask if all compact complex manifolds whose first Chern class "has a sign" can be equipped with KE metrics. This and a related question was posed by Calabi in the $'50$s (the famous \emph{Calabi's conjectures}, \cite{calabi-ICM, calabi-57}). However, it was immediately realized that some issues occur: Matsushima \cite{matsushima} found an obstruction for the case $\lambda>0$ in terms of reductivity of the Lie algebra of holomorphic vector fields. In particular, very simple Fano manifolds, such as the blow-up in a point of the projective plane, cannot admit KE metric. 

 However, for non-positive first Chern class, the situation turned out to much nicer. Actually, the  best it could be: 

\begin{theorem}[{Aubin \cite{aubin-CRAS, aubin-BSM}, Yau \cite{yau-PNAS, yau-CPAM}}]
Let $X^n$ be a compact complex manifold.
If $c_1(X)<0$, then there exists up to scaling a unique K\"ahler-Einstein metric on $X^n$.
\end{theorem}

\begin{theorem}[{Yau \cite{yau-PNAS, yau-CPAM}}]
Let $X^n$ be a compact complex K\"ahler manifold.
If $c_1(X)=0$, then there exists a unique K\"ahler-Einstein metric in any K\"ahler class on $X^n$.
\end{theorem}

In the Fano case, the situation remained unclear. Starting from the $'80$s there have been several discoveries of new obstructions: Futaki's invariant  \cite{futaki}, Tian's examples of a smooth Fano $3$-fold with no non-trivial holomorphic vector fields with not KE metric  \cite{tian-Inv} based on obstructions coming from algebro-geometric notions of stability conditions (see also, {\itshape e.g.} \cite{donaldson-JDG05}). Furthermore, some existence criteria and non-trivial examples were  found in the last $30$ years:  $\alpha$-invariant existence criterion \cite{tian-Inv87}, study of highly symmetric cases ({\itshape e.g.} toric varieties  \cite{wang-zhou}), and a  complete understanding in dimension two  \cite{tian-Invent90}.

Motivated by the above important results and in analogy with Hitchin-Kobayashi correspondence for holomorphic vector bundles \cite{donaldson-Duke87, uhlenbeck-yau, li-yau, lubke-teleman}, it was conjectured, e.g.  \cite{tian-Inv, donaldson-JDG}), that the relations of the existence problem with some (to be understood) algebro-geometric notion of stability is indeed the crucial aspect to completely characterize, in purely algebro-geometric terms, the existence of KE metrics on Fano manifolds (and, even more generally, for constant scalar curvature K\"ahler metrics). Such conjecture is known as \emph{Yau-Tian-Donaldson (YTD) conjecture}. By refining the Futaki's invariant, the notion of \emph{K-stability} has been introduced by Tian \cite{tian-CAG, tian-Inv}, and refined by Donaldson \cite{donaldson-JDG}. This notion turned out to be the correct one:

\begin{theorem} [{Yau-Tian-Donaldson conjecture for the Fano case; "if" Chen-Donaldson-Sun \cite{chen-donaldson-sun-IMRN, chen-donaldson-sun}}, "only if" Berman \cite{berman-Inv}]\label{CDS}
A Fano manifold $X^n$ admits a K\"ahler-Einstein metric if and only if $X$ is K-polystable.
\end{theorem}

The "only if" direction has been studied by Tian \cite{tian-Inv}, Donaldson \cite{donaldson-JDG05}, Mabuchi \cite{mabuchi-arXiv}, Stoppa \cite{stoppa-AdvMath}, and in full generality by Berman \cite{berman-Inv}. The "if" direction is the celebrated breakthrough result obtained in 2012 by Chen-Donaldson-Sun \cite{chen-donaldson-sun-IMRN, chen-donaldson-sun}.

In these six lectures we are going to describe some aspects of the above outlined  "KE story". Due to the vastness of the subject, there is going to be plenty of omission of fundamental results. We apologize in advance. The first lectures are supposed to be more elementary and more detailed. In approaching the final lecture the "hand waving" is going to become more and more dominant. In any case, we hope that we at least succeed in describing some of the ideas  behind the proofs of  important results in the field, and that such "road map" can be of some use for the ones that would like to proceed further in the study by reading  the original papers.

This is a more detailed description of the actual content of the lectures.

\begin{itemize}
	\item Lecture $1$. Review of basic K\"ahler geometry with emphasis on  curvature formulas.
	\item Lecture $2$. Calabi's conjectures: preliminaries, their proofs and some consequences.
	\item  Lecture $3$. Obstructions to the existence of KE metrics on Fanos: Matsushima's theorem and Futaki's invariant. Mabuchi's energy and geodesics: the "fake" proof of uniqueness.
	\item Lecture $4$. First results on existence of KE metric on Fanos: energy functionals, properness and the  $\alpha$-invariant criterion.
	\item Lecture $5$. Towards K-stability: review of moment maps and GIT, Donaldson-Fujiki's infinite dimensional picture, definition of K-stability.
	\item Lecture $6$. YTD conjecture for Fanos: descriptions of  Berman and Chen-Donaldson-Sun proofs. Compact Moduli spaces of KE Fanos.
\end{itemize}

\bigskip

{\small
\noindent{\sl Acknowledgments.}
This note resumes the contents of the SMI school "Kaehler-Einstein metrics" given by C.S. in Cortona on April 30th--May 06th, 2017, see \url{https://gecogedi.dimai.unifi.it/event/404/}. The event has been organized by the two authors together with Simone Calamai, Graziano Gentili and Nicoletta Tardini, with the support by SMI Scuola Ma\-te\-ma\-ti\-ca Interuniversitaria, Dipartimento di Ma\-te\-ma\-ti\-ca e Informatica "Ulisse Dini" of Università di Firenze, SIR2014 AnHyC "Analytic aspects in complex and hypercomplex geometry" (code RBSI14DYEB), GNSAGA of INdAM.
The authors warmly thank all the participants for the wonderful environment they contributed to, and the secretaries for their fundamental assistance during the organization. In particular, thanks to Gao Chen, Paul Gauduchon, Long Li, Alexandra Otiman, Chengjian Yao and Kai Zheng for accepting the invitation to give seminar lectures.
Many thanks also to the anonymous Reviewer, whose suggestions improved the presentation of the paper.
}

\section{Curvature formulas on K\"ahler manifolds}
In this preliminary section, we recall some basic facts concerning differential complex geometry and curvature of K\"ahler metrics; references are {\itshape e.g.} \cite{gabor-book, tian-book, moroianu, demailly-agbook, voisin-1, gauduchon-book, huybrechts, kodaira-morrow, griffiths-harris, kodaira-book, wells-book, ballmann, kobayashi-nomizu-2}.

\subsection{Basic complex geometry}

Let $X^n$ be a {\em complex manifold}, {\itshape i.e.} a smooth manifold with an atlas whose transition functions are holomorphic. Examples include $\mathbb{CP}^n$, hypersurfaces in $\mathbb{CP}^n$ namely zero-set of a homogeneous polynomial, tori $\mathbb{C}^n\slash \Gamma$ where $\Gamma$ is a discrete subgroup of maximal rank. We denote by $M^{2n}$ the underlying smooth manifold, and we use local holomorphic coordinates $(z_j=x_j+\sqrt{-1}y_j)_j$, so $(x_j,y_j)_j$ are local differential coordinates and $(\partial_{x_j}, \partial_{y_j})_j$ denotes the corresponding local frame for the tangent bundle.

Multiplication by $\sqrt{-1}$ defines a tensor $J \in \mathrm{End}(TM^{2n})$ as $J \partial_{x_j}=\partial_{y_j}$ and $J\partial_{y_j}=-\partial_{x_j}$, which is globally defined and with the property that $J^2=-\mathrm{id}$. Moreover, if we define the {\em Nijenhuis tensor} of such an endomorphism $J$ as
$$ N_J(Y,T) := \frac{1}{4} \left( [JY,JT] - J[JY,T] - J[Y,JT] - [Y,T] \right), $$
for $Y,T \in TM^{2n}$,
then clearly we have $N_J=0$.
Conversely, by the Newlander and Nirenberg theorem \cite{newlander-nirenberg}, if $M^{2n}$ is a smooth manifold and $J\in\mathrm{End}(M^{2n})$ is such that $J^2=-\mathrm{id}$ (say that $J$ is an {\em almost-complex} structure), then $J$ is {\em integrable} ({\itshape i.e.} it is induced by a holomorphic atlas) if and only if $N_J=0$.

Consider the eigenspaces splitting
$$ TM\otimes {\mathbb C} = TM^{(1,0)} \oplus TM^{(0,1)} = \langle \partial_{z_1}, \ldots, \partial_{z_n} \rangle \oplus \langle \partial_{\bar z_1}, \ldots, \partial_{\bar z_n} \rangle , $$
where
$$
\partial_{z_j} = \frac{1}{2} \left( \partial_{x_j} - \sqrt{-1} \partial_{y_j} \right) ,
\qquad
\partial_{\bar z_j} = \frac{1}{2} \left( \partial_{x_j} + \sqrt{-1} \partial_{y_j} \right) .
$$
Then we can identify
$$ (TM,J) \stackrel{\simeq}{\to} (TM^{(1,0)}, \sqrt{-1}) $$
by means of the maps
$$ Y \mapsto Y^{(1,0)}:=\frac{1}{2}(Y-\sqrt{-1}JY), \qquad 2 \Re v = v+\bar v \mapsfrom v . $$
A {\em holomorphic vector field} is a holomorphic section of $TM^{(1,0)}$; locally, it is given by $v\stackrel{\text{loc}}{=}\sum v_j \partial_{z_j}$ where $v_j$ are holomorphic function, {\itshape i.e.} smooth  functions such that $\partial_{\bar z_k}v_j=0$ for any $k$. We see easily that a smooth vector field  $v \in TM^{(1,0)}$ is holomorphic if and only if $\mathcal{L}_{\Re v}J=0$ ({\itshape i.e.} $\Re v$ is \emph{real holomorphic}); see {\itshape e.g.} \cite[page 23]{gauduchon-book}. 

Similarly, one can consider the splitting at the level of differential forms:
$$ TM^\vee \otimes \mathbb C = \wedge^{1,0}M \oplus \wedge^{0,1}M = \langle dz_1 , \ldots, dz_n \rangle \oplus \langle d\bar z_1, \ldots, d\bar z_n \rangle $$
where
$$ d z_j = dx_j+\sqrt{-1}dy_j, \qquad d\bar z_j = dx_j - \sqrt{-1} dy_j . $$
Then
$$ \wedge^j TM^\vee \otimes \mathbb C = \bigoplus_{p+q=j} \wedge^{p,q}M , $$
that is, locally, a {\em $(p,q)$-form} can be expressed as $\eta \stackrel{\text{loc}}{=} \sum f_{I\bar J} dz_{i_1}\wedge\cdots\wedge dz_{i_p}\wedge d\bar z_{j_1}\wedge\cdots\wedge d\bar z_{j_q}$, where $f_{I\bar J}=f_{i_1, \ldots, i_p, \bar j_1, \ldots, \bar j_q}$ are smooth functions. We have $df=\partial f+\overline\partial f$ on functions, and this extends to forms. By $d^2=0$, we get $\partial^2=\overline\partial^2=\partial\overline\partial+\overline\partial\partial=0$. 

 The {\em Dolbeault cohomology} is defined as
$$ H^{p,q}_{\overline\partial}(X) = \frac{\ker(\overline\partial\colon \wedge^{p,q}X \to \wedge^{p,q+1}X)}{\mathrm{im}(\overline\partial\colon \wedge^{p,q-1}X \to \wedge^{p,q}X)} . $$
When $X$ is compact, $h^{p,q}:=\dim_{\mathbb{C}} H^{p,q}(X)<+\infty$ are called the {\em Hodge numbers}. Note that $h^{p,0}$ is just the dimension of the space of holomorphic $p$-forms ({\itshape e.g.}, $h^{1,0}=g$ for compact  complex curves of genus $g$). More generally, for people who know a bit of sheaf theory, if we denote by $\mathcal{A}^{p,q}_X$ the sheaf of germs of smooth $(p,q)$-forms, and define the sheaf of germs of {\em holomorphic $p$-forms} as $\Omega^p_X:=\ker(\overline\partial \colon \mathcal{A}^{p,0}_X \to \mathcal{A}^{p,1}_X)$, then it is a standard result that $H^{p,q}(X)\cong H^q(X, \Omega^p_X)$, where the last denotes sheaf cohomology \cite{dolbeault}, see {\itshape e.g.} \cite{voisin-1}.

\subsection{K\"ahler manifolds}

A {\em Hermitian metric} $h$ on $X^n$ is a smoothly varying family of Hermitian products on $TX=TM^{(1,0)}$. Namely, for any $p\in X$, it holds $h_p(v,v)>0$ for any $v\neq0$ and $h_p(v,w)=\overline{h_p(w,v)}$; locally, $h\stackrel{\text{loc}}{=}\sqrt{-1}\sum h_{j\bar k} dz^j\otimes d\bar z^k$ where $(h_{j\bar k})$ is a positive-definite Hermitian matrix of functions.

A Hermitian manifold is called {\em K\"ahler} if the associated $(1,1)$-form $\omega\stackrel{\text{loc}}{:=} \sqrt{-1} \sum  h_{j\bar k} dz_j\wedge d\bar z_k$ is closed:
$$ d\omega=0 , $$
that is, $\omega$ is symplectic. Equivalently if the following PDE is satisfied:
$$
\frac{\partial h_{j\bar k}}{\partial z_\ell} = \frac{\partial h_{\ell \bar k}}{\partial z_j} , \qquad
\frac{\partial h_{j\bar k}}{\partial \bar z_\ell} = \frac{\partial h_{j \bar \ell}}{\partial \bar z_k} .
$$
Examples include:
\begin{itemize}
\item $(\mathbb C^n, \omega_{\text{flat}})$ with $\omega_{\text{flat}}=\sum \frac{\sqrt{-1}}{2} dz_j \wedge d\bar z_j$ for $(z_j)$ standard coordinates; and $\mathbb C^n \slash \Gamma$ flat tori.
\item $(\mathbb{CP}^n, \omega_{FS})$ where the Fubini-Study metric is
$$ \omega_{FS}\stackrel{\text{loc}}{=} \sqrt{-1}\partial\overline\partial \log\left( 1+\sum \frac{|Z_j|^2}{|Z_0|^2} \right) $$
away from $\{Z_0=0\}$, where $(Z_j)$ are homogeneous coordinates on $\mathbb{CP}^n$.
For example, for $\mathbb{CP}^1$, we get the round metric $\omega = \sqrt{-1} \partial\overline\partial \log (1+|\xi|^2) = \frac{\sqrt{-1}}{(1+|\xi|^2)^2} d\xi\wedge d\bar\xi >0$. 
We remark here that, more generally, $(\mathbb{CP}^n,\omega_{FS})$ has constant  \emph{holomorphic sectional curvature}, that is,  the sectional curvature is constant on \emph{complex lines} in $TX$. See later for definitions.

\item $(\mathbb{B}, \omega_{\mathbb B})$ where $\mathbb{B}\subset \mathbb C^n$ is the unit ball and $\omega_{\mathbb B}=\sqrt{-1}\partial\overline\partial\log(1-\sum |z_j|^2)$ for $(z_j)$ standard coordinates.
\item Submanifolds of K\"ahler manifolds are naturally  K\"ahler. In particular, algebraic manifolds (here meaning zeros of homogeneus polynomial in $\mathbb{CP}^n$) are K\"ahler. Note that submanifolds of $\mathbb{CP}^n$ with restricted $\omega_{FS}$ are  almost never K\"ahler-Einstein for the induced metric.
\end{itemize}

We compare the real and the complex viewpoints. Starting from $(X^n,h)$ Hermitian, we can define a Riemannian metric as
$$ g(Y,Z):=2\Re h(Y^{(1,0)},Z^{(1,0)}) $$
which is {\em Hermitian} in the sense that $J$ is an isometry: $g(JY,JZ)=g(Y,Z)$.
Conversely, we recover $h_{j\bar k}=\frac{1}{2} g^{\mathbb C} (\partial_{z_j},\partial_{\bar z_k})$.
Moreover,
$$ \omega=g(J\_,\_) . $$
In the following, we will confuse $h$, $\omega$, $g$ on a complex manifold $X^n$.

The K\"ahler condition $d\omega=0$ gives a K\"ahler class $[\omega]\in H^2(X;\mathbb R)$.
Notice that, at $p$, we can always assume that $h_{j\bar k}(p)=\delta_{jk}$. From here, one can easily see that $dV_g=\frac{\omega^n}{n!}$. This implies that, if $(X^n,\omega)$ is compact K\"ahler, then the volume is cohomological: $\mathrm{Vol}_\omega=\mathrm{Vol}_{[\omega]}$, and the even Betti numbers are positive: $b_{2k}>0$ for any $k \in \{0, \ldots, n\}$.

We give now another interpretation of the K\"ahler condition. Take $(M^{2n},J,g)$ where $g$ is a Riemannian metric and $J\in\mathrm{End}(TM)$ a compatible almost complex structure, {\itshape i.e.}  a pointwise  isometry such that $J^2=-\mathrm{id}$.
\begin{proposition}\label{LC}
Let $g$ be a Riemannian metric and $J$ be a compatible almost complex structure on $M^{2n}$.
Then $g$ is K\"ahler if and only if $\nabla^{LC}J=0$, for the Levi Civita connection $\nabla^{LC}$.
\end{proposition} 
Geometrically, this means that in the K\"ahler case, complex rotation by $J$ commutes with parallel transport. That is, we immediately get:
\begin{corollary}
Let $g$ be a Riemannian metric and $J$ be a compatible almost complex structure on $M^{2n}$.
Then $g$ is K\"ahler if and only if $\mathrm{Hol}(M^{2n},g)\subseteq \mathrm{U}(n)$.
\end{corollary}

The proof of Proposition \ref{LC} goes as follow. For the "$\Leftarrow$" direction, using that $\nabla^{LC}$ is torsion free, we can rewrite the Nijenhuis  tensor as $ 4 N_J(Y,T)=(\nabla^{LC}_{JY}J)(T)-J((\nabla^{LC}_{JY}J)(T))-(\nabla^{LC}_{JT}J)(Y)+J((\nabla^{LC}_{T}J)(Y))$. Thus if $J$ is parallel, $N_J$ vanishes and $J$ is integrable. Moreover, by expressing the differential of  the associated two form   $\omega=g(J\_,\_)$  using the covariant derivative 
$ d\omega(Y,T,Z)=(\nabla^{LC}_Y\omega)(T,Z)-(\nabla^{LC}_T\omega)(Y,Z)+(\nabla^{LC}_Z\omega)(Y,T),$ we see that $\omega$ must be closed since $\omega$ is parallel, being so $g$ and $J$. That is, $g$ is K\"ahler. Conversely, it is easy (but slightly long) to check that  
$ 2 g((\nabla^{LC}_XJ)Y,Z) = d\omega(X, Y, Z) - d\omega(X, JY, JZ) - g (N_J(Y,JZ), X)$. Hence $J$ has to be parallel if $J$ is integrable and $\omega$ is closed.

Finally, a fundamental tool in the following section will be the:
\begin{theorem}[{$\sqrt{-1}\partial\overline\partial$-Lemma \cite[pages 266--267]{deligne-griffiths-morgan-sullivan}}]
Let $X^n$ be a compact K\"ahler manifold, and let $\eta$ be a real $(1,1)$-form such that $\eta=d\alpha$ for some $1$-form $\alpha$. Then there exists $f\in\mathcal{C}^\infty(X;\mathbb R)$ such that $\eta=\sqrt{-1}\partial\overline\partial f$.
\end{theorem}

\begin{proof}
We briefly review the proof. Write $\alpha=\beta+\bar\beta$ for $\beta\in\wedge^{1,0}X$ and note that $\partial\beta=0$. Consider $\partial^*$, the $L^2$-adjoint of $\partial$ with respect to the $L^2$-pairing $(\varphi,\psi)_{L^2}:=\int_X g^{p \bar j} g^{k\bar q} \varphi_{p\bar q}\overline{\psi_{j \bar k}} dV_g$. Here Einstein's summation convention to be understood and $g^{p \bar{j}}$ is the inverse of the Hermitian metric $g$.  As for $\Delta_d=dd^*+d^*d$, we can analogously define the Laplacians $\Delta_\partial$ (and $\Delta_{\bar\partial})$. A fundamental fact is that for a K\"ahler metric, $\frac{1}{2}\Delta_d= \Delta_\partial=\Delta_{\bar\partial}$ (this is a consequence of  the so-called K\"ahler identities, see {\itshape e.g.} \cite[Theorem 6.7]{voisin-1}, \cite[Proposition 3.1.12]{huybrechts}).

Since we have $\int_X \partial^*\beta \frac{\omega^n}{n!}=0$,  we can solve the equation
$$ \frac{1}{2} \Delta_d\varphi = -\partial^*\beta $$
for $\varphi$ smooth. Thus we get
$ \partial^*(\beta+\partial\varphi)=0$ and $ \partial(\beta+\partial\varphi)=0$,
that is, $\beta+\partial\varphi$ is $\partial$-harmonic. Since  $\Delta_\partial=\Delta_{\overline\partial}$, $\beta+\partial\varphi$ is also $\overline\partial$-harmonic; in particular, it is $\overline\partial$-closed: $\overline\partial\beta=\partial\overline\partial\varphi$. Finally, $\eta=d\alpha=\overline\partial\beta+\partial\overline\beta=\frac{\sqrt{-1}}{2}\partial\overline\partial\Im \varphi$.
\end{proof}

This lemma allows to identify the space $\mathcal{K}_{[\omega]}$ of K\"ahler metrics in the K\"ahler class $[\omega]\in H^2(X;\mathbb R)$ on $X^n$ with the set
$$ \left\{ \omega+\sqrt{-1}\partial\overline\partial\varphi>0 \;:\; \varphi\in\mathcal{C}^\infty(X;\mathbb R) \right\}, $$
once fixed some background metric $\omega$.
This is the first ingredient that explains why the K\"ahler-Einstein problem is easier than the general Einstein problem: in order to find a KE metric in a fixed cohomology class we need to look for a \emph{function} $\varphi$ and not for a \emph{tensor}.
 The other reason is that there is a very nice formula for the Ricci curvature of a K\"ahler metric, as we are going to explain.

\subsection{Curvatures of K\"ahler metrics}

Consider $(M^{2n},g)$. In this section, we denote by $\nabla:=\nabla^{LC}_g$ its Levi Civita connection, and by $\nabla^{\mathbb C}$ its complexification to $TM\otimes\mathbb C$. Note that $\nabla^{\mathbb C}_{\bar v}\bar w = \overline{\nabla^{\mathbb C}_v w}$. The complexified {\em Christoffel symbols} are defined as $ \nabla_{\partial_{z_i}}^{\mathbb C} \partial_{z_j} = \sum_k\left( \Gamma_{ij}^k \partial_{z_k} + \Gamma_{ij}^{\bar k} \partial_{z_{\bar k}}\right)$ and similar. See {\itshape e.g.} \cite{gabor-book}, \cite[Section 1.2]{tian-book} for more details on what follows.

\begin{lemma}\label{lemma:christoffel-kahler}
Let $g$ be a K\"ahler metric on $M^{2n}$. Then the only non-zero Christoffel symbols are $\Gamma_{ij}^k$ and $\Gamma_{\bar i \bar j}^{\bar k} = \overline{\Gamma_{ij}^k}$.
Moreover, $\Gamma_{ij}^{k}=g^{k\bar\ell}\partial_{z_j}g_{i\bar \ell}$.
\end{lemma}

Let us check $\Gamma_{ij}^{\bar k}=0$, for the other is similar.
Since $\nabla J=0$, then $\nabla^{\C}_{\partial_{z_i}} (J^{\C} \partial_{z_j})= J^{\C}(\nabla^{\C}_{\partial_{z_i}} ( \partial_{z_j}))$, that is: $\sqrt{-1} \sum_k\left( \Gamma_{ij}^k \partial_{z_k} + \Gamma_{ij}^{\bar k} \partial_{z_{\bar k}}\right)=  \sum_k\left( \sqrt{-1} \Gamma_{ij}^k  \partial_{z_k} - \sqrt{-1}  \Gamma_{ij}^{\bar k} \partial_{z_{\bar k}}\right)$. Thus, $\Gamma_{ij}^{\bar k}=0$.  Next,  using $\nabla g=0$ and the vanishing of mixed Christoffel, $\partial_{z_k} g_{i\bar{j}}=g(\nabla_{\partial_{z_k}}\partial_{z_i}, \partial_{\bar{z}_j})=g (\Gamma_{ki}^{l} \partial_{z_l}, \partial_{\bar{z}_j}) = \Gamma_{ki}^{l} g_{l \bar{j}}$. By multiplying with the inverse the proof of the lemma is concluded.

We now consider the {\em Riemannian curvature} operator:
$$ \mathrm{Rm}^{\C}(v,w)z := \nabla^{\mathbb C}_v \nabla^{\mathbb C}_w z - \nabla^{\mathbb C}_w \nabla^{\mathbb C}_v z - \nabla^{\mathbb C}_{[v,w]} z . $$
Clearly, $\mathrm{Rm}(\bar v, \bar w)\bar z = \overline{\mathrm{Rm}(v,w)z}$.
By using that $\nabla J=0$, we have that $\mathrm{Rm}(v,w)Jz = J\mathrm{Rm}(v,w)z$. This implies that $\mathrm{Rm}(z,w)(TM^{1,0})\subseteq TM^{1,0}$, and also $\mathrm{Rm}(TM^{1,0},TM^{1,0})=0$ thanks to the symmetries of $\mathrm{Rm}$.
This means that the only interesting coefficients of $\mathrm{Rm}$ are
$$ R_{i\bar j k \bar \ell} := g(\mathrm{Rm}(\partial_{z_i}, \partial_{\bar z_j})\partial_{z_k}, \partial_{\bar z_{\ell}}) $$
and $R_{\bar ij \bar k \ell}=\overline{R_{i\bar j k \bar \ell}}$.
We have a nice expression for these coefficients:

\begin{proposition}\label{lemma:riem-kahler}
Let $g$ be a K\"ahler metric on $M^{2n}$. Then $ R_{i\bar j k \bar \ell} = -g_{i\bar j, k\bar \ell} + g^{m\bar n} g_{i\bar n,k}g_{m\bar j,\bar \ell} .$
\end{proposition}
We have denoted with  $g_{i\bar j, k\bar \ell}:=\partial_{z_k \bar z_\ell}^2 g_{i\bar j}$ (and similar) for short. Here the proof: $$R_{i\bar j k \bar \ell} = -g(\nabla_{\bar j}(\nabla_i \partial_k), \partial_{\bar \ell})=- g(\nabla_{\bar j}(\Gamma_{ik}^m \partial_m), \partial_{\bar \ell}) \stackrel{\text{Leibniz}}{=} -g (\Gamma_{ik,\bar j}^m \partial_m, \partial_{\bar \ell})= \Gamma_{ik,\bar j}^m g_{m \bar \ell}.$$
But $\Gamma_{ik,\bar j}^m =-\partial_{\bar j}(g^{m \bar q}g_{i \bar{q}, k})={g^{m\bar{q}}}_{,\bar j}g_{i \bar{q}, k}+ g^{m \bar q}g_{i \bar{q}, k \bar l}$. Since ${g^{i \bar j}}_{, \bar k}=-g^{i\bar m}g^{l \bar j} g_{l \bar m, \bar k}$, we get
$$R_{i\bar j k \bar \ell} =-g_{m\bar{\ell}}(-g^{m\bar{r}} g ^{s \bar q}g_{s \bar{r},\bar{j}}g_{i\bar q, k} + g^{m \bar{q}} g_{i \bar{q}, n \bar{j}})= -g_{i\bar j, k\bar \ell} + g^{m\bar n} g_{i\bar n,k}g_{m\bar j,\bar \ell}, $$
where for the last step we have used the K\"ahler property and renamed the variables.

Define now the {\em Ricci tensor}
$$ \mathrm{Ric}_{i \bar j} := g^{k\bar \ell} R_{i\bar j k \bar \ell} , $$
and the {\em scalar curvature}
$$ \mathrm{Sc} := g^{i\bar j} \mathrm{Ric}_{i\bar j} . $$

For K\"ahler metrics, we have the following nice expression: 
\begin{proposition}\label{prop:ricci-kahler}
Let $g$ be a K\"ahler metric on $M^{2n}$. Then $ \mathrm{Ric}_{i\bar{j}} = -\partial^2_{i {\bar j}} \log \det (g_{k\bar \ell}) . $
\end{proposition}
The proof follows by recalling the  formula $\partial(\log\det A)=\mathrm{tr}(A^{-1}\partial A)$, and plugging in the above formula for the curvature: $-\partial^2_{i {\bar j}} \log \det (g_{k\bar \ell})=-\partial_i (g^{k\bar \ell}g_{ k \bar \ell, j })= \mathrm{Ric}_{i\bar{j}}.$

Here we remark that clearly  in such K\"ahler case the above are  essentially the only non-vanishing terms of the (complexified) Riemannian Ricci tensor. 

The expression for $\mathrm{Ric}$ in Proposition \ref{prop:ricci-kahler} is local, but it actually yields a global real $(1,1)$-form: define the {\em Ricci form} as
$$ \mathrm{Ric}(\omega) := \sqrt{-1} \mathrm{Ric}_{j\bar k} dz_j \wedge d\bar z_k \stackrel{\text{loc}}{=} -\sqrt{-1}\partial\overline\partial \log \omega^n . $$
Here, $\log\omega^n$ is short version for $\log \det (g_{k\bar \ell})$.  Thus, we can say that in the K\"ahler case the Ricci curvature is simply given by the complex Hessian of the logarithm of the volume form! This beautiful expression will have important consequences in the next section, especially in combination with the above $\sqrt{-1}\partial \bar \partial$-lemma.

 Thus the  \emph{K\"ahler-Einstein} (KE) condition can be phrased in term of $(1,1)$-forms by saying that  $$\mathrm{Ric}(\omega)=\lambda\omega,$$ for some  "cosmological constant" $\lambda\in\mathbb R$. Note that, when $n\geq 2$ (the case $n=1$ being described by the Uniformization Theorem), even taking $\lambda\in\mathcal{C}^\infty(X;\mathbb R)$, the cosmological factor is in fact constant, since $0=d \mathrm{Ric} (\omega)=d\lambda\wedge \omega \, \Rightarrow \, \lambda=\mathrm{constant}$. This is a short K\"ahler version of the  Riemannian Schur lemma, see {\itshape e.g.} \cite[Theorem 1.97]{besse}. More generally, if only the scalar curvature is constant we say that $\omega$ is \emph{cscK}.

Finally, let us note that if $\omega'$ is another K\"ahler metric, then one has
$$ \mathrm{Ric}(\omega')-\mathrm{Ric}(\omega)= \sqrt{-1} \partial\overline\partial \log \frac{\omega^n}{{\omega'}^n} $$
where now $\log \frac{\omega^n}{{\omega'}^n}$ is a global function on $M^{2n}$. Being the Ricci form clearly closed, this means that any K\"ahler metric defines a unique class
$$ [\mathrm{Ric}(\omega)] =: 2\pi c_1(X) \in H^2(X;\mathbb R)  \cap H^{1,1}(X), $$
and we call $c_1(X)$ the {\em first Chern class} of the complex manifold $X^n$ (that is, of the holomorphic tangent bundle, which is the same as the first Chern class of the anti-canonical line bundle $K_X^{-1}:=\bigwedge^nTX$, {\itshape i.e.} the dual of the canonical bundle; of course, the definition will agree with any other definition of first Chern class you may have seen before).

\section{Calabi's conjectures}

\subsection{Statement of Calabi's conjectures and basic K\"ahler-Einstein}

We have seen that, for any compact K\"ahler manifold $(X^n,\omega)$, we have $\mathrm{Ric}(\omega) \in 2\pi c_1(X)$; then it is natural to ask what about the converse statement:

\begin{theorem}[first Calabi conjecture, proven by Yau \cite{yau-PNAS, yau-CPAM}]\label{thm:1st-calabi-conj}
Let $(X^n,\omega)$ be a compact K\"ahler manifold, and let $\rho\in 2\pi c_1(X) \in H^2(X;\mathbb R) \cap H^{1,1}(X)$. Then there exists a unique metric $\omega' \in [\omega]$ such that $\mathrm{Ric}(\omega')=\rho$.
\end{theorem}

In particular, when $c_1(X)=0$ (that is, when $X$ is \emph{Calabi-Yau}), the above theorem gives K\"ahler metric $\omega'$ with $\mathrm{Ric}(\omega')=0$. Historically, these were the first non-flat examples of compact Ricci flat manifolds to have been found.  More precisely, Yau's result implies:

\begin{theorem}[{Yau \cite{yau-PNAS, yau-CPAM}}]
Let $X^n$ be a compact K\"ahler Calabi-Yau manifold, {\itshape i.e.} $c_1(X)=0$. Then in any K\"ahler class there is a unique K\"ahler-Einstein metric with cosmological constant $\lambda=0$.
\end{theorem}

As regards the other K\"ahler-Einstein metrics, we first  observe the following important fact:
\begin{proposition}[{after \cite{kodaira-Embedding}}]\label{Alg}
Let $(X^n,\omega)$ be a compact K\"ahler-Einstein manifold with cosmological constant $\lambda \neq 0$. Then $X$ is projective algebraic (actually of a very special type, see next), {\itshape i.e.} it has an embedding $X \hookrightarrow \mathbb{CP}^N$.
\end{proposition}

\begin{remark}
	\begin{itemize}
		\item This is not true for $\lambda=0$ ({\itshape e.g.} the generic flat complex torus in dimension bigger than one).
		\item The KE metric \emph{are not} given by the restriction of Fubini-Study metric.
	\end{itemize}

\end{remark}

The proof requires a little, but important, digression. First, we recall that for any holomorphic Hermitian line bundle $(L,h)$ on a complex manifold $X^n$ (obvious definitions), there exists a unique connection $\nabla^{Ch}_h$ on $L$ being compatible with the Hermitian structure and such that $\nabla^{0,1}=\overline\partial$: it is called the {\em Chern connection}, see {\itshape e.g.} \cite[Section 1.6]{gabor-book}, \cite[Section 1.7]{gauduchon-book}, \cite[Section V.12]{demailly-agbook}. More precisely, let $s$ be a local non-vanishing holomorphic section of $L$, and write $h(s)=|s|^2$, then the connection $1$-form is $\alpha=\frac{\partial h}{h}$. Thus its curvature is $\sqrt{-1}F_{\nabla^{Ch}}=-\sqrt{-1}d\alpha=-\sqrt{-1}\partial\overline\partial \log h$, and 
it represents, by definition, the class $2\pi  c_1(L)$ where $c_1(L)\in H^2(X;\mathbb R) \cap H^{1,1}(X)$ is the {\em first Chern class of $L$}.

Second, we say that the Hermitian line bundle $(L,h)$ over $X^n$ is {\em positive} if $\sqrt{-1}F_{\nabla^{Ch}}$ is a K\"ahler metric on $X^n$.
Note that this is equivalent to ask that there exists a K\"ahler metric $\omega \in 2\pi c_1(L)$ ({\itshape a priori} a weaker requirement). Indeed, take any Hermitian structure $\tilde h$ on $L$, then by the $\sqrt{-1}\partial\overline\partial$-Lemma there exists $\varphi$ such that $\omega-\sqrt{-1}F_{\nabla^{Ch}_{\tilde h}}=\sqrt{-1}\partial\overline\partial \varphi$, and then $\omega=\sqrt{-1}F_{\nabla^{Ch}_{h}}$ for $h:=\exp(-\varphi)\tilde h$. By the way, this argument precisely says that we can always solve the \emph{Hermitian-Einstein equation} for line bundles!

We also introduce another notion: a holomorphic line bundle $L$ over $X$ is called {\em ample} if there exists a holomorphic embedding $f \colon X \hookrightarrow \mathbb{CP}^N$ such that $L^{\otimes k}\simeq f^* \mathcal{O}_{\mathbb{CP}^N}(1)$ for some $k$. Here, $\mathcal{O}_{\mathbb{CP}^N}(1)$ is the dual of the tautological line bundle $\mathcal{O}_{\mathbb{CP}^N}(-1)$ of $\mathbb{CP}^N$ (the one that takes the line $\ell \in \mathbb C^{n+1}$ over the point $[\ell]\in \mathbb{CP}^n$). 

Next we  recall the following  very deep theorem by Kodaira that relates the two notions, see {\itshape e.g.} \cite[Theorem 7.11]{voisin-1}, \cite[Corollary VII.13.3]{demailly-agbook}:

\begin{theorem}[{Kodaira \cite{kodaira-Embedding}}]
A holomorphic line bundle over a compact complex manifold admits a positive Hermitian structure if and only if it is ample.
\end{theorem}

Of course one direction is trivial: just take the restriction of Fubini-Study metric. The other requires considerably more work. We will describe some of the idea of the proof in subsequent sections using the "pick sections technique". 

\begin{proof}[Proof of Proposition \ref{Alg}]
If we have a KE metric with positive cosmological section $\mathrm{Ric}(\omega)=\lambda \omega >0$, hence, by the discussion at the end of the last lecture the anticanonical bundle is positive. Thus, by Kodaira's theorem, $X$ can be embedded in some projective space (hence it is algebraic, {\itshape i.e.} it is cut  by homogeneous polynomial by Chow's theorem, {\itshape e.g.}  \cite{griffiths-harris}). The case of negative cosmological constant goes along the same lines.
\end{proof}

We say that a compact complex manifold is \emph{Fano} if $c_1(X)>0$, {\itshape i.e.} if the first Chern class can be represented by a K\"ahler form. This is equivalent to say that the anticanonical line bundle is ample, and it is thus a "trivial" necessary condition for a complex manifold to admit KE metric of positive scalar curvature. Note that as an immediate consequence of the first Calabi conjecture, Fano manifolds are precisely the compact complex manifolds admitting K\"ahler metrics of positive (non necessary costant!) Ricci curvature.  If $c_1(X)<0$ (equivalenty, the canonical bundle is ample) the manifold is, in particular, of \emph{general type} ({\itshape i.e.} of top Kodaira dimension, see {\itshape e.g.} \cite{voisin-1}).  Again this condition is necessary for KE metrics with negative cosmological constant.

Fano manifolds are definitely fewer that manifolds with negative first Chern class. In particular, there are only finitely many diffeomorphism classes of  Fano manifolds for fixed dimension \cite{kmm}, see {\itshape e.g.} \cite{kollar}. The next proposition  provides some simple examples Fanos, Calabi-Yaus, and algebraic manifolds with ample canonical bundle.

\begin{proposition}\label{prop:algebraic-hypersurf}
Let $X_d=\{F_d=0\} \subseteq \mathbb{CP}^n$ be a smooth hypersurface of degree $d$. Then:
\begin{itemize}
\item if $d<n+1$, then $X_d$ is Fano;
\item if $d=n+1$, then $X_d$ is Calabi-Yau;
\item if $d>n+1$, then $X_d$ is with ample canonical bundle.
\end{itemize}
\end{proposition}

\begin{proof}

We sketch here some ideas for the proof. See {\itshape e.g.} \cite{griffiths-harris} for more details on the definitions.

For $Y \subseteq X$ a complex submanifold of codimension $1$, we have the short exact sequence of holomorphic vector bundles
$$ 0 \to TY \to TX \lfloor_{Y} \to N_{Y} \to 0 , $$
and its dual, whence the adjunction formula
$$ K_Y^{-1} = K_X^{-1}\lfloor_Y \otimes N_{Y} , $$
where
$ N_{Y}^{\vee} = \mathcal{O}(Y)\lfloor_Y. $ Here, given locally $Y \cap U_\alpha = \{f_\alpha=0\}$, then $\mathcal{O}(Y)$ is the line bundle with transition functions $g_{\alpha\beta}:=\frac{f_\alpha}{f_\beta}$. Since $df_\alpha=g_{\alpha\beta}df_\beta$, we have a non-vanishing global section of $N_{Y}^\vee\otimes\mathcal{O}(Y)$, that is, $N_{Y}^\vee\otimes\mathcal{O}(Y)\simeq\mathcal{O}$ the trivial line bundle.

In our case, we recall that the space of isomorphisms classes of holomorphic line bundles on $\mathbb{CP}^n$ is $\mathrm{Pic}(\mathbb{CP}^n)=\langle \mathcal{O}(1) \rangle \simeq \mathbb{Z}$. We have $K_{\mathbb{CP}^n}\simeq \mathcal{O}(-1)^{\otimes (n+1)}$ ({\itshape e.g.} since the standard holomorphic volume form on $\C^n$ has a pole of order $n+1$ at infinity). Since an hypersurface of degree $d$ is given by a section of $\mathcal{O}(d)$, 
$$ K_{X_d}^{-1} \simeq \mathcal{O}(n+1-d)\lfloor_{X_d} . $$
The statement follows.
\end{proof}

It is worth noting that by choosing different homogeneous polynomials in the above proposition we have different (non-biholomorphic) complex manifolds. For fixed degree $d$ they are all diffeomorphic. The rough "count" of the number of complex moduli parameters considering the $\mathrm{PGL}(n)$ reparametrization is  equal to $\binom{n+d}{d}-(n+1)^2$. Clearly, from the point of view of studying existence of KE metrics on them, such complex manifolds are all to be considered different.

In our search for KE metrics, the next theorem is crucial:

\begin{theorem}[{second Calabi conjecture, proven by Aubin \cite{aubin-CRAS, aubin-BSM}, Yau \cite{yau-PNAS, yau-CPAM}}]\label{thm:2nd-calabi-conj}
Let $X^n$ be a compact complex manifold with ample canonical bundle. Then there exists a unique K\"ahler metric $\omega\in 2\pi c_1(X)$ such that $\mathrm{Ric}(\omega)=-\omega$.
\end{theorem}

Thus, thanks to Theorems \ref{thm:1st-calabi-conj} and \ref{thm:2nd-calabi-conj},  Proposition \ref{prop:algebraic-hypersurf} gives many  examples of (K\"ahler)-Einstein metrics: \emph{any} hypersurface of degree $d\geq n+1$ in $\mathbb{CP}^n$ is  K\"ahler-Einstein.

By the way, as regards hypersurfaces in the Fano case, the problem is still open! It is conjectured they all do. 

\begin{remark} 
It is known that KE metrics exists on  hypersurfaces of Fermat type $\sum_i z_i^d=0$ ({\itshape e.g.} \cite{tian-book}).  The case of hypersurfaces for  $n=3$, $d=3$ was proved  in \cite{tian-Invent90}. Donaldson  (also Odaka independently)  showed that the set of K\"ahler-Einstein manifolds in the space of Fano manifolds is Zariski open under the assumption that the group of holomorphic automorphisms of each manifold is discrete ({\itshape e.g.} \cite{donaldson-Zariski}), hypothesis which holds for hypersurfaces.   Fujita  recently proved  existence of KE metrics on all smooth hypersurfaces of degree $d=n$ in $\mathbb{CP}^n$ \cite{fujita-arxiv}. For $n=4$ the conjecture follows by (very recent!) works \cite{spotti-sun-arxiv1705, liu-xu-arxiv1706}. Thus for $3\leq d \leq n-1$  and $n\geq5$ the existence problem of KE metrics on all smooth hypersurfaces is open.
\end{remark}

\subsection{Reduction of Calabi's conjectures to certain PDE's}

In this section, we reduce first and second Calabi's conjectures, Theorems \ref{thm:1st-calabi-conj} and \ref{thm:2nd-calabi-conj}, to a PDE, whose solutions are going to be studied in the following sections.

Consider the first Calabi's conjecture, Theorem \ref{thm:1st-calabi-conj}, that is, the problem of representing a chosen form $\rho$ in $2\pi c_1(X)$ as the Ricci form of a metric in $[\omega]$, where $\omega$ is the fixed background metric. Thanks to the $\sqrt{-1}\partial\overline\partial$-Lemma, we are reduced to look for $\varphi\in\mathcal{C}^\infty(X;\mathbb R)$ such that $\omega_\varphi:=\omega+\sqrt{-1}\partial\overline\partial\varphi>0$ has $\mathrm{Ric}({\omega_\varphi})=\rho$. Again by the $\sqrt{-1}\partial\overline\partial$-Lemma, we have $f\in\mathcal{C}^\infty(X;\mathbb R)$ such that $\mathrm{Ric}(\omega)-\rho=\sqrt{-1}\partial\overline\partial f$; we can normalize the Ricci potential $f$ in such a way that $\int_X \exp f \omega^n = \int_X \omega^n$. By the expression for the Ricci form of a K\"ahler metric, Proposition \ref{prop:ricci-kahler}, we are reduced to solve $\sqrt{-1}\partial\overline\partial\log \frac{\omega_\varphi^n}{\omega^n} = \sqrt{-1}\partial\overline\partial f$. When $X$ is compact, this reduces to
\begin{equation}\label{eq:1st-calabi-conj}
\omega_\varphi^n =\exp f \omega^n .
\end{equation}
In other words, we are prescribing the volume form.
In coordinates, this can be written as
$$ \det ( g_{i\bar j}+\partial^2_{z_i\bar z_j}\varphi ) = \exp f \det g_{i\bar j} , $$
which is a {\em complex Monge-Ampère equation} a very much studied second oreder non-linear PDE.

\begin{remark}
 As regards deep study of \emph{complex} Monge-Ampère equations, see {\itshape e.g.} \cite{aubin-book, guedj-book, guedj-zeriahi}). As regards \emph{real} Monge-Ampère equations, see {\itshape e.g.} \cite{gilbarg-trudinger, caffarelli-cabre, figalli-book} and the references therein.
\end{remark}

Similarly, the problem of finding K\"ahler-Einstein metrics with cosmological constant $\lambda \in \mathbb R$ reduces to the equation
\begin{equation}\label{eq:eq-ke}
\omega_\varphi^n = \exp(f-\lambda \varphi) \omega^n .
\end{equation}
In particular, the second Calabi's conjecture, Theorem \ref{thm:2nd-calabi-conj}, corresponds to $\lambda=-1$, up to homothety:
\begin{equation}\label{eq:2nd-calabi-conj}
\omega_\varphi^n = \exp(f+\varphi) \omega^n .
\end{equation}

\subsection{Uniqueness of solutions for Calabi's conjectures}
In this section, we prove uniqueness of solutions for equations \eqref{eq:1st-calabi-conj} and \eqref{eq:2nd-calabi-conj}.

Consider first the equation \eqref{eq:2nd-calabi-conj} for the second Calabi's conjecture. In this case the proof of uniqueness follows easily by the maximum principle. Assume that $\omega$ and $\omega_\varphi$ are two different K\"ahler-Einstein metrics. We can assume $f=0$ in \eqref{eq:2nd-calabi-conj}, so we have $\omega_\varphi^n=\exp\varphi\omega^n$. Let $p_{\text{max}}$ be a maximum point for $\varphi$. Then $\partial^2_{z_i\bar z_j}\varphi(p_{\text{max}})\leq 0$. By using $\omega_\varphi^n=\exp\varphi\omega^n$, we get that $\exp\varphi(p_{\text{max}})\leq 1$, then $\varphi(x)\leq \varphi(p_{\text{max}})\leq 0$ for any $x$. The same argument for a minimum point $q_{\text{min}}$ gives $\varphi(x)\geq 0$ for any $x$. Thus $\varphi=0$.

Uniqueness, up to additive constants, of solutions for first Calabi's conjecture \eqref{eq:1st-calabi-conj} has been proven by Calabi \cite{calabi-57}. The proof now is more subtle, and it needs some non-pointwise estimates. Assume that $\omega_\varphi=\omega+\sqrt{-1}\partial\overline\partial\varphi$ and $\omega_\psi=\omega+\sqrt{-1}\partial\overline\partial\psi$ are two different metrics in the same K\"ahler class $[\omega]$ with the same volume form $\omega_\varphi^n=\omega^n=\omega_\psi^n$. Consider the {\em Aubin functional} \cite{aubin-JFA} as a generalization of the classical Dirichlet energy (see also {\itshape e.g.} \cite[Section 4.2]{gauduchon-book}) and \cite[page 46]{tian-book}):
$$ I_{\omega}(\varphi) := \frac{1}{n! \mathrm{Vol}_\omega} \int_X  \varphi (\omega^n - \omega_\varphi^n) . $$
We have:
\begin{eqnarray*}
0 &=& \int_X (\varphi-\psi) (\omega^n_\psi-\omega_\varphi^n) = -\sqrt{-1} \int_X (\varphi-\psi) \partial\overline\partial(\varphi-\psi) \wedge \sum_{k=0}^{n-1} \omega_\varphi^k \wedge \omega_\psi^{n-k-1} \\
&=& \sqrt{-1} \int_X \partial(\varphi-\psi) \wedge \overline\partial(\varphi-\psi) \wedge \omega_\varphi^{n-1} \\
&& + \sqrt{-1} \sum_{k=0}^{n-2} \int_X \partial(\varphi-\psi) \wedge \overline\partial(\varphi-\psi) \wedge \omega_\varphi^k\wedge\omega_\psi^{n-k} \\
&\geq& \sqrt{-1} \int_X \partial(\varphi-\psi) \wedge \overline\partial(\varphi-\psi) \wedge \omega_\varphi^{n-1} \\
&=& \frac{1}{n} \int_X |\partial(\varphi-\psi)|^2_{\omega_\varphi} \omega_\varphi^n \geq 0 ,
\end{eqnarray*}
yielding $\partial(\varphi-\psi)=0$. Since $\varphi-\psi$ is real, this gives that $\varphi-\psi$ is constant and so $\omega_\varphi=\omega_\psi$.

As a consequence, we get:
\begin{proposition}
Let $X^n$ be a compact complex manifold with ample canonical bundle. Then any $F\in\mathrm{Aut}(X)$ is an isometry for the K\"ahler-Einstein metric.
\end{proposition}
That is, K\"ahler-Einstein metrics are {\em canonical}: namely, the symmetries for the metric are the holomorphic symmetries.
Recall that, in such a case, the group of isometries is finite, by Bochner formula, see {\itshape e.g.} \cite[Theorem 1.84]{besse} and discussions in the next section.

\subsection{Existence of solutions for second Calabi's conjecture}

We prove existence of a smooth solution for equation \eqref{eq:2nd-calabi-conj} by applying the {\em continuity method}. We consider the family of equations
\begin{equation}\label{eq:2nd-continuity}\tag{$\star_t$}
\omega_\varphi^n = \exp(t f +\varphi) \omega^n
\end{equation}
varying $t\in[0,1]$. Set
$$ E := \left\{ t \in [0,1] \;:\; \text{there exists } \varphi_s \in\mathcal{C}^\infty(X;\mathbb R) \text{ solution of \eqref{eq:2nd-continuity}} \mbox{ for all }   0\leq s\leq t \right\} . $$
We prove that:
\begin{itemize}
\item $E\neq\varnothing$ (this is straightforward because $\varphi=0$ is a solution for \eqref{eq:2nd-continuity} when $t=0$);
\item $E$ is open (this will follow by an implicit function argument);
\item $E$ is closed (this will follow by {\itshape a priori} estimates on the solutions).
\end{itemize}
In this case, since $[0,1]$ is connected, we get $E=[0,1]$, and in particular $1\in E$, assuring that we have a solution for \eqref{eq:2nd-continuity} for $t=1$, which is exactly \eqref{eq:2nd-calabi-conj}.

We prove {\em openness}.\label{page:open-2nd-calabi}
We can define the H\"older spaces $\mathcal{C}^{k,\alpha} (X)$ (on a compact manifold just take the usual definition in charts; one can also use a fixed background metric and compute the H\"older differences via parallel transport. All these possible constructions give the same spaces with equivalent norms). 
Consider
$$ \mathcal{E} \colon \mathcal{C}^{k,\alpha}(X) \times [0,1] \to \mathcal{C}^{k-2, \alpha}(X), \qquad (\varphi,t) \mapsto \log\frac{\omega_\varphi^n}{\omega^n} -\varphi-tf , $$
We can then rewrite: $\varphi$ is a solution of \eqref{eq:2nd-continuity} if and only if $\mathcal{E}(\varphi,t)=0$.
We compute
\begin{eqnarray*}
D_{\varphi} \mathcal{E} (\varphi,t) (\psi) &:=& \frac{d}{ds}\lfloor_{s=0} \mathcal{E}(\varphi+s\psi,t) \\
&=& \mathrm{tr}_{\omega_\varphi}\sqrt{-1}\partial\overline\partial\psi-\psi=-\Delta_{\overline\partial}\psi-\psi .
\end{eqnarray*}
Since $-\Delta_{\overline\partial}$ has non-positive eigenvalues, then the kernel of $D_{\varphi} \mathcal{E} (\varphi,t)$ is zero. Being self-adjoint,, the set of $\mathcal{C}^{k,\alpha}_{\omega}$-solutions is open by implicit function theorem.
To prove smoothness of the solution, we can start from $k=2$ using the Evans-Krylov theory \cite{evans, krylov}, or from $k=3$ using the Schauder theory (see {\itshape e.g.} \cite[Chapter 6]{gilbarg-trudinger}). The idea of the Schauder theory is the following. Consider $\varphi_t$ a $\mathcal{C}^{3,\alpha}$-regular solution of $\log\det (g_{j\bar k}+\partial^2_{z_j \bar z_k}\varphi_t)_{j,k}-\log\det (g_{j\bar k})_{j,k}-\varphi_t-tf=0$. By differentiating the equation in $z_\rho$, we get
$$ g_{\varphi_t}^{j\bar k} \left( \partial_{z_\rho} g_{j\bar k} + \partial^3_{z_\rho z_j \bar z_k} \varphi_t \right) - \partial_{z_\rho} \log\det (g_{j\bar k})_{j,k} - \partial_{z_\rho} \varphi_t - t \partial_{z_\rho} f = 0 . $$
So locally $\partial_{z_\rho}\varphi_t$ is a $\mathcal{C}^2$-regular solution of
$$ L\partial_{z_\rho}\varphi_t = h $$
where $L:=\Delta_{\omega_{\varphi_t}}-1$  (here and later $\Delta_{\omega}:=-\frac{1}{2} \Delta_d$, as analyst would prefer) is a second-order linear elliptic operator with $\mathcal{C}^{1,\alpha}$-regular coefficients, and $h:=-g_{\varphi_t}^{j\bar k}\partial_{z_\rho} g_{j\bar k} + \partial_{z_\rho} \log\det (g_{j\bar k})_{j,k} + t \partial_{z_\rho} f$ is $\mathcal{C}^{1,\alpha}$-regular. By Schauder estimates, see {\itshape e.g.} \cite[Theorem 6.19]{gilbarg-trudinger}, $d\varphi$ is $\mathcal{C}^{3,\alpha}$-regular, then $\varphi$ is $\mathcal{C}^{4,\alpha}$-regular. By bootstrap, we get that $\varphi$ is $\mathcal{C}^{k,\alpha}$-regular for any $k$ and $\alpha$, so $\varphi$ is smooth.

We prove now {\em closedness}. (See also {\itshape e.g.} \cite[Chapter 5]{tian-book}, \cite[Chapter 3]{gabor-book}.)
More precisely, let $(\varphi_{t_j})_j$ be a sequence where $\varphi_{t_j}$ is a solution of \eqref{eq:2nd-continuity} for $t=t_j$. We want to prove that $\varphi_{t_j} \to \varphi_{\bar t}$ as $t_j\to \bar t$, and that $\varphi_{\bar t}$ is a solution of \eqref{eq:2nd-continuity} for $t=\bar t$.
This will follow by the following {\itshape a priori} estimates on the solutions of \eqref{eq:2nd-continuity}:
\begin{enumerate}[(i)]
\item uniform $\mathcal{C}^0$-estimate;
\item uniform Laplacian estimate;
\item uniform higher-order estimate.
\end{enumerate}
Indeed, if we have $\|\varphi_{t_j}\|_{\mathcal{C}^{3,\alpha}\omega(X)}<C$, then by the Ascoli-Arzelà theorem, see {\itshape e.g.} \cite[3.15]{aubin-book}, we get that $\varphi_{t_j} \to \varphi_{\bar t} \in \mathcal{C}^{3,\alpha'}(X)$ for some $0<\alpha'<\alpha$, and $\varphi_{\bar t}$ is still a solution of \eqref{eq:2nd-continuity} for $t=\bar t$.
Moreover, by the estimate in (ii), we will also get $g_\varphi>0$.

Notice that (ii) and (iii), as well as the argument in the paragraph above, hold also for equation \eqref{eq:1st-calabi-conj} and for equation \eqref{eq:eq-ke}. So the difficulty to solve the first Calabi's conjecture  essentially arises in adapting the $\mathcal{C}^0$-estimate.

The  {\em $\mathcal{C}^0$-estimate} is very easy. This follows directly by applying the maximum principle to \eqref{eq:2nd-continuity} as we did in the proof of uniqueness. We get $\|\varphi_t\|_{L^\infty}\leq C \|f\|_{L^\infty}$ for some constant $C=C(X^n,\omega)$.

We prove the {\em Laplacian estimate}. Roughly, we look for an estimate of the type
$$ 0 < n + \Delta_g \varphi \leq C(\|\varphi\|_{L^\infty}),  $$
where $C$ is a function (exponential) so that together with the $L^\infty$ bound give the control on the Laplacian.
The crucial estimate is the following: denote $g_\varphi=g+\partial^2\varphi$, so $\mathrm{tr}_{g}g_{\varphi}=n+\Delta_g\varphi$; we prove that there exists a constant $B=B(X^n,g)$ such that
\begin{equation}\label{eq:laplacian-est}
\Delta_{g_\varphi} \log\mathrm{tr}_g g_\varphi \geq -B \mathrm{tr}_{g_\varphi}g - \frac{\mathrm{tr}_g\mathrm{Ric}(g_\varphi)}{\mathrm{tr}_gg_\varphi} .
\end{equation}
We prove \eqref{eq:laplacian-est}.
At a point $p_0$, we can assume that $g(p_0)=\mathrm{id}$ and $g_\varphi(p_0)$ is diagonal. We use also the K\"ahler condition so at the point $p_0$ we can further assume that we have complex normal coordinates, {\itshape i.e.} the first derivatives of the metric $g$ at $p_0$ vanish. This can be done  for any K\"ahler metric \cite{griffiths-harris}. Then
$$ \mathrm{tr}_g g_\varphi(p_0) = \sum_i (g_{\varphi})_{i\bar i}(p_0), \qquad \mathrm{tr}_{g_\varphi}g(p_0) = \sum_{j} g^{j \bar j}_\varphi(p_0) = \sum_{j} (g_{\varphi})_{j\bar j}(p_0)^{-1} . $$
At $p_0$, we compute:
\begin{eqnarray*}
\Delta_{g_\varphi} \mathrm{tr}_g g_{\varphi} (p_0) &=& g_{\varphi}^{p\bar q} \partial^2_{z_p\bar z_q}(g^{j\bar k}(g_{\varphi})_{j\bar k}) \\
&=& g_{\varphi}^{p\bar q} \partial^2_{z_p\bar z_q}g^{j\bar k}(g_{\varphi})_{j\bar k}
+ g_{\varphi}^{p\bar q} g^{j\bar k}\partial^2_{z_p\bar z_q}(g_{\varphi})_{j\bar k} \\
&=& g_{\varphi}^{p\bar q} \partial^2_{z_p\bar z_q}g^{j\bar k} (g_{\varphi})_{j\bar k} - g_\varphi^{p\bar q} g^{j\bar k} (\mathrm{R}_{\varphi})_{j\bar k p \bar q} + g_\varphi^{a\bar b} g_\varphi^{p\bar q} g^{j\bar k} (\partial_{z_j} (g_{\varphi})_{p \bar b}) (\partial_{\bar z_k}(g_{\varphi})_{a\bar q}) \\
&=& g_{\varphi}^{p\bar q} \partial^2_{z_p\bar z_q}g^{j\bar k} (g_{\varphi})_{j\bar k} - g^{j\bar k} (\mathrm{Ric}(\omega_{\varphi}))_{j\bar k} + g_\varphi^{a\bar b} g_\varphi^{p\bar q} g^{j\bar k} (\partial_{z_j} (g_{\varphi})_{p \bar b}) (\partial_{\bar z_k}(g_{\varphi})_{a\bar q}) \\
&=& g_{\varphi}^{p\bar p} \partial^2_{z_p\bar z_p}g^{j\bar j} (g_{\varphi})_{j\bar j} - g^{j\bar k} (\mathrm{Ric}(\omega_{\varphi}))_{j\bar k} + g_\varphi^{a\bar b} g_\varphi^{p\bar q} g^{j\bar k} (\partial_{z_j} (g_{\varphi})_{p \bar b}) (\partial_{\bar z_k}(g_{\varphi})_{a\bar q}) \\
&\geq& -B g_\varphi^{p\bar p}(g_\varphi)_{j\bar j} - g^{j\bar k} (\mathrm{Ric}(\omega_{\varphi}))_{j\bar k} + g_\varphi^{a\bar b} g_\varphi^{p\bar q} g^{j\bar k} (\partial_{z_j} (g_{\varphi})_{p \bar b}) (\partial_{\bar z_k}(g_{\varphi})_{a\bar q}) \\
&=& -B \mathrm{tr}_g g_\varphi \mathrm{tr}_{g_\varphi} g - g^{j\bar k} (\mathrm{Ric}(\omega_{\varphi}))_{j\bar k} + g_\varphi^{a\bar b} g_\varphi^{p\bar q} g^{j\bar k} (\partial_{z_j} (g_{\varphi})_{p \bar b}) (\partial_{\bar z_k}(g_{\varphi})_{a\bar q}) \\
&=& -B \mathrm{tr}_g g_\varphi \mathrm{tr}_{g_\varphi} g - g^{j\bar k} (\mathrm{Ric}(\omega_{\varphi}))_{j\bar k} + \sum_j g_\varphi^{a\bar a} g_\varphi^{p\bar p} |\partial_{z_j} (g_{\varphi})_{p \bar a}|_g
\end{eqnarray*}
where $-B=-B(X^n,g)$ is a lower bound for the holomorphic bisectional curvature of $g$ at $p_0$ (which is finite since $X$ is compact).
We now compute, at $p_0$:
\begin{eqnarray*}
\Delta_{g_\varphi} \log \mathrm{tr}_g g_{\varphi} (p_0) &=& \frac{\Delta_{g_\varphi} \mathrm{tr}_g g_\varphi}{\mathrm{tr}_g g_\varphi} - \frac{g_\varphi^{p\bar q}\partial_p(\mathrm{tr}_g g_\varphi)\partial_{\bar q}(\mathrm{tr}_g g_\varphi)}{(\mathrm{tr}_g g_\varphi)^2} \\
&\geq& -B \mathrm{tr}_{g_\varphi} g - \frac{g^{j \bar k}(\mathrm{Ric}(\omega_{\varphi}))_{j\bar k}}{\mathrm{tr}_g g_{\varphi}}+\frac{1}{\mathrm{tr}_g g_\varphi} \sum_j g_\varphi^{a\bar a} g_\varphi^{p\bar p} |\partial_{z_j} (g_{\varphi})_{p \bar a}|_g \\
&& - \frac{1}{(\mathrm{tr}_g g_\varphi)^2} g_\varphi^{p\bar p}\partial_p (g_{\varphi})_{a\bar a}\partial_{\bar p} (g_{\varphi})_{b\bar b} .
\end{eqnarray*}
We estimate the last term by the Cauchy-Schwarz inequality twice:
\begin{eqnarray*}
g_\varphi^{p\bar p}\partial_p (g_{\varphi})_{a\bar a}\partial_{\bar p} (g_{\varphi})_{b\bar b} &=&
(g_\varphi^{p\bar p})^{\frac{1}{2}} (\partial_p (g_{\varphi})_{a\bar a}) (g_\varphi^{p\bar p})^{\frac{1}{2}} (\partial_p (g_{\varphi})_{b\bar b}) \\
&\leq& \sum_{a,b} (\sum_p g_\varphi^{p\bar p} |\partial_p (g_{\varphi})_{a\bar a}|^2)^{\frac12} (\sum_p g_\varphi^{p\bar p}|\partial_q (g_{\varphi})_{b\bar b}|^2)^{\frac12} \\
&=& \sum_a \sum_p g_\varphi^{p\bar p} |\partial_p (g_{\varphi})_{a\bar a}|^2 \\
&=& (\sum_a ((g_\varphi)_{a\bar a})^{\frac12} (\sum_p g_\varphi^{p\bar p} g_\varphi^{a\bar a} |\partial_p (g_{\varphi})_{a\bar a}|^2)^{\frac{1}{2}})^2 \\
&\leq& (\sum_a (g_{\varphi})_{a\bar a})(\sum_b \sum_p g_\varphi^{p\bar p} g_\varphi^{b\bar b} |\partial_p (g_{\varphi})_{b\bar b}|^2) .
\end{eqnarray*}
Thus we get
\begin{eqnarray*}
\Delta_{g_\varphi} \log \mathrm{tr}_g g_{\varphi} (p_0)
&\geq& -B \mathrm{tr}_{g_\varphi} g - \frac{g^{j \bar k}(\mathrm{Ric}(\omega_{\varphi}))_{j\bar k}}{\mathrm{tr}_g g_{\varphi}}+\frac{1}{\mathrm{tr}_g g_\varphi} \sum_j g_\varphi^{a\bar a} g_\varphi^{p\bar p} |\partial_{z_j} (g_{\varphi})_{p \bar a}|_g \\
&& - \frac{1}{(\mathrm{tr}_g g_\varphi)^2} g_\varphi^{p\bar p}\partial_p (g_{\varphi})_{a\bar a}\partial_{\bar p} (g_{\varphi})_{b\bar b} \\
&\geq& -B \mathrm{tr}_{g_\varphi} g - \frac{g^{j \bar k}(\mathrm{Ric}(\omega_{\varphi}))_{j\bar k}}{\mathrm{tr}_g g_{\varphi}}+\frac{1}{\mathrm{tr}_g g_\varphi} \sum_j g_\varphi^{a\bar a} g_\varphi^{p\bar p} |\partial_{z_j} (g_{\varphi})_{p \bar a}|_g \\
&& - \frac{1}{\mathrm{tr}_g g_\varphi} g_\varphi^{p\bar p} g_\varphi^{a\bar a} |\partial_p (g_{\varphi})_{a\bar a}|^2 \\
&\geq& -B \mathrm{tr}_{g_\varphi} g - \frac{g^{j \bar k}(\mathrm{Ric}(\omega_{\varphi}))_{j\bar k}}{\mathrm{tr}_g g_{\varphi}}+\frac{1}{\mathrm{tr}_g g_\varphi} \sum_j g_\varphi^{a\bar a} g_\varphi^{p\bar p} |\partial_{z_j} (g_{\varphi})_{p \bar a}|_g \\
&& - \frac{1}{\mathrm{tr}_g g_\varphi} \sum_j g_\varphi^{p\bar p} g_\varphi^{a\bar a} |\partial_p (g_{\varphi})_{j\bar a}|^2 \\
&\geq& -B \mathrm{tr}_{g_\varphi} g - \frac{g^{j \bar k}(\mathrm{Ric}(\omega_{\varphi}))_{j\bar k}}{\mathrm{tr}_g g_{\varphi}} ,
\end{eqnarray*}
where we used the K\"ahler condition. This proves \eqref{eq:laplacian-est}.
We prove now the uniform equivalence between the metrics $g$ and $g_\varphi$, that will yield to the Laplacian and to partial $\mathcal{C}^2$-estimates. More precisely, we prove that there exists a constant $C=C(X^n, \omega, \|f\|_{L^\infty}, \inf \Delta_\omega f, \|\varphi\|_{L^\infty})>0$ such that
$$ C^{-1} g_{j\bar k} \leq (g_\varphi)_{j\bar k} \leq C g_{j\bar k} , $$
by using the maximum principle.
We plug-in \eqref{eq:2nd-continuity} in the estimate \eqref{eq:laplacian-est} to get:
$$ \Delta_{g_\varphi} \log \mathrm{tr}_g g_\varphi \geq -B \mathrm{tr}_{g_\varphi} g + \frac{\Delta f + \mathrm{tr}_g g_\varphi - n - \mathrm{Sc}_\omega}{\mathrm{tr}_g g_\varphi} . $$
By the Cauchy-Schwarz inequality, we have
$$ \mathrm{tr}_{g}g_\varphi \cdot \mathrm{tr}_{g_\varphi}g \geq n^2 . $$
Then, under the normalization for $f$, we have a constant $C$ such that
$$ \Delta_\varphi \log\mathrm{tr}_g g_\varphi \geq - C \mathrm{tr}_{g_\varphi} g . $$
Since $\Delta_{g_\varphi} \varphi = n - \mathrm{tr}_{g_\varphi}g$, then for $A\gg 1$:
$$ \Delta_{g_\varphi}(\log\mathrm{tr}_g g_\varphi - A\varphi)\geq \mathrm{tr}_{g_\varphi}g-nA . $$
By the maximum principle, at a maximum point $p_{\text{max}}$, we get
$$ \mathrm{tr}_{g_\varphi}g (p_{\text{max}}) \leq A . $$
On the other side, by the equation, we get, in suitable coordinates at $p_{\text{max}}$, that $\prod g_{\varphi, i\bar i}(p_{\text{max}}) = \exp (f(p_{\text{max}})+\varphi(p_{\text{max}})) \leq C$ for some constant $C$, then
$$ \mathrm{tr}_g g_\varphi(p_{\text{max}}) \leq C. $$
Then, at any point,
$$ \log\mathrm{tr}_g g_\varphi - A \varphi \leq C - A \varphi(p_{\text{max}}) , $$
whence
$$ \sup\log\mathrm{tr}_g g_\varphi \leq C . $$

We  now shortly discuss  {\em higher-order estimates}. For the $\mathcal{C}^3$ estimates, after Phong-Sesum-Sturm \cite{pss}, it is convenient to look at estimates of the Christoffel symbols: this will suffice because $g_\varphi$ and $g$ are equivalent. Set
$$ S(\varphi)_{j k}^{i}:=\Gamma(g)_{jk}^i-\Gamma(g_\varphi)_{jk}^i . $$
A long computation as before yields,
$$ \|S\|_{L^\infty}\leq C . $$
By the above $\mathcal{C}^0$-estimate, we have that $g_\varphi$ is equivalent to $g$, then we get
$ |\partial_{z_j z_k \bar z_\ell}^3 \varphi |\leq C . $
This gives $\partial_{z_j \bar z_k}\varphi \in \mathcal{C}^{0,\alpha}_{\omega}(X)$, for any $j$, $k$. By Schauder as before, we get $\partial_{z_j} \varphi \in \mathcal{C}^{2,\alpha}_{\omega}(X)$, whence finally $\partial^3\varphi \in\mathcal{C}^{0,\alpha}_{\omega}(X)$.

\subsection{Existence of solutions for first Calabi's conjecture}

We consider now the first Calabi conjecture. Essentially, the only difference is the $\mathcal{C}^0$-estimate: the Laplacian and higher-order estimates, as well as the general argument, can be repeated with very small variations, see \cite{tian-book}.

Under the suitable normalization $\sup_X \varphi =-1$, we prove the Moser iteration:
$$ \|\varphi\|_{L^\infty} \leq C (\|\varphi\|_{L^2_\omega}+1) . $$
By Poincaré inequality, we prove:
$$ \|\varphi\|_{L^2_\omega} \leq C (\|\varphi\|_{L_\omega^1}+1) . $$
Finally, by Green function argument, we prove:
$$ \|\varphi\|_{L_\omega^1}\leq C . $$

The essential steps in the Moser iteration are the following. Set $\varphi_{-}:=-\varphi\geq 1$. Consider the equation
\begin{eqnarray*}
(\exp f-1)\omega^n &=& \omega_\varphi^n-\omega^n \\
&=& -\sqrt{-1}\partial\overline\partial\varphi_{-} \wedge (\omega_\varphi^{n-1}+\omega_{\varphi}^{n-2}\wedge\omega+\cdots+\omega_\varphi\wedge\omega^{n-2}+\omega^{n-1}) ;
\end{eqnarray*}
now muplitply by $\varphi_{-}^p$ and integrate as before:
\begin{eqnarray*}
\dashint_X \varphi_{-}^p (\exp f-1)\omega^n &\geq& p \cdot \dashint_X \varphi_{-}^{p-1} \sqrt{-1} \partial\varphi_{-}\wedge\overline\partial\varphi_{-} \wedge\omega^{n-1} \\
&=& \frac{4p}{n(p+1)^2} \dashint_X (|\nabla (\varphi_{-}^{\frac{p+1}{2}})|^2\omega^n ,
\end{eqnarray*}
(where, hereafter, we use the notation $\dashint_X f \omega^n = \frac{1}{\mathrm{Vol}_{\omega}}\int_X f \omega^n$,)
and on the other side clearly
$$ C f \dashint_X (\varphi_{-})^{p+1}\omega^n \geq \dashint_X \varphi_{-}^p (\exp f-1)\omega^n . $$
We use the Sobolev embedding, see {\itshape e.g.} \cite[Theorem 2.21]{aubin-book}, \cite[Theorem 7.10]{gilbarg-trudinger}:
$$ W^{1,q}_\omega \hookrightarrow L^{\frac{q\cdot 2n}{2n-q}}_\omega . $$
Then, for $q=2$:	
$$ \dashint_X (\varphi_{-})^{p+1} \omega^n \geq C_1 \frac{p}{(p+1)^2} (\dashint_X |(\varphi_{-})^{\frac{p+1}{2}}|^{\frac{2n}{n-1}})^{\frac{n-1}{2n}}-C_2 \dashint_X (|\varphi|^{\frac{p+1}{2}})^2 , $$
which yields
$$ \|\varphi_{-}\|_{L_\omega^{(p+1)\frac{n}{n-1}}} \leq (C(p+1))^{\frac{1}{p+1}} \|\varphi_{-}\|_{L_\omega^{p+1}} . $$
We iterate (observing that the constant in the above estimate stay under control) to get:
$$ \|\varphi\|_{L^\infty} \leq \tilde C \|\varphi\|_{L_\omega^2} . $$

We now prove the reduction from $L^2$-control to $L^1$-control thanks to the Poincaré inequality, see {\itshape e.g.} \cite[Corollary 4.3]{aubin-book}, \cite[page 164]{gilbarg-trudinger}:
\begin{eqnarray*}
C \dashint_X |\varphi| \omega^n &\geq& \dashint_X \varphi (1-\exp f)\omega^n \\
&\geq& \frac{1}{n} \dashint_X |\nabla \varphi |^2 \omega^n \\
&\geq& \frac{\lambda_j(\omega)}{n} (\dashint_X |\varphi|^2 \omega^n - (\dashint_X \varphi \omega^n)^2), \end{eqnarray*}
using the normalization: $\varphi \leq -1$.

We finally prove the $L^2$-control. By consider the Green function $G_\omega$, see {\itshape e.g.} \cite[Theorem 4.13]{aubin-book}: let $p_{\text{max}}$ a maximum point, then
$$ -1 = \varphi(p_{\text{max}}) = - \dashint_X |\varphi| \omega^n - \dashint_X \Delta_\omega \varphi G_\omega (p,\_) \omega^n . $$
Then, using that $\Delta_\omega \varphi>-n$, we get:
$$ -1 + \int_X |\varphi| \leq n \int_X G_\omega(p,\_) \omega^n \leq C . $$
This completes the $\mathcal{C}^0$ estimate for the first Calabi conjecture.

\subsection{Consequences of Calabi's conjectures}
We give here some applications of the Calabi conjecture.

\begin{theorem}[{Kobayashi \cite{kobayashi-Ann}, Yau \cite{yau-PNAS}}]
A Fano manifold is simply-connected.
\end{theorem}

\begin{proof}
Consider $\pi\colon \tilde X \to X$ the universal cover.
By assumption, $X$ is Fano, that is, there exists a K\"ahler form $\omega\in 2\pi c_1(X)>0$.
By the first Calabi conjecture, Theorem \ref{eq:1st-calabi-conj}, there exists a K\"ahler form $\omega'$ on $X$ such that $\mathrm{Ric}({\omega'})=\omega$. Consider the K\"ahler metric $\pi^*\omega'$ on $\tilde X$. It has $\mathrm{Ric}({\pi^*\omega')}=\pi^*\omega>0$ bounded below away from zero. Then, by the Myers theorem in Riemannian geometry, see {\itshape e.g.} \cite[Theorem 6.51]{besse}, $\tilde X$ is compact, and so it is Fano too.
By using that $X$ is Fano, that is the anti-canonical bundle is ample whence positive, and thanks to the Kodaira vanishing theorem, see {\itshape e.g.} \cite[Theorem 7.13]{voisin-1}, we have $h^{0,j}(X)=\dim H^{j}(X;\mathcal{O}_X)=\dim H^{j}(X;K_X\otimes K_X^{-1})=0$ for $j>0$; it follows that the Euler characteristic for the structure sheaf $\mathcal{O}_X$ is $\chi(\mathcal{O}_X):=\sum_j (-1)^j h^{0,j}(X)=1$.
For the same reasons, also $\chi(\tilde X)=1$.
On the other side, the Euler characteristic is additive on a cover, for example, as a consequence of the Hirzebruch-Riemann-Roch theorem. Then $X=\tilde X$.
\end{proof}
 
 As innocent as it seems, the above statement is kind of deep:
\begin{remark}
	\begin{itemize}
		\item The  Calabi's conjecture is crucially used to prove that the \emph{universal cover is compact}. All the other arguments would work the same for arbitrary finite covers, giving the weaker statement that $X$  has trivial pro-finite completion of the fundamental group.
		\item There exists a fully algebraic prove that makes use of rationally connectivity of Fanos (hence also very deep results in algebraic geometry based on Mori's Bend-and-Break, which uses reduction in finite  characteristic $p$ ({\itshape e.g.} \cite{mori, kollar-mori, debarre}).
		\end{itemize}
\end{remark}

The next consequence is a very famous inequality between Chern numbers. 

\begin{theorem}[{Bogomolov-Miyaoka-Yau inequality \cite{vandeven, bogomolov, miyaoka, yau-PNAS, yau-CPAM}}]
Let $X^n$ be a compact K\"ahler manifold.
\begin{itemize}
\item If $X^n$ is Fano K\"ahler-Einstein, then
$$ n c_1^n(X) \leq 2(n+1) c_1^{n-2}(X) c_2(X) , $$
with equality if and only if $X=\mathbb{CP}^n$.
\item If $X^n$ is with ample canonical bundle, and then K\"ahler-Einstein by Theorem \ref{eq:2nd-calabi-conj}, then
$$ n (-1)^{n-2} c_1^n(X) \leq 2(n+1) (-1)^{n-2} c_1^{n-2}(X) c_2(X) , $$
with equality if and only if $X^n=\mathbb{B}^n/\Gamma$.
\end{itemize}
\end{theorem}

\begin{proof}
For more details see {\itshape e.g.} \cite{tosatti-Expos}. First of all, we fix $\omega$ K\"ahler-Einstein, with cosmological constant $\lambda$, and we notice that we can rewrite the two inequalities that we want to prove as
$$ \left( \frac{2(n+1)}{n} c_2 - c_1^2 \right) [\omega]^{n-2} \geq 0 , $$
and we also want to prove that equality holds if and only if $X$ has constant holomorphic sectional curvature. Indeed, recall that simply-connected K\"ahler manifolds with constant holomorphic sectional curvature are holomorphically isometric to either $\mathbb{C}^n$, or $\mathbb{B}^n$, or $\mathbb{CP}^n$, up to homothety, see {\itshape e.g.} \cite[Theorems IX.7.8, IX.7.9]{kobayashi-nomizu-2} (the proof is very similar to the corresponding result in the real case). Recall that the {\em holomorphic sectional curvature} is $H(X) := K(X \wedge JX):= R(X,JX,JX,X)$ for $|X|=1$.
Then $H$ is constant if and only if the traceless part $R^0$ of $R$ vanishes, where:
$$ R^0_{i\bar j k \bar \ell} = R_{i\bar j k \bar \ell} - \frac{\lambda}{n+1} (g_{i\bar j}g_{k\bar \ell}+g_{i\bar \ell}g_{k\bar j} ). $$
Using the KE condition, we immediately see that \begin{equation}
\label{a} |R^0|^2=|R|^2-\frac{2\lambda^2 n}{n+1}
\end{equation}
Let $\Omega_i^j= \sqrt{-1}\sum {{R_i}^j}_{k \bar l} dz_k \wedge d\bar{z}_l$ be the curvature seen as a form valued endomorphism of the tangent bundle. By  Chern-Weil theory we can represent the first Pontryagin class (if you want, this a possible definition of such class) as: $$(2\pi^2)^{-1}\mathrm{tr}\,(\Omega\wedge\Omega):=(2\pi^2)^{-1}\sum_{i,k}\Omega_i^k\wedge \Omega_k^{i} \in p_1(X)=c_1^2(X)-2c_2(X).$$ 
Then, using the KE condition again, we compute:
\begin{equation} \label{b}
n(n-1)\mathrm{tr}\,(\Omega\wedge\Omega)\wedge \omega^{n-2}=(\lambda^2n-|R|^2)\omega^n
\end{equation} 

Finally, integrating \eqref{a} and \eqref{b} on $X$, 
$$ \frac{1}{n(n-1)4\pi^2} \int_X |R^0|^2 \omega^n = \left( 2 c_2(X)-\frac{n}{n+1} c_1^2(X) \right) [\omega]^{n-2} , $$
whence we get the statement.

\end{proof}

\begin{remark}
We would like to make some observations on the  Bogomolov-Miyaoka-Yau (BMY) inequality.
\begin{itemize}
\item For compact complex surfaces, the BMY inequalities reduce to
$$ c_1^2 \leq 3c_2 $$
where $c_2=e(X)$ is the topological Euler characteristic and $c_1^2=2e(X)+3\tau(X)$ with $\tau(X)$ the signature of the intersection form on the second cohomology. Hence it is purely topological. Surfaces which realize the equality are hard to construct ({\itshape e.g.} Mumford Fake projective spaces \cite{mumford-Fake}).

\item For Fano KE threefolds, the BMY inequality reduces to
$$ c_1^3(X) \leq 64 = c_1^3(\mathbb{CP}^3) $$
thanks to the fact that $c_1(X)c_2(X) = 24 \int_X \mathrm{td}(X)=\chi(\mathcal{O}_X)=1$. In dimension $n\geq 4$,  it is proven that for KE Fano manifolds  \cite{fujita-arxiv1508}	 
$$ c_1^n(X) \leq (n+1)^n= c_1^n(\mathbb{CP}^n)$$
continues to hold  (\emph{not} as a consequence of BMY, but using K- stability considerations). We should remark that there are Fano manifolds whose volume is bigger than $c_1^n(\mathbb{CP}^n)$: {\itshape e.g.} certain $\C\mathbb{P}^1$-bundles over $\C\mathbb{P}^3$. 
\end{itemize}
\end{remark}

\section{First obstructions to existence of K\"ahler-Einstein metrics}
In this section, we discuss two obstructions to the existence of K\"ahler-Einstein metrics:
\begin{description}
\item[Matsushima] the automorphisms group of K\"ahler-Einstein Fano manifolds is reductive (namely, it is the complexification of a maximal compact subgroup);
\item[Futaki] the functional $F$ defined on the space of holomorphic vector fields as in \eqref{eq:futaki} vanishes when there exists a K\"ahler-Einstein metric.
\end{description}
More in general, these obstructions hold for {\em constant scalar curvature K\"ahler metric} (cscK). Note that, if $\omega\in 2\pi c_1(X)$ is a cscK metric on $X^n$ Fano, then $\omega$ is in fact KE. Indeed, by the $\sqrt{-1}\partial\overline\partial$-Lemma,  we have that   $Ric(\omega)-\omega=\sqrt{-1}\partial\overline\partial f$. Taking trace with the metric, we find that $0=Sc(\omega)-n=\Delta_\omega f$, {\itshape i.e.} $f$ has to be constant. 

We are going to prove the statements in the more general cscK case, since they require only a bit more effort than in the KE case.

\subsection{Matsushima-(Lichnerowicz) theorem}\label{sec:matsushima}
On a $(X^n,g)$ Hermitian manifold, we consider the Lie algebra $\mathfrak{h}$ of {\em holomorphic vector fields} (that is, real vector fields $V$ such that $\mathcal{L}_V J=0$, equivalently, $\overline\partial V^{(1,0)}=0$). 

We also consider the Lie algebra $\mathfrak{k}_g$ of the isometry group $\mathrm{Isom}(X,g)$ (that is, $\mathfrak{k}_g$ contains vector fields $V$ such that $\mathcal{L}_V g=0$, called {\em Killing vector fields}). If $(X^n,g)$ is compact K\"ahler, then $\mathfrak{k}_g \subseteq \mathfrak{h}$, see {\itshape e.g.} \cite[Proposition 2.2.1]{gauduchon-book}.
Indeed, notice that the associated $(1,1)$-form $\omega$ to $g$ is harmonic, {\itshape e.g.}  since $d \ast \omega=\frac{1}{(n-1)!}d\omega^{n-1}=0$. Take any $\psi\in \mathrm{Isom}_0(X,g)$, the connected component of the identity. Then $\psi^*[\omega]=[\omega]$. By uniqueness of the harmonic representative, $\psi^*\omega=\omega$. Thus, we have proven that, for any $V\in \mathfrak{k}_g$, it holds $\mathcal{L}_{V}\omega=0$. This, together with $\mathcal{L}_Vg=0$, gives $\mathcal{L}_VJ=0$.

We also consider the Lie algebra $\mathfrak{a}_g$ of {\em parallel vector fields} (that is, $\mathfrak{a}_g$ contains vector fields $V$ such that $\nabla V=0$, where $\nabla$ is the Levi-Civita connection of $g$). Clearly, $\mathfrak{a}_g \subseteq \mathfrak{k}_g$ and $\mathfrak{a}_g$ is Abelian, since $\nabla$ is torsion-free. If $g$ is K\"ahler, then $\mathfrak{a}_g$ is complex, that is, if $V\in \mathfrak{a}_g$ then also $JV\in \mathfrak{a}_g$ (just using that $J$ is parallel).

Finally, we consider the space $\mathfrak{h}_0$ of {\em holomorphic vector fields with zeroes}, consisting of $V\in\mathfrak{h}$ such that $V(p)=0$ at some point $p$, and $\mathfrak{k}_0$ the space of Killing vector fields with zeros.

Before stating the next theorem, we recall that the {\em Albanese map} \cite{blanchard}, also called Jacobi map, associated to a compact K\"ahler manifold $X^n$ is
$$
A \colon X \to \mathrm{Alb}(X) := H^0(X;\Omega_X^1)^\vee \slash H_1(X;\mathbb Z) \simeq \mathbb C^{h^{1,0}} \slash \Gamma ,
$$
$$
p \mapsto \int_{p_0}^{p} \_ ,
$$
where $p_0$ is fixed. It is well defined since any holomorphic form is closed on a compact K\"ahler manifold. Any morphism from $X$ to a compact complex torus factors uniquely through the Albanese map.
It yields an isomorphism
\begin{equation}\label{eq:Alb-isom}
 A^* \colon H^0(\mathrm{Alb}(X);\Omega^1_{\mathrm{Alb}(X)}) \stackrel{\simeq}{\to} H^0(X;\Omega^1_X) .
\end{equation}

We also consider the operator that acts on $f\in\mathcal{C}^\infty(X;\mathbb C)$ as
$$ \partial^\sharp f := (\mathrm{grad}_g f)^{1,0} \stackrel{\text{loc}}{=} g^{i\bar j}f_{,\bar j} \partial_{z_i} . $$

\begin{theorem}[{Matsushima \cite{matsushima}, Carrell-Lieberman \cite{carrell-lieberman}, LeBrun-Simanca \cite{lebrun-simanca}}]
Let $(X^n,\omega)$ be a compact K\"ahler manifold. Then the following properties are equivalent for a holomorphic vector field $Z\in\mathfrak{h}$:
\begin{enumerate}[(i)]
\item $Z\in\mathfrak{h}_0$;
\item $Z$ is tangent to the fibres of the Albanese map;
\item $Z=\partial^\sharp f$ for some $f\in\mathcal{C}^\infty(X;\mathbb C)$.
\end{enumerate}
\end{theorem}

\begin{proof}

We prove that {\itshape (i) implies (ii)}.
We want to prove that $Z$ is tangent to the fibres of the Albanese map, equivalently, $\mu(dA(Z))=0$ for any $\mu\in H^0(\mathrm{Alb}(X),\Omega^1_{\mathrm{Alb}(X)})$. By the isomorphism \eqref{eq:Alb-isom}, this is equivalent to show that $\eta(Z)=0$ for any $\eta \in H^0(X,\Omega^1_X)$, where $\eta=A^*\mu$.
We conclude by noticing that $\eta(Z)$ is a holomorphic function on $X$ compact, then it is constant, in fact, $\eta(Z)=0$ because $Z$ has non-empty zero-set.

We prove that {\itshape (ii) implies (iii)}.
Define the following $(0,1)$-form:
$$ \varphi:=\sqrt{-1}\iota_Z\omega \stackrel{\text{loc}}{=} \sqrt{-1} Z^kg_{k\bar j} d\bar z^j . $$
By the Cartan formula, we have $\mathcal{L}_{\Re Z}\omega=-\sqrt{-1}d\varphi$, where the right-hand side is a $(1,1)$-form since $Z$ is holomorphic: then $\overline\partial\varphi=0$.  We now use Hodge theory to construct the potential $f$. Take any $\alpha\in\mathcal{H}^{0,1}_{\omega}$, the space of harmonic $(0,1)$-form; set $\beta:=\bar\alpha$. Then $(\varphi,\alpha)_\omega=g^{k\bar j}\varphi_k\overline{\alpha_j} =\sqrt{-1}Z^j\beta_j=\sqrt{-1}\beta(Z)=0$ by the assumption that $Z$ is tangent to the fibres of the Albanese map. We have proven that $\varphi\perp\mathcal{H}^{0,1}_\omega$. By Hodge theory, we have $\varphi=\overline\partial f + \overline\partial^* h$. But $\|\overline\partial^*h\|_{L^2}^2=(\overline\partial^* h, \overline\partial^* h)_{L^2}=(h,\overline\partial(\varphi-\overline\partial f))_{L^2}=0$. So $\varphi=\overline\partial f$, and $Z=\partial^\sharp f$.

We prove that {\itshape (iii) implies (i)}. Take $Z=\partial^\sharp f$ holomorphic.
(Note that $f$ is complex-valued, and not just real-valued, so we cannot apply the maximum principle directly.)
We have that $Z(\bar f)=g^{i\bar j}f_{,\bar j}\overline{f_{,\bar i}}=|Z|^2\geq 0$. Set $c=\min|Z|^2>0$ and $C=\max|Z|^2$. We have to prove that $c=0$.
Fix $p\in X$. Construct $F\colon \mathbb C \to X^n$ with $F(0)=p$ and $dF(\frac{d}{dz})=Z$. Define $h:=f \circ F\colon \mathbb C \to \mathbb C$. We have $0\leq |h| \leq C$ and $\frac{d\bar h}{dz_j} = Z(\bar f)\geq c$. Apply the Stokes theorem to a disk $\Delta_r$ of radius $r>0$. We have $\frac{1}{2\pi}\int_0^{2\pi} h(r\exp(\sqrt{-1}\vartheta)) \exp(\sqrt{-1}\vartheta) d\vartheta = \frac{1}{2\pi r}\int_{\partial \Delta_r} h dz = \frac{2\sqrt{-1}}{8\pi r}\int\frac{d\bar h}{dz} dx dy \geq cr$. Letting $r\rightarrow +\infty$ we see that $c=0$, since the first integral is clearly bounded. 
\end{proof}

\begin{remark}
\begin{itemize}
 \item When $X^n$ is Fano, then $\mathfrak{h}=\mathfrak{h}_0$ (this follows either by $\mathrm{Alb}(X)$ being trivial by the vanishing result, or because $X$ is simply-connected).
 \item $\mathfrak{h}_0$ is an ideal in $\mathfrak{h}$ when the manifold is K\"ahler, but there are counterexamples in the non-K\"ahler case where it is not a linear subspace (consider for example the Hopf surface $\mathbb C^2\setminus 0 \slash \langle z \mapsto 2z\rangle$, with local coordinates $(z,w)\in\mathbb C^2$, then both $v=z\partial_z \in \mathfrak{h}_0$ and $t=w\partial_w \in \mathfrak{h}_0$, but $v+t \not \in \mathfrak h_0$).
\end{itemize}
\end{remark}
		
		Cleraly $Z=\partial^\sharp f$ is holomorphic if and only if $f \in \mbox{ker}(\bar \partial \partial^\sharp)^*\bar \partial \partial^\sharp$. Such fourth order operator is called \emph{Lichnerowicz operator} and by a (slightly tedious) computation can be show to be equal to  $$\mathbb{L}_\omega(\varphi):=\Delta_\omega^2\varphi + (\mathrm{Ric}(\omega), \partial\overline\partial\varphi)_\omega + \partial^\sharp \varphi(\mathrm{Sc}_\omega) , $$
		see {\itshape e.g.} \cite[page 59]{gabor-book}, \cite[page 63, Proposition 2.6.1]{gauduchon-book}.

In general, $\mathbb{L}_\omega$ is not a real operator, that is, $\overline{\mathbb{L}_\omega(\varphi)} \neq \mathbb{L}_\omega(\bar \varphi)$. We will define the complex conjugate
$$ \tilde{\mathbb L}_\omega(\varphi) := \overline{\mathbb{L}_\omega(\bar\varphi)} . $$
We have $\tilde{\mathbb{L}}_\omega(\varphi)-\mathbb{L}_\omega(\varphi)=\partial^\sharp\varphi(\mathrm{Sc}_\omega)-\overline\partial^\sharp\bar\varphi(\mathrm{Sc}_\omega)$.

\begin{proposition}
Let $(X^n,\omega)$ be a compact K\"ahler manifold.
Then the space of Killing vector fields with zeroes is given by
$$ \mathfrak{k}_0 = \left\{ J \mathrm{grad}_g\, f \;:\; f \in \mathcal{C}^\infty(X;\mathbb R) \text{ such that }\partial^\sharp f\in\mathfrak{h}_0 \right\} . $$
\end{proposition}

\begin{proof}
The inclusion "$\supseteq$" follows since $J\mathrm{grad}_g\, f$ is a real holomorphic vector field being also Hamiltonian, whence it is Killing.
The inclusion "$\subseteq$" is as follows. Take $X\in\mathfrak{k}_0$. Then $JX\in\mathfrak{h}_0$. Consider $Z=JX+\sqrt{-1}X$ with $\overline\partial Z=0$. We know that $Z=\partial^\sharp f$ where $f\in\mathcal{C}^\infty(X;\mathbb C)$, say $f=u+\sqrt{-1}v$. Then $\Re\partial^\sharp f=\mathrm{grad}_g\,  u + J \mathrm{grad}_g\,  v$. Then $X=-\mathrm{grad}_g\,  v+J \mathrm{grad}_g\,  u$. We claim that $\mathrm{grad}_g\, v=0$. By Cartan, $0=\mathcal{L}_X\omega = d\iota_{-\mathrm{grad}_g\, v + J \mathrm{grad}_g\, u}(\omega)=d g (-J\mathrm{grad}_g\, v,\_)-d({g(\mathrm{grad}_g\, u,\_)})=dJd v=2\sqrt{-1}\partial\overline\partial v$. Since $X$ is compact, then $v$ is constant.
\end{proof}

\begin{theorem}[{Matsushima \cite{matsushima}, Lichnerowicz \cite{lichnerowicz}}]

Let $\omega$ be a cscK metric on $X^n$ compact complex manifold. Then
$$ \mathfrak{h}_0 = \mathfrak{k}_0 \oplus J\mathfrak{k}_0 . $$
\end{theorem}

\begin{proof}
Clearly, $\mathfrak{k}_0 \oplus J\mathfrak{k}_0 \subseteq \mathfrak{h}_0$. Take $Z=\partial^\sharp f$ with $\mathbb{L}_{\omega}(f)=0$. Since $\omega$ is cscK, then $\mathbb L=\tilde{\mathbb L}_\omega$ is a real operator. Then $\mathbb{L}_\omega(\Re f)=\mathbb{L}_\omega(\Im f)=0$. Then we can conclude by the previous proposition.
(See also {\itshape e.g.} \cite[Theorems 3.6.1-2]{gauduchon-book}, \cite[Proposition 4.18]{gabor-book}).
\end{proof}

\begin{remark}
By Bochner \cite{bochner}, if $V$ is Killing, then $\frac{1}{2}\Delta_g |V|^2 = - \mathrm{Ric}(g)(V,V)+|\nabla V|^2$. Thus:
\begin{itemize}
\item if $\mathrm{Ric}(\omega)=0$, then $\mathfrak{h}=\mathfrak{a}$;
\item if $\mathrm{Ric}(\omega)=-\omega$, then $\mathfrak{h}=0$.
\end{itemize}
\end{remark}

\begin{remark}
Recall that $\mathrm{Isom}(M,g)$ is always a compact Lie group \cite{myers-steenrod}.
In particular, $\mathfrak{h}_0$ is reductive.
\end{remark}

\begin{example}
Consider $\hat X^2 = \mathrm{Bl}_p\mathbb{CP}^2$, that is, locally on $U\ni p$, it is $\hat X \stackrel{\text{loc}}{=} \{((z,w),[u:v]) \in U \times \mathbb{CP}^1 \;:\; zv-wu=0\}$. We have $K_{\hat X}=\pi^* K_{\mathbb{CP}^2} \otimes \mathcal{O}(E)$, where $E\simeq \mathbb{CP}^1$ is the exceptional divisor. $K^{-1}_{\hat X}$ is still ample. Then
$$ \mathrm{Aut}(\hat X) = \left. \left\{\left(\begin{matrix}
\star & \star & \star \\
0 & \star & \star \\
0 & \star & \star
\end{matrix}\right)\right\} \middle\slash \mathbb C^* \right. $$
and
$$ \mathfrak{h}(\hat X) = \left\{\left(\begin{matrix}
0 & \star & \star \\
0 & \star & \star \\
0 & \star & \star
\end{matrix}\right)\right\} \simeq \mathfrak{gl}(2;\mathbb C) \oplus \mathbb C^2 $$
which cannot be given by a complexification of a compact Lie group.
Then, by the Matsushima-Lichnerowicz theorem, $\hat X$ is an example of a Fano manifold that does not admit cscK metric, in particular, does not admit K\"ahler-Einstein metrics.
\end{example}

Even if the blow-up in a point of $\C \mathbb{P}^2$ is not KE, such space is related to two important (positive) Einstein metrics.

\begin{itemize}
\item The \emph{Page metric} \cite{page} (see {\itshape e.g.} \cite[9.125]{besse})  constructed as $g_P=\mathrm{Sc}_{\text{extr}}^{-2}g_{\text{extr}}$ where $g_{\text{extr}}$ is the explicit Calabi's extremal metric (that is, with holomorphic grandient of the scalar curvature) in $2\pi c_1(X)$.
\item The $Y^{2,1}$ irregular Sasaki-Einstein metric constructed  on the smooth bundle  $S^1 \hookrightarrow K \to \mbox{Bl}_p \C \mathbb{P}^2 $ by  deforming the natural $\C^\ast$-action. This metric was first constructed by physicists studying the AdS-CFT correspondence \cite{martelli-sparks}.
\end{itemize}

\subsection{Futaki invariant}\label{sec:futaki}
Let $(X^n,\omega)$ be a compact K\"ahler metric. Set $\hat S$ the value of the eventual constant scalar curvature of some metric in $[\omega]$. Note that $\hat S$ is a cohomological invariant depending just on $c_1(X)$ and $[\omega]$, determined by
$$[\omega]^n\hat S = \int_X \mathrm{Sc}_\omega \omega^n= n  \int_X Ric(\omega) \wedge \omega^{n-1}=2\pi n c_1(X).[\omega]^{n-1}. $$
For $Z=\partial^\sharp f$, define
\begin{equation}\label{eq:futaki}
F(Z,\omega) := \int_X f (\mathrm{Sc}_\omega-\hat S) \omega^n .
\end{equation}

\begin{theorem}[{Futaki \cite{futaki}}]
If $\omega$ and $\omega'$ are cohomologous, then $F(Z,\omega)=F(Z,\omega')$. Therefore $F(Z,\omega)=F(Z,[\omega])$. In particular, if there exists $\omega'\in[\omega]$ cscK, then $F=0$.
\end{theorem}

\begin{proof}
Note that the space of K\"ahler metrics in $[\omega]$ is connected. Take $(\omega_t)_t$ a path of metrics in $[\omega]$, for example of the form $\omega_t=\omega+\sqrt{-1}\partial\overline\partial(t\varphi)$. We show that $\dot F(Z,\omega_t):=\frac{d}{dt}\lfloor_{t=0}F(Z,\omega_t)=0$. We compute:
\begin{itemize}
\item $\dot f_t=\partial^\sharp f(\varphi)$:\\
Indeed, $\overline\partial \dot f_t = \overline\partial(Z^j\varphi_{,j})$ since $Z$ is holomorphic;
\item $\dot{\mathrm{Sc}}_{\omega_t} = -\tilde{\mathbb L}_\omega(\varphi)+\partial^\sharp \mathrm{Sc}_\omega(\varphi)$:\\
indeed, $\mathrm{Ric}({\omega_t})=\mathrm{Ric}(\omega)-\sqrt{-1}t\partial\overline\partial\Delta\varphi+o(t)$, then $\mathrm{Sc}_{\omega_t}=g_t^{i\bar j}(\mathrm{Ric}({\omega_t}))_{i\bar j} = (g^{i\bar j}-tg^{i\bar \ell}g^{k\bar j} \varphi_{,k\bar \ell})(\mathrm{Ric}_{i\bar j}-t(\Delta\varphi)_{,i\bar j})+o(t)=\mathrm{Sc}_\omega-t\Delta^2\varphi-tg^{i\bar \ell}g^{k\bar j}\varphi_{,k\bar \ell}\mathrm{Ric}_{i\bar j}+o(t)=\mathrm{Sc}_\omega-t\tilde{\mathbb L}_{\omega}+t\partial^\sharp \mathrm{Sc}_\omega(\varphi)+o(t)$;
\item $\dot\omega_t^n=\Delta \varphi\omega^n$.
\end{itemize}
Thus
\begin{eqnarray*}
\dot F_t &=&
\int_X (\partial^\sharp f(\varphi) (\mathrm{Sc}_\omega-\hat S)- f \bar{\mathbb L}_\omega(\varphi)+f \partial^\sharp \mathrm{Sc}_\omega(\varphi) + f(\mathrm{Sc}(\omega)-\hat S)\Delta\varphi) \omega^n \\
&=& - \int_X f \bar{\mathbb{L}}_\omega(\varphi)\omega^n = - (f, \mathbb{L}_\omega(\varphi))_{L^2} \\
&=& - (\mathbb{L}_\omega(f), \varphi)_{L^2} = 0 ,
\end{eqnarray*}
by integrating by part and by using that $\mathbb{L}_\omega(f)=0$, since $f$ is holomorphic. 
\end{proof}

The  Futaki invariant will be  central in the definition of the notion of K-stability, see Section \ref{sec:K-stab-variational-pov}.

\begin{remark}[Other equivalent formulations of the Futaki invariant]
Let $\psi$ a solution of $\mathrm{Sc}_\omega-\hat S = \Delta_\omega \psi$, so, for $Z=\partial^\sharp f$, we have
\begin{eqnarray*}
F(Z,[\omega])
&=& \int_X f (\mathrm{Sc}(\omega)-\hat S) \omega^n \\
&=& (f, \overline\partial^*\overline\partial\psi) = (\overline\partial f, \overline\partial\psi) \\
&=& \int_X \partial^\sharp f (\psi) = \int_X Z(\psi) \omega^n
\end{eqnarray*}
so that we can write
$$ F(Z,[\omega])=\int_X Z(\psi) \omega^n . $$
\end{remark}

\subsection{Mabuchi functional and evidences for uniqueness}\label{sec:mabuchi}

On the space $\mathcal{K}_{[\omega]}=\{\varphi \;:\; 	\omega_\varphi:=\omega+\sqrt{-1}\partial\overline\partial\varphi>0\}$ of K\"ahler metrics in $[\omega]$, consider the following $1$-form:
$$ \alpha_\varphi(\psi) := \int_X \psi (\hat S - \mathrm{Sc}_{\omega_\varphi}) \omega_\varphi^n .$$

We compute
\begin{eqnarray*}
\frac{d}{dt}(\alpha_{\varphi+t\psi_2}(\psi_1))
&=& \int (\psi_1(\mathbb{L}_{\omega_\varphi}\psi_2-\partial^\sharp\psi_2(\mathrm{Sc}_{\omega_\varphi}))+\psi_1(\hat S-\mathrm{Sc}_{\omega_\varphi})\Delta_{\omega_\varphi}\psi_2)\omega_\varphi^n \\
&=& \int_X (\psi_1 \mathbb{L}_{\omega_\varphi} \psi_2 - (\hat S - \mathrm{Sc}_{\omega_\varphi})\partial^\sharp\psi_1(\psi_2))
\end{eqnarray*}
which is symmetric in $\psi_1$ and $\psi_2$, since they are both real.
Then
$$ d\alpha=0 . $$
Then there exists $\mathcal{M}\colon \mathcal{K}_{[\omega]} \to \mathbb R$ such that
$$ \alpha=d\mathcal{M} . $$
Such $\mathcal{M}$ is called {\em Mabuchi K-energy}.
(See also {\itshape e.g.} \cite[Section 4.10]{gauduchon-book}, \cite[Proposition 4.23]{gabor-book}.)

The {\em Mabuchi metric} \cite{mabuchi, semmes, donaldson-AMST} on $\mathcal{K}_{[\omega]}$ is defined as
$$ \langle \psi_1, \psi_2 \rangle_{\varphi} := \int_X \psi_1 \psi_2 \omega_\varphi^n . $$
We look at the geodesics for such metric: let $(\varphi_t)_t$ a path in $\mathcal{K}_{[\omega]}$ connecting $\varphi_0=0$ and $\varphi_1=\varphi$, and consider the energy functional
$$ E(\varphi) = \int_0^1 \int_X \dot\varphi_t^2 \omega_{\varphi_t}^n dt . $$
We look at geodesics, that is, critical point of $E$. We compute
$$ \left.\frac{d}{ds}\right\lfloor_{s=0} E(\varphi+s\psi) = \int_0^1 \int_X -2 \psi_t (\ddot \varphi_t -\partial^\sharp \dot \varphi_t (\dot \varphi_t)) \omega_{\varphi_t}^n $$
so the geodesic equation is
$$ \ddot \varphi_t - |\partial\dot\varphi_t|^2_{\omega_{\varphi_t}} = 0 . $$

\begin{remark}\label{rmk:donaldson-semmes-chma}
Consider now $\hat X = X \times [0,1] \times \mathbb{S}^1 $. Donaldson \cite{donaldson-AMST} and Semmes \cite{semmes} observed that, by extending
$$ \Phi(p,t,s):=\varphi_t(p)$$
then the geodesic equation can be given as
$$ (\omega_0+\sqrt{-1}\partial\overline\partial\phi)^{n+1}=0 $$
on $\hat X$, from which one can use pluripotential theory to study geodesics \cite{chen-JDG00}.
\end{remark}

\begin{lemma}
The Mabuchi functional is convex along smooth geodesics.
\end{lemma}

\begin{proof}
Let $(\varphi_t)_t$ be a geodesic.
We compute
\begin{eqnarray*}
\frac{d^2}{dt^2}\mathcal{M}(\varphi_t)
&=& \frac{d}{dt} \int_X \dot \varphi_t (\hat S - \mathrm{Sc}_{\omega_{\varphi_t}}) \omega_{\varphi_t}^n \\
&=& \int_X (\ddot \varphi_t (\hat S - \mathrm{Sc}_{\omega_{\varphi_t}}) + \dot\varphi_t (\mathbb{L}_\omega(\varphi_t)-\partial^\sharp\varphi_t(\mathrm{Sc}_{\omega_{\varphi_t}})) + \dot\varphi_t(\hat S-\mathrm{Sc}_{\omega_{\varphi_t}})\Delta_{\omega_{\varphi_t}}\dot\varphi_t)\omega_{\omega_{\varphi_t}}^n \\
&=& \int_X (|\overline\partial\partial^\sharp\dot\varphi_t|_{\omega_{\varphi_t}}^2+(\hat S -\mathrm{Sc}_{\omega_{\varphi_t}})(\ddot\varphi_t-|\partial\dot\varphi_t|^2))\omega_{\omega_{\varphi_t}}^n \\
&=& \int |\overline\partial\partial^\sharp\dot\varphi_t|_{\omega_{\varphi_t}}^2\omega_{\omega_{\varphi_t}}^n \geq 0 .
\end{eqnarray*}
\end{proof}

Now, assume that any $\varphi_1$, $\varphi_2$ can be connected by smooth geodesics \cite{donaldson-AMST}. The previous result would then yield uniqueness of cscK metrics. Indeed, if there existed two cscK metrics, by convexity we get that $\overline\partial\partial^\sharp\varphi_t=0$, and then there exists $F$ automorphism such that $F^*\omega_2=\omega_1$ ({\itshape i.e.} the metric would be unique up to automorphisms). But the assumption on the regularity of geodesics is false: the optimal regularity is $C^{1,1}$ as proven by \cite{darvas-lempert}, which follows previous works \cite{donaldson-JSG, lempert-vivas}. In any case, such approach turned out to be the correct one, at the price of working with weaker regularity \cite{chu-tosatti-weinkove, berman-berndtsson}. In the KE Fano case, the first proof of uniqueness (up to automorphisms) was given by Bando and Mabuchi with different techniques \cite{bando-mabuchi}.

\section{A criterion for the existence of K\"ahler-Einstein metrics on Fano manifolds}

We try to construct a K\"ahler-Einstein metric in $2\pi c_1(X)>0$, {\itshape i.e.}, $\mathrm{Ric}(\omega)=\omega$. This is equivalent to solve the Monge-Ampère \eqref{eq:eq-ke} with $\lambda=1$, namely, the equation
\begin{equation}\label{eq:ke-fano}
\omega_\varphi^n = \exp(f-\varphi)\omega^n
\end{equation}
for $\omega_\varphi:=\omega+\sqrt{-1}\partial\overline\partial\varphi>0$,
where $f$ is the Ricci potential normalized such that $\dashint_X \exp f \omega^n=1$.

We set up the continuity method along the {\em Aubin path}
\begin{equation}\label{eq:ke-fano-continuity}
\omega_{\varphi_t}^n = \exp(f-t\varphi)\omega^n
\end{equation}
varying $t\in[0,1]$.
This equation is equivalent to asking
$$ \mathrm{Ric}(\omega_{\varphi_t}) = (1-t) \omega + t \omega_{\varphi_t} \geq t \omega_{\varphi_t} > 0 , $$
Note that we (crucially!)  get a lower bound on the Ricci curvature along the path.

\begin{remark}
Define the {\em Székelyhidi invariant} \cite{tian-IJM, gabor-CM} of $(X^n,\omega)$ as
$$ R(X) := \sup \{ t \in [0,1] \;:\; \mathrm{Ric}(\omega_{\varphi_t}) \geq t \omega_t \} \in (0,1] , $$
where $(\omega_t)_t$ is the Aubin path starting at the background metric $\omega$. In fact, $R(X)$ does not depend on the background $\omega$ \cite[Theorem 1]{gabor-CM}. For example, $R(\mathrm{Bl}_p\mathbb{CP}^2)=\frac{6}{7}$.  Explicit computations for toric manifolds can be found in \cite{li-AdvMath}. It was conjectured \cite[Conjecture 13]{gabor-CM} that $R(X)=1$ if and only if $K$-semi-stable. Such conjecture has been proven by C. Li \cite{li-CCM}.
\end{remark}

Solution for $t=0$ the solvability of the Aubin's path follows by first Calabi's conjecture, Theorem \ref{thm:1st-calabi-conj}. Openness is studies by means of the linearization
$$ D \mathcal{E}(\_,t) = \Delta_{\omega_{\varphi_t}}+t , $$
with the same notation as at page \pageref{page:open-2nd-calabi}. It is relative easy to show that the first eigenvalue of the Laplacians satisfies $-\lambda_1(t) \geq t$, with equality only for $t=1$  see {\itshape e.g.} \cite[Remark 6.13]{tian-book}. Thus we have that the kernel is trivial for $t\in(0,1)$.
For $t=1$, we have that the kernel is given by holomorphic vector fields with zeros.

\subsection{Some {\itshape a priori} estimates}
The Laplacian and higher order estimates can be repeated. We study now the $\mathcal{C}^0$-estimates, which we know they would fail in general, but we can hope to find some sufficient condition for them to hold.

On $(X^n,\omega)$ compact K\"ahler manifold, set $\omega_\varphi:=\omega+\sqrt{-1}\partial\overline\partial\varphi$, and consider the functional
$$ I_\omega(\varphi) = \dashint_X \varphi (\omega^n-\omega_\varphi^n) . $$
We have $I_\omega(\varphi)\geq 0$. In fact, $I_\omega$ is equivalent to the {\em generalized energy} functional
$$ J_\omega(\varphi) := \sqrt{-1} \sum_{k=0}^{n-1} \frac{k+1}{n+1} \dashint_X \partial\varphi \wedge \overline\partial\varphi \wedge \omega^k \wedge \omega_\varphi^{n-k} \geq 0 . $$
More precisely, \cite[Lemma 2.2]{tian-Inv87},
$$ \frac{1}{n+1} I_\omega(\varphi) \leq J_\omega(\varphi) \leq \frac{n}{n+1} I_\omega(\varphi) . $$

\begin{theorem}[{Tian \cite{tian-Inv87}}]\label{thm:tian-estimates-fano}
Let $(X^n,\omega)$ be a compact Fano K\"ahler manifold. Along the Aubin path $\{\omega_{\varphi_t}:=\omega+\sqrt{-1}\partial\overline\partial\varphi_t\}_{t\in[0,1]}$, we have the {\itshape a priori} estimates
$$ \|\varphi_t\|_{L^\infty} \leq C_t ( 1 + I_\omega(\varphi_t) ) , $$
with $C_t\leq C(\epsilon)$  for all $t>\epsilon$ for any chosen $\epsilon>0$.

\end{theorem}

\begin{proof}
The equation for the Aubin path yields $\dashint_X \exp(t\varphi)\omega_\varphi^n=\dashint_X\exp f\omega^n=1$, where recall that $f$ is the Ricci potential of $\omega$ with the normalization above. It follows that $\inf\varphi\leq 0$ and $\sup\varphi\geq 0$. We claim that
$$ \|\varphi\|_{L^\infty} \leq \mathrm{osc}\varphi := \sup\varphi - \inf\varphi \leq C_t ( 1 + I_\omega(\varphi_t) ) . $$

As for $\sup\varphi$, an estimate follows by using the Green function as for first Calabi's conjecture:
$$ \sup\varphi \leq C \left( 1 + \dashint_X \varphi_t \omega^n \right) . $$
Indeed, if $G$ is the Green function such that $G\geq0$, then write $\varphi(x)=\dashint_X \varphi \omega^n+\dashint -\Delta_\omega \varphi G(x,\_) \omega^n \leq C(1+\dashint_X \varphi \omega^n)$ since $-\Delta_\omega \varphi\leq n$.

We prove now an estimate for $\inf\varphi$. Take $1\leq \tilde\varphi_t := \max\{1,-\varphi_t\}$. Note that $\tilde\varphi_t^n(n-\Delta_{\omega_{\varphi_t}}\varphi_t)\geq 0$, and then integrate on $X$ with respect to the moving metric  $\omega_{\varphi_t}$ to get
$$ \dashint_X \left| \nabla_{\omega_{\varphi_t}}(\tilde\varphi_t)^{\frac{p+1}{2}} \right| \omega_{\varphi_t}^n\leq C(p,n)\dashint_X \tilde\varphi^{\frac{p+1}{2}} \omega_{\varphi_t}^n . $$
Use the Sobolev embedding with respect to $\omega_{\varphi_t}^n$. Since $\mathrm{Ric}({\omega_{\varphi_t}})\geq t$, then we can take a Sobolev constant that depends only on $C$ and $n$. Then, by Moser's iteration, and using the recalled first eigenvalue estimate:
\begin{eqnarray*}
\sup\tilde\varphi_t &\leq& \|\tilde\varphi_t\|_{L^2} \leq \frac{C}{t} (\|\tilde\varphi_t\|_{L^1} + 1) \\[5pt]
&\leq& \frac{C}{t} (1-\dashint_X \varphi_t \omega_{\varphi_t}^n ) ,
\end{eqnarray*}
where the second inequality follows by the Poincaré inequality.
At the end, we get
$$ \|\varphi_t\|_{L^\infty} \leq \sup\varphi_t-\inf\varphi_t \leq C (1+\dashint_X \varphi_t \omega^n)+\frac{C}{t} (1-\dashint_X \varphi_t\omega_{\varphi_t}^n) , $$
proving the statement.
See also {\itshape e.g.} \cite[Lemma 6.19]{tian-book}.
\end{proof}

\begin{definition}[{Tian's properness \cite{tian-Inv87}}]
Let $(X^n,\omega)$ be a compact K\"ahler manifold. Consider the space $\mathcal{K}_{[\omega]}$ of K\"ahler metrics in the K\"ahler class $[\omega]$. A functional $F_\omega$ on $\mathcal{K}_{[\omega]}$ is called {\em proper} if
$$ F_\omega(\varphi) \geq f(I_\omega(\varphi)) $$
for some $f(t)$ increasing function as $t\to+\infty$.
\end{definition}

We take the Mabuchi functional $\mathcal{M}_\omega$. Recall that it is the primitive of the closed $1$-form $\alpha(\psi)=(\psi,\mathrm{Sc}_{\omega_\varphi}-\hat S)_{L^2_{\omega_\varphi}}$, and it can be expressed as, see {\itshape e.g.} \cite[page 95]{tian-book},
$$ \mathcal{M}_\omega(\varphi) = \dashint_X \log\frac{\omega_\varphi^n}{\omega^n} \omega_\varphi^n + \dashint_X \left( f (\omega^n-\omega_\varphi^n) - (I_\omega-J_\omega)(\varphi) \right) , $$
where $f$ is the Ricci potential of $\omega$.
Here the first integral represents the "entropy" and the second integral is the "energy".

\begin{theorem}[{Tian}]\label{thm:mabuchi-proper-ke}
Let $(X^n,\omega)$ be a compact Fano K\"ahler manifold.
If the Mabuchi functional $\mathcal{M}_\omega$ is proper, then there exists a K\"ahler-Einstein metric.
\end{theorem}

\begin{proof}
 Let $\varphi_t$ a solution of \eqref{eq:ke-fano-continuity}. We compute
 \begin{eqnarray*}
 \frac{d}{dt} \mathcal{M}_\omega(\varphi_t) &=& - \frac{d}{dt} \dashint_X t \varphi_t \omega_{\varphi_t}^n - \frac{d}{dt}(I_\omega-J_\omega)(\varphi_t) \\
 &=& - \dashint_X \varphi_t \omega_{\varphi_t}^n - t \dashint_X \dot\varphi_t\omega_{\varphi_t}^n - t\dashint \varphi_t \Delta_{\omega_{\varphi_t}} \dot\varphi_t \omega_{\varphi_t}^n - \frac{d}{dt}(I_\omega-J_\omega)(\varphi_t) \\
 &=& \dashint_X 	\Delta_{\omega_{\varphi_t}}\dot\varphi_t \omega_{\varphi_t}^n - t \dashint_X \varphi_t \Delta_{\omega_{\varphi_t}} \dot \varphi_t\omega_{\varphi_t}^n - \frac{d}{dt}(I_\omega-J_\omega) (\varphi_t) \\
 &=& - t \dashint_X \varphi_t \Delta_{\omega_{\varphi_t}} \dot \varphi_t\omega_{\varphi_t}^n - \frac{d}{dt}(I_\omega-J_\omega) (\varphi_t) ,
 \end{eqnarray*}
 where we are using that $\Delta_{\omega_{\varphi_t}}\dot\varphi_t=-\varphi_t-t\dot\varphi_t$ as follows by differentiating the equation.
 By the equation, we compute
 \begin{eqnarray*}
  \frac{d}{dt} (I_\omega - J_\omega)(\varphi_t) &=& - \dashint_X \varphi_t \Delta_{\omega_{\varphi_t}} \dot \varphi_t \omega_{\varphi_t}^n \\
  &=& \dashint_X (\Delta_{\omega_{\varphi_t}}\dot\varphi_t+t\dot\varphi_t)\Delta_{\omega_{\varphi_t}}\dot\varphi_t \geq 0 .
 \end{eqnarray*}
 At the end, we get
 $$ \frac{d}{dt} \mathcal{M}_\omega(\varphi_t) = -(1-t) \frac{d}{dt} (I_\omega-J_\omega)(\varphi_t) \leq 0 . $$
 If we assume that $\mathcal{M}_\omega$ is proper, then $I_\omega-J_\omega$ is bounded along the Aubin path. Then, by Theorem \ref{thm:tian-estimates-fano}, we get $\|\varphi_t\|_{L^\infty} \leq C$.
 (See also {\itshape e.g.} \cite[Theorem 7.13]{tian-book}.)
\end{proof}

\begin{remark} It is proven  by Tian in \cite{tian-Inv} that for manifold with discrete automorphism also the {\itshape viceversa} of the above Theorem holds. This was used for proving that there exists a Fano threefold with no holomorphic vector fields which has no KE metric (the famous deformation of the Mukai-Umemura manifold \cite[Section 7]{tian-Inv}). This result was crucial for the understanding of the relations between existence of KE metrics and stability conditions, and it sees in the work of Berman we will describe in the last lecture its ultimate incarnation.
\end{remark}

\subsection{$\alpha$-invariant criterion}

For a compact K\"ahler manifold $(X^n,\omega)$, define the {\em $\alpha$-invariant} \cite{tian-Inv87} as
\begin{eqnarray}\label{eq:tian-alpha}
\alpha(X) &:=& \sup \{ \alpha>0 \;:\; \text{ there exists } C_\alpha \text{ such that,} \nonumber\\
&& \text{for any } \omega+\sqrt{-1}\partial\overline\partial\varphi>0, \nonumber\\
&& \text{it holds } \dashint_X \exp(-\alpha(\varphi-\sup\varphi)) dV_g \leq C_\alpha \} .
\end{eqnarray}

where $dV_g$ is any smooth volume form.

\begin{theorem}[{Tian \cite[Theorem 2.1]{tian-Inv87}}]
Let $(X^n,\omega)$ be a compact Fano K\"ahler manifold.
If $\alpha(X)>\frac{n}{n+1}$, then there exists a K\"ahler-Einstein metric.
\end{theorem}

\begin{proof}
Let $\alpha<\alpha(X)$ be such that the condition in \eqref{eq:tian-alpha} holds. By definition, there exists $C_\alpha$ such that, for any $\omega+\sqrt{-1}\partial\overline\partial\varphi>0$,

\begin{eqnarray*}
 \log C_\alpha &\geq& \log \dashint_X \exp(-\alpha(\varphi-\sup\varphi)) \exp f \omega^n \\
 &=& \log \dashint_X \exp(-\alpha(\varphi-\sup\varphi)-\log\frac{\omega_\varphi^n}{\omega^n} + f) \omega_\varphi^n \\
 &\geq& -\dashint_X \log\frac{\omega_\varphi^n}{\omega^n} \omega_\varphi^n + \alpha (\sup\varphi-\dashint_X \varphi \omega_\varphi^n) \\
 &\geq& -\dashint_X \log\frac{\omega_\varphi^n}{\omega^n} \omega_\varphi^n + \alpha I_\omega(\varphi) .
\end{eqnarray*}

So we have
\begin{eqnarray*}
\mathcal{M}_\omega(\varphi) &\geq& - (I_\omega-J_\omega)(\varphi)+\alpha I_\omega(\varphi)-C'_\alpha \\
&\geq& (\alpha-\frac{n}{n+1}) I_\omega(\varphi) - C''_\alpha ,
\end{eqnarray*}
which proves that the Mabuchi functional is proper. The conclusion follows by Theorem \ref{thm:mabuchi-proper-ke}.
\end{proof}

\begin{remark}
The same result holds by considering the action of a finite group $G$: one can define an invariant $\alpha_G$, and if $\alpha_G> \frac{n}{n+1}$, then there exists a $G$-invariant K\"ahler-Einstein metric \cite[Theorem 4.1]{tian-Inv87}.
\end{remark}

\begin{remark}
If $\alpha=\frac{n}{n+1}$, then Fujita \cite{fujita-arxiv} proved $K$-stability. The existence of a KE metrics follows  by  the crucial result of \cite{chen-donaldson-sun} we will describe in the last lecture.
\end{remark}

\begin{remark}
Demailly (see appendix of \cite{cheltsov-shramov}) identified $\alpha(X)=\mathrm{glct}(X)$ as the {\em global log canonical threshold}. It is defined as
\begin{eqnarray*}
 \mathrm{glct}(X) &:=& \sup \{ \lambda>0 \;:\; (X,\lambda D) \text{ is }\log\text{-canonical}\\
 && \text{for all } D \text{ effective } \mathbb{Q}\text{-divisor with } D'\sim_{\mathbb Q} - K_X, \\
 && \text{\itshape i.e. } D\in |-mK_X| \text{ such that } D=\frac{1}{m} D' \}
 \end{eqnarray*}
 The point is that, for $D=s_D^{-1}(0)\in |-mK_X|$, then $\eta=\frac{1}{m}\log|s_D|^2$ is a singular metric in $c_1(X)$ that can be used (by Demailly's approximation) to  estimate the alpha invariant.
 
 For example, Chetsov \cite{cheltsov} computed the $\alpha$-invariant for Del Pezzo surfaces, {\itshape i.e.} two dimensional Fanos, of degree $c_1^2(X)=1$ (which have small anticanonical linear system): if there are no anti-canonical curves with cusps, then $\alpha(X)=1$; if there are cusps, then $\alpha(X)=\frac{5}{6}$, which still satisfies the criterion for existence of KE metrics.
\end{remark}

\section{Towards K-stability}
We are aimed at the Yau-Tian-Donaldson conjecture on existence of K\"ahler-Einstein metrics in the Fano case Theorem \ref{CDS}, which gives equivalence between the existence of KE metrics on Fano manifolds and the algebraic geometric notion of \emph{K-polystability}.

Now some reasons why we should expect something as Theorem \ref{CDS}.

\begin{itemize}
	\item \emph{Historical:} The problem is the analogous in the case of varieties of the problem of equipping vector bundles with Hermitian-Einstein metrics. This is addressed by the Hitchin-Kobayashi correspondence, which indeed show that the existence of such special metric on the vector bundles is equivalent to stability in the sense of Mumford \cite{donaldson-Duke87, uhlenbeck-yau, li-yau, lubke-teleman}. In a certain sense, the Yau-Tian-Donaldson conjecture is a "more non-linear" version of the Hitchin-Kobayashi correspondence.
	\item \emph{Moduli:} From a more heuristic point of view, recall the following fact concerning \emph{jump of complex structures} for Fano manifolds.
	Let $\mathcal{X}\stackrel{\pi}{\to}\Delta$ be a holomorphic submersion. There are cases in which $X_t=\pi^{-1}(t)$ are all smooth Fano, with  $X_t\cong X_{t'}$ for any $t,t'\neq0$, but $X_0\ncong X_t$.  It follows that the moduli space of complex structures cannot be Hausdorff, since there would be non-closed points that cannot be separated: $[X_0]\in \overline{[X_t]}$. But the moduli space of Einstein \emph{metrics} are Hausdorff. This suggests that by "removing" non KE Fanos we could hopefully  get a nice (algebraic) moduli space of Fano manifolds, with topology compatible with distance (Gromov-Hausdorff, see next \ref{sec:GH}) induced by the KE metrics. Algebro-geometric notion of stabilities are good for constructing separated moduli in algebraic geometry.
\end{itemize}

Now we move to some more mathematical reasons which should motivate why the notion of K-stability is a natural one.

\subsection{Variational point of view}\label{sec:K-stab-variational-pov}
The space $\mathcal{K}_{\omega}$ of K\"ahler potentials (say on a Fano, even if this description is more general for the cscK problem) looks like an infinite dimensional negative curvature space in the Mabuchi metric:  $K_\varphi(\phi_1,\phi_2){\phi_3}=-\{\{\phi_1,\phi_2\}_{\varphi}, \phi_3\}_{\varphi}$, where $\{,\}_{\varphi}$ is the Poisson's bracket for $\omega_{\varphi}$. As we have seen in section \ref{sec:mabuchi} such space admits geodesics for which the Mabuchi energy is convex. Note the following: take $f$ such that $\mathrm{grad}\,f$ is holomorphic. Denote by $\psi_t$ the flux of $\mathrm{grad}\,f$: then $\psi_t^*\omega=\omega+\sqrt{-1}\partial\overline\partial\varphi_t$ where $\dot\varphi_t=f$. Essentially by definition, the Futaki invariant is the derivative of the Mabuchi's energy:
$$ \frac{d}{dt} \mathcal{M}_\omega(\varphi_t)=F((\mathrm{grad}\, f)^{1,0}). $$
Thus we can make the following "speculation" (since we are on an infinite dimensional space, so care must be considered...): since the Mabuchi's energy is convex on $\mathcal{K}_{\omega}$ and its critical points are precisely the KE metrics, if at the "infinity" of $\mathcal{K}_{\omega}$, $\mathcal{M}_\omega$ has positive slope, we can suspect that a critical point actually do exists. In order to check the "slope at infinity" we could look for something like:
$$  \lim_{t\to+\infty} \frac{d\mathcal{M}_\omega(\varphi_t)}{dt}=: \tilde{\mathrm{DF}}((\varphi_t)_t)$$
where $\tilde{\mathrm{DF}}$ should be a kind of  "Futaki invariant".

Thus, we would have the following "geodesic stability picture": if we know that $\tilde{\mathrm{DF}}>0$ for all geodesic rays probing infinity, then we expect to have a critical point: in this case, we would say that $X$ is {\em stable}, and hopefully there exists some minimum, {\itshape i.e.}, a K\"ahler-Einstein metric. If $\tilde{\mathrm{DF}}\geq0$, we would say that $X$ is {\em semistable}. The case in which we can find rays such that  $\tilde{\mathrm{DF}}<0$ is  {\em unstable} (and there is no minimum, {\itshape i.e.}, a KE metric should  not exist). 

The precise mathematical definition of $\tilde{\mathrm{DF}}$ will be given for special (algebraic) rays: \emph{test configurations}. In such case, the limits at infinity are singular algebraic varieties with a $\C^*$ action, and $\tilde{\mathrm{DF}}$ will be precisely equal to the Futaki invariant on the singular limits at infinity if such limits are not too singular (actually this was  used in the first definition of K-stability proposed by Tian \cite{tian-Inv}). Moreover, as we will see, such Futaki invariant admits a purely algebro geometric definition  which extends to very singular limits (after Donaldson \cite{donaldson-JDG}). It is worth noting that, thanks to deep algebro-geometric results, to test stability is sufficient to consider only mildly singular degenerate limits for which the original Tian's definitition works \cite{li-xu}.

Thus we should expect that if we have a KE metric, it should be relatively "easier" to show that on such algebraic geodesic rays we have positivity of the $DF$ invariant (this is the content of Berman's result we will discuss in Section $6$). The other direction seems to be harder: we are imposing a condition only on certain algebraic directions in $\mathcal{K}_{\omega}$, which are much less than all possible infinities. The fact that this will be sufficient is deep. The proof of CDS, as we will see, is essentially arguing by contradiction by constructing destabilizing (algebraic) test configurations, but the actual way to do this is quite involved.  Anyway, notice that a more pure variational viewpoint has been recently proved to be also successful to show existence for K-stable Fano manifolds, at least in the case of  no holomorphic vector fields, see Berman-Boucksom-Jonsson \cite{berman-boucksom-jonsson}.

\subsection{Moment map picture for cscK}
Let $(M^{2n},\omega)$ be a compact symplectic manifold.
When $\pi_1(X)=\{1\}$, then $\mathcal{L}_X\omega=0$ yields that there exists $h$ such that $dh=-\iota_X\omega$.
Let $G$ be a compact group in $\mathrm{Sympl}(M,\omega)$. A moment map for the action of $G$ is given by
$$ \mu\colon M \to \mathfrak{g}^\vee $$
a $G$-equivariant map such that, for any $v\in \mathfrak{g}$,
$$ d \langle \mu, v \rangle = -\omega(\rho(v),\_) , $$
where $\rho$ denotes the derivative of the action.
The {\em symplectic quotient} is then defined as
$$ M /^{\text{sym}} G := \mu^{-1}(0)/G . $$

For example, $\C^n/^{\text{sym}} U(1)=\mu^{-1}(0)/U(1) \cong S^{2n-1}/S^1\cong \C \mathbb{P}^{n-1}$ with $\mu(z)=|z|^2-1$.

Consider now $(X^n,\omega)$ a compact K\"ahler manifold. We focus on $\omega$ the symplectic structure on the underlying smooth manifold $M^{2n}$. Consider the infinite-dimensional K\"ahler manifold
$$ \mathcal{I}_\omega = \{ \omega\text{-compatible complex structures on } M \} . $$
Its tangent space is
$$ T_J\mathcal{I}_\omega = \{ A \in \mathrm{End}(TM) \;:\; AJ+ JA=0, \quad \omega(Ax,y)+\omega(x,Ay)=0\} . $$
There is a tautological complex structure
$$ \mathcal{J}_J A = JA .$$
There is a natural $L^2$-metric:
$$ \langle A, B \rangle_J = \int_M g_J(A,B)\frac{\omega^n}{n!} , $$
where $g_J=\omega(\_,J\_)$.
Then
$$ (\mathcal{I}_\omega, \mathcal{J}, \langle\_\vert\_\rangle) $$
is K\"ahler.
(See also {\itshape e.g.} \cite[Chapter 4]{tian-book}, \cite[Section 4.1]{gabor-book}.)

Consider the action of $\mathcal{G}_\omega$ the group of (Hamiltonian) symplectomorphisms. Its tangent space can be identified as $T\mathcal{G}_\omega\simeq \mathcal{C}^\infty_0(M;\mathbb R)$ the space of zero-average smooth functions.
Then:

\begin{theorem}[{Donaldson \cite{donaldson-Fields}, Fujiki \cite{fujiki}}]
	Let $(M^{2n},\omega)$ be a  K\"ahler manifold, and consider the action $\mathcal{G}_\omega$ of (Hamiltonian) symplectomorphism on the space $\mathcal{I}_\omega$ of $\omega$-compatible complex structures. Then the action $\mathcal{G}_\omega \circlearrowleft \mathcal{I}_\omega$ is symplectic, with moment map
	$$ \mu \colon \mathcal{I}_\omega \to \mathcal{C}^\infty_0(M;\mathbb R) , \qquad J \mapsto \mathrm{Sc}_{g_J}-\hat S $$
	where $g_J:=\omega(\_,J\_)$ and $\hat S$ is a cohomological invariant.
\end{theorem}

This means we have two operators $R:\mathcal{C}^\infty_0(M;\mathbb R) \rightarrow T_J \mathcal{I}_{\omega}$ defined as $h \mapsto \mathcal{L}_{X_h}J$ (the infinitesimal action) and 
$S: T_J \mathcal{I}_{\omega}  \rightarrow  \mathcal{C}^\infty_0(M;\mathbb R)$ defined as $D(A)=DSc_{J}(A)$ (the derivative of the moment map), such that $(S(A),h)_{L^2}=-\Omega(R(h),A):=<JA,R(A)>$. That is, $S^*=-JR$.

Then the symplectic quotient $\mu^{-1}(0)/\mathcal{G}_\omega$ would be the "moduli space of cscK metrics in $[\omega]$". Let's us elaborate a bit more about this. If $F$ is a symplectomorphism, then $F^*\omega=\omega$, and $g_{F^{*}J}=F^{*}g$, {\itshape i.e.} nothing interesting is really happing about the metric (the tensor changes, but the underlying  metric spaces are still isometric). To do something metrically non trivial we should look at "the complexification of the symplectomorphisms". While such group does not exist, at least infinitesimally is clear how it should act: $$R(i h):=JR(h)=J\mathcal{L}_{X_h}J=\mathcal{L}_{JX_h}J.$$  Of course this  does not preserve the symplectic form, but instead, as we computed previously: $\mathcal{L}_{JX_h}\omega=2\sqrt{-1}\partial \bar \partial h,$ Thus it acts precisely by deforming within the K\"ahler class! So looking at zero of the moment map within a complexified orbit corresponds to the problem of searching for a KE metric (or more generally, for a csck metric) in the space of K\"ahler potentials as we did in previous sections. From this point of view, the Mabuchi functional should be seen as a convex norm in such orbit (see discussion of Kempf-Ness theorem in the next section), whose critical points are the zeros of the moment map.  Of course, if we find a $J$ within the complexified orbit which is inside the zero of the moment map ({\itshape i.e.} a cscK metric), we can apply Moser's theorem \cite{moser, cannasdasilva} to find a diffeomorphism which pull back the new symplectic K\"ahler form to the fixed symplectic background one (while keeping the abstract biholomorphism type of the complex manifold unchanged). Thus the complexified orbits for which we can find such points should be "special" (read "stable"), see previous discussion on "geodesic stability".

\begin{remark}
Note that smooth hypersurfaces in $\C\mathbb{P}^n$ are all symplectomorphic (for the natural symplectic structure induced by the Fubini-Study form), and the existence problem of KE metric on them fits the above picture.
\end{remark}

The next section is going to describe the classical  Geometric Invariant Theory (GIT) notion of stability. This is crucial to motivate algebraically  the notion of K-stability (that is, the rigorous notion of stability needed to make sense of the vague "stability" notion we described so far).

\subsection{Introduction to GIT}

For a gentle introduction see \cite{Thomas:intro GIT}. More advanced one are  \cite{Newsted_Tata_mod} and \cite{mumford}. The standard complete  reference is \cite{mumford-fogarty-kirwan}. The following example contains the main ideas of GIT:

\begin{example}\label{exa:basic-exa-GIT}
	Consider the action $\mathbb C^* \circlearrowleft \mathbb C^2$ given by
	$$ t \mapsto \left(\begin{matrix} t & 0 \\ 0 & t^{-1} \end{matrix}\right) . $$
	Note that the topological quotient $\mathbb C^2 /^{\text{top}} \mathbb C^*$ is not Hausdorff. From the algebro-geometric point of view, we have
	$$ \mathbb C^2 = \mathrm{Spec}\, \mathbb C[x,y] . $$
	We know what the functions on "$\mathbb C^2 / \mathbb C^*$" should be: simply look at the $\mathbb C^*$-invariant polynomials, that is, $\mathbb C[x,y]^{\mathbb C^*}=\mathbb C[xy]=\mathbb C[t]$. Then
	$$ \mathbb C^2 /\!\!/ \mathbb C^* = \mathrm{Spec}\, \mathbb C[x,y]^{\mathbb C^*} = \mathrm{Spec}\, \mathbb C[t] = \mathbb C . $$
	We have the map
	$$ \mathbb C^2 \to \mathbb C, \qquad (x,y) \mapsto xy . $$
	Substantially, you use invariants to define quotients.
\end{example}

More generally: take $G$ reductive acting as $G \circlearrowleft \mathrm{Spec}\, R=X$ on the algebraic variety $X$, where $R$ is a finitely generated graded $\mathbb C$-algebra. Then define
$$ X /\!\!/ G := \mathrm{Spec}\, R^G $$
where $R^G$ is finitely generated thanks to reductivity.

Consider now compact quotients.

\begin{definition}
	Let $G \subseteq \mathrm{SL}(N+1)$ act as $G \circlearrowleft X \subseteq \mathbb{CP}^N$, and assume that $G$ is reductive.
	Take $x\in X$, and let $\hat x$ be a lift to $\mathbb C^{N+1}$. We say that $x\in X$ is:
	\begin{itemize}
		\item {\em GIT-semistable} if $0\not\in \overline{G \cdot \hat x}$;
		\item {\em GIT-stable} if $G \cdot \hat x$ is closed and its stabilizer is finite;
		\item {\em GIT-polystable} if $G\cdot \hat x$ is closed.
	\end{itemize}
\end{definition}

This is the fundamental result of GIT:

\begin{theorem}[{\cite{mumford}}]
	Let $G \circlearrowleft X$ be reductive. Define the compact space by taking invariant \emph{sections}: 
	$$ Y := \mathrm{Proj}\, \bigoplus_k H^0(X;\mathcal{O}_X(k))^G . $$
	Then there exists
	$$ \Phi \colon X \supset X^{ss} \to Y =: X /\!\!/ G $$
	surjective such that, for $x,y \in X^{ss}$ the set of GIT-semistable points:
	$$ \Phi(x)=\Phi(y) \quad \text{ if and only if } \quad \overline{G \cdot x} \cap \overline{G \cdot y} \cap X^{ss} \neq \varnothing ; $$
	and, for $x,y \in X^s$ the set of GIT-stable points:
	$$ \Phi(x)=\Phi(y) \quad \text{ if and only if } \quad x=g \cdot y \text{ for some } g \in G . $$
\end{theorem}

There is only one closed orbit in semi-stable-equivalence classes, and it is the polystable orbit. Thus,
$ X^s / G \subseteq X^{ss} // G $
is a compactification of the moduli space of stable points, for which we can think at $X^{ss} // G$,  set-theoretically, as the moduli space of polystable objects. In Example \ref{exa:basic-exa-GIT}, the origin has to be considered as a polystable point. The points in the axes as semistable ones. The points in the hyperbolas $xy=s\neq0$ as stable.

We introduce now the {\em Hilbert-Mumford criterion}.
Let $G \circlearrowleft X \subseteq \mathbb{CP}^N$ reductive, and consider a $1$-parameter subgroup $\lambda \colon \mathbb C^* \to G$ of $G$. Take a point $x\in X$ and consider $\lim_{t \to 0} \lambda(t)\cdot x = x_0\in X$. Look at the complex line $\mathbb C_{x_0}$ over $x_0$. There is a $\mathbb C^*$-action on it: $\lambda(t) \circlearrowleft \mathbb C_{x_0}$ as some power $t^{-w}$, where $w$ is called the {\em weight}. Define
$$ \mu(x,\lambda):= w . $$

\begin{theorem}[{Hilbert-Mumford criterion \cite{mumford}}]
	Let $G \circlearrowleft X$ be reductive, and $x\in X$. Then:
	\begin{itemize}
		\item $x$ is GIT-stable if and only if $\mu(x,\lambda)>0$ for any $\lambda$ one-parameter subgroup.
		\item $x$ is GIT-semistable if and only if $\mu(x,\lambda)\geq0$ for any $\lambda$ one-parameter subgroup.
	\end{itemize}
\end{theorem}

The implication "$x$ stable implies $\mu(x,\lambda)>0$ for any $\lambda$" follows directly by "pictures". The implication "$\mu(x,\lambda)\geq0$ for any $\lambda$ implies $x$ semistable" is hard; see \cite[Theorem 2.1]{mumford-fogarty-kirwan}.

\begin{example}
	For example, for $\mathrm{SL}(2;\mathbb C)$ acting on binary forms identified with $\mathrm{Sym}^2(\mathbb C^2)$, GIT-stability is related to polynomials with multiple roots. This is a nice exercise to do in order to practice with GIT stability (use Hilbert-Mumford criterion).
\end{example}

We are now going to quickly review the \emph{Kempf-Ness theorem}, which relate finite dimensional symplectic quotient to GIT. Moreover, we discuss the relation with the infinite dimensional picture of the scalar curvature as a moment map.

Look at $\mathbb C^* \circlearrowleft (\mathbb C^2, \omega_{\text{std}})$, with $\mathrm{U}(1) \subset \mathbb C^*$ Hamitonian  action ($h=|x|^2 - |y|^2$). We clearly get $$\mathbb C^2 /^{\text{sym}} \mathrm{U}(1) \stackrel{\text{top}}{\simeq} \mathbb C^2 /\!\!/ \mathbb C^*,$$
{\itshape i.e.} the GIT quotient and the symplectic quotient are identified!

More generally, the Kempf-Ness theorem states that if we have some compact  $K\subseteq\mathrm{SU}(N+1)$  with $K^{\mathbb C}\subseteq\mathrm{SL}(N+1)$,  and  $K$ acting symplectically on $(X, \omega_{FS})$ with moment map $\mu(Z)=\sqrt{-1} \frac{Z_j\bar Z_j}{|Z|^2}$. Then
$$X/\!\!/K^{\mathbb C}\cong X/^{\text{sym}}K .$$
The  idea of the proof consists in studying the \emph{log norm} function $m=\log|g\cdot \hat{x}|^2$ on $K^{\C}/K$. The  crucial point is that such function is \emph{convex} on the complexified orbits, and its critical points coincide with the zero of the moment map. Thus, if the slope at infinity of $m$  is positive (which holds  if the orbit is stable, since the weights at infinity are positive), we can find a critical point (that is, a zero of the moment map).

Let us go back to the infinite dimensional case. In such situation the log norm corresponds to  the Mabuchi energy on the "complexified orbits" of the symplectomorphisms. Thus the complexified orbits which intersect the zero of the moment map (cscK) should be "stable".  Thus, in this infinite dimensional setting, we should still have:
$$ `` \mu^{-1}(0)/\mathcal{G}_{\omega}\cong \mathcal{I}_{\omega}^{\text{stab}}/\mathcal{G}_{\omega}^{\C}=: \mathcal{I}_{\omega} /\!\!/\mathcal{G}_{\omega}^{\C} " .$$
In the KE case, the rigorous notion of stability condition is going to be, of course, K-stability.  $``\mathcal{I}_{\omega} /\!\!/\mathcal{G}_{\omega}^{\C}"$ should be a purely \emph{algebraic}  moduli spaces of stable varieties (K-moduli), which could also be rigorously constructed in the Fano case ({\itshape e.g.} see the survey \cite{Spotti_Cortona_moduli}).

\begin{remark}
	This moduli picture can be thought as an higher dimensional analogous of the identification between moduli spaces of metric with constant Gauss curvature on surface of genus $g$ and the complex Deligne-Mumford moduli space of curves.
\end{remark}

We are now ready to finally define $K$-stabilty.

\subsection{Definition of K-stability}

For $(X,K_X^{-1})$ Fano manifold, we want to define a "GIT-like" notion of stability, following the dictionary:
\begin{itemize}
	\item $1$-parameter subgroups $\lambda$ correspond to test configurations $\mathcal{X}\to \mathbb C$;
	\item weight $\mu(x,\lambda)$ corresponds to Donaldson-Futaki invariant $DF(X,\mathcal{X})$.
\end{itemize}

In analogy with the Hilbert-Mumford criterion, we define:
\begin{itemize}
	\item $X$ is {\em K-stable} if $DF(X,\mathcal{X})>0$ for any test-configuration $\mathcal{X}$;
	\item $X$ is {\em K-semistable} if $DF(X,\mathcal{X})\geq0$ for any test-configuration $\mathcal{X}$; 
	\item $X$ is {\em K-polystable} if $X$ is K-semistable and $DF(X,\mathcal{X})=0$ if and only if $\mathcal{X}=X \times\mathbb{C}$.
\end{itemize}

So let us define what is a test configuration and what is the Donaldson-Futaki invariant.

A {\em test configuration} of exponent $\ell\in\mathbb N$ is a flat family
$$
\xymatrix{
	\mathcal{L} \ar[d]_{\mathbb C^* \circlearrowleft} & \\
	\mathcal{X} \ar[r]_{\pi} & \mathbb{C}
} $$
such that, for $X_1=\pi^{-1}(1)$:
\begin{enumerate}
	\item $\mathcal{L}$ is relative ample and $\mathcal{L}\lfloor_{X_1}\simeq\ K_{X_1}^{-\ell}$;
	\item $\pi$ is $\mathbb C^*$-equivariant.
\end{enumerate}

Although a priori quite abstract, test configurations are concrete:
\begin{theorem}[{Ross-Thomas \cite{ross-thomas}}]
	All test configurations are realized by $1$-parameters subgroups in $PGL(N+1)$ for embedding of $X \hookrightarrow \mathbb{CP}^N$.
\end{theorem}

\begin{remark}
	Note that $N$ in the above result is not fixed! In principle, we should consider test configurations arising as $N\rightarrow +\infty$. This is why K-stability is very hard to check! K-stability should be considered as a  (new, non-GIT) stability condition defined on the "stack" of Fano varieties, where the $DF$ invariant can be seen as the weigh of a stacky line bundle, the so-called CM line \cite{paul-tian}. This fancier approach is relevant for precise moduli considerations.
\end{remark}

The next  is an example of a test configuration for $\C \mathbb{P}^1$.

\begin{example}\label{ex:Kstable}
	For example, consider $X=\{xy=z^2\}$ obtained as the image of the Veronese embedding $\mathbb{CP}^1\ni[x:y] \mapsto[x^2:y^2:xy]\in\mathbb{CP}^2$. Conider the $\mathbb C^*$-action
	$$ \lambda(t) = \left(\begin{matrix}
	t & & \\
	& 1 & \\
	& & t^{-1}
	\end{matrix}\right) , $$
	and define
	$$ X_t := \lambda(t)\cdot X = \{xy=tz^2\} . $$
	As $t\to0$, we get $X_0=\{xy=0\}$.
\end{example}

Now we introduce  the {\em Donaldson-Futaki invariant}. Consider a test configuration. We get an action $\mathbb C^* \circlearrowleft X_0$ inducing an action $\mathbb C^* \circlearrowleft H^0(X_0;\mathcal{L}\lfloor_{X_0}^{\otimes k})$ for any $k$. By the Riemann–Roch theorem, we write the Hilbert polynomial
$$ h_k:=\dim H^0(X_1;\mathcal{L}\lfloor_{X_1}^{\otimes k}) = a_0 k^n+a_1 k^{n-1}+\mathcal{O}(k^{n-2}) . $$
Similarly, we write the weight for the action $\mathbb C^* \circlearrowleft \wedge^{\text{top}}H^0(X_0;\mathcal{L}_0^{\otimes k})$ as:
$$ w_k = b_0 k^{n+1} + b_1 k^n + \mathcal{O}(k^{n-1}). $$
Set the Donaldson-Fujiki invariant as
$$ DF(X,\mathcal{X}) :=\frac{a_1b_0-a_0b_1}{a_0} . $$

Let us compute the $DF$ invariant for the previous Example \ref{ex:Kstable} of test configuration.

\begin{example}
	
	Recall that $X_0=\mathrm{Proj}\,\mathbb{C}[x,y,z]/\langle xy \rangle $.
	We have
	\begin{eqnarray*}
		H^0(\mathcal{O}_{X_0}(1)) &=& \langle x,y,z \rangle ,\\
		H^0(\mathcal{O}_{X_0}(2)) &=& \langle x^2, xz, y^2, yz, z^2 \rangle ,\\
		H^0(\mathcal{O}_{X_0}(3)) &=& \langle x^3, x^2z, xz^2, y^3, y^2z, yz^2, z^3 \rangle ,
	\end{eqnarray*}
	and so on.
	In general,
	\begin{eqnarray*}
		h_k&:=&\dim\{\text{polynomials of degree }k\}-\dim\{\text{polynomials of degree }k-2\} \\
		&=& 2k+1 .
	\end{eqnarray*}
	Similar computations for the weight give:
	$$ w_k=\frac{k(k-1)}{2} . $$
	At the end, we get:
	$$ DF(\mathbb{CP}^1, \mathcal{X}) = \left(\frac{1}{2}+2\cdot\frac{1}{2}\right)\cdot\frac{1}{2}=\frac{3}{4}>0 , $$
	as it should be, since $X$ is $K$-stable because it is KE.
\end{example}

Finally, let us see that at least for smooth central fiber, the $DF$ invariant recovers the classical Futaki invariant.

\begin{proposition}\label{DF=F}
	If the central fiber $X_0$ is smooth, then the algebro-geometric Futaki invariant $F$ and the $DF$  invariant are essentially the same: if $Z$ generates the $\C^\ast$-action 
	$$ DF(X_0, \mathcal{X})=\frac{1}{4\pi} F(Z) . $$
\end{proposition}

\begin{proof}
	For more details on this approach via {\em K\"ahler quantization}, see {\itshape e.g.} \cite[Proposition 7.15]{gabor-book}.
	
	Take $L\to (X_0,\omega)$ be a K\"ahler manifold with $L$ ample. Consider $H^0(X_0;L^{\otimes k})$ and take $\underline{s}^k=(s_0^k, \ldots, s_{N_k}^k)$ an orthonormal basis of holomorphic sections with respect to the $L^2$-metric induced by $\omega$.
	Denote by $h$ the Hermitian structure on $L$ with positive curvature $\omega$.
	Introduce the {\em Bergman kernel}
	$$ b_h(p)=\sum_{j=0}^{N_k} |s_j^k(p)|^2_h >0 , $$
	where $N_k+1=\dim H^0(X_0;L^{\otimes k})$.
	It is easy to see that does not depend on the choice of the orthonormal basis.
	
	Since $L$ is ample,
	$$ \varphi_{\underline{s}^k} \colon X_0 \hookrightarrow \mathbb{CP}^{N_k}, \qquad p \mapsto [s_0^k(p) : \cdots : s_{N_k}^k(p)] $$
	is well-defined for $k\gg1$, and we compute
	$$ \varphi^*_{\underline{s}^k} \omega_{FS} = k\omega + \sqrt{-1} \partial\overline\partial \log b_{h^n}.$$

	The important (and not trivial!) fact proven by Tian \cite{tian-JDG90}, Ruan \cite{ruan}, Zelditch \cite{zelditch}, Lu \cite{lu}, Catlin \cite{catlin} is that
	$$ b_{h^k}=1+\frac{\mathrm{Sc}_\omega}{4\pi}\cdot\frac{1}{k}+\mathcal{O}\left(\frac{1}{k^2}\right) $$
	As a consequence,
	$$ \omega = \frac{1}{2\pi k}\varphi^*_k\omega_{FS}+\mathcal{O}\left(\frac{1}{k^2}\right) , $$
	and  we recover the Hilbert polynomial expression:
	$$ h_k = k^n \int_{X_0} \frac{\omega^n}{n!} + \frac{k^{n-1}}{4\pi}\int_{X_0} \mathrm{Sc}_\omega \frac{\omega^n}{n!} + \mathcal{O}(k^{n-2}).$$
	
	This gives
	$$ a_0 = \int_{X_0} \frac{\omega^n}{n!}, \qquad a_1 = \frac{1}{4\pi} \int_{X_0} \mathrm{Sc}_\omega \frac{\omega^n}{n!} $$
	
	Let us do similar computation for the weights. We assume that the $\mathbb C^*$ action comes from complexification of a $\mathrm{U}(1)$-action for the embedding:
	$$
	\xymatrix{
		\mathbb C^* \ar@{^{(}->}[r] & \mathrm{SL}(n+1) \ar[r] & \mathrm{Aut}(X_0 \subseteq \mathbb{CP}^{n+1}) \\
		\mathrm{U}(1) \ar@{^{(}->}[u] \ar@{^{(}->}[r] & \mathrm{U}(n+1) . \ar@{^{(}->}[u] & 
	}
	$$
	Recall that the moment map for $\omega_{FS}$ and the $U(n+1)$-action is $\mu(Z)=\frac{\sqrt{-1}Z_i\bar Z_j}{|Z|^2}$. Let $f$ Hamiltonian for $\omega$, and  $f_k=-2\pi (A_k)_{i,j} \frac{(s_i^k,s_j^k)_{h^k}}{b_h^k}$ is Hamiltonian for $\omega_{FS|k}$. Then
	$$ f-\frac{1}{2\pi k}f_k=\mathcal{O}(k^{-2}) . $$
	
	We compute
	\begin{eqnarray*}
		\int_{X_0} f b_{h^k} \frac{\omega^n}{n!} &=& \frac{1}{2\pi k} \int_{X_0} f_k b_{h^k} \frac{(k\omega)^n}{k^nn!}+\mathcal{O}(k^{-2}) \\
		&=& \frac{1}{2\pi k} \int_{X_0} f \left(1+\frac{\mathrm{Sc}_\omega}{4\pi} \frac{1}{k}+ \mathcal{O}(k^{-2})\right)\frac{\omega^n}{n!} .
	\end{eqnarray*}
	Then
	$$ w_k = -k^{n+1} \left( \int_{X_0} f \frac{\omega^n}{n!} + \frac{1}{4\pi k} \int_{X_0} \mathrm{Sc}_\omega\omega^n + \mathcal{O}(k^{-2})\right) , $$
	whence we get
	$$ b_0=-\int_{X_0} f \frac{\omega^n}{n!}, \qquad b_1=-\frac{1}{4\pi k} \int_{X_0} \mathrm{Sc}_\omega\omega^n . $$
	By plugging-in these expressions in $DF$, we get the statement.
\end{proof}

\begin{remark}
	\begin{itemize}
		\item	In the Fano case, the same holds for mildly singular (klt) central fiber $X_0$.
		\item By analyzing in deeper detail the K\"ahler quantization picture, Donaldson proved  \cite{donaldson-JDG01} that cscK manifolds with no holomorphic vector fields are asymptotically balanced, condition which coincides with certain asymptotic GIT stability (Chow stability). However, there are examples which show that asymptotic Chow stability is not enough to capture the cscK or KE existence problem in general.
		\item With similar technique, Donaldson showed \cite{donaldson-JDG05} a lower bound on the $L^2$ norm of the scalar curvature: $\|\mathrm{Sc}_{\omega}-\overline{\mathrm{Sc}}\|_{L^2}\geq -DF/\|\mathcal{X}\|$. This is clearly enough to conclude K-semistability for the general cscK case (improved to K-polystability by Stoppa in the case with no holomorphic vector fields \cite{stoppa-AdvMath})
	\end{itemize}
\end{remark}

With the above results we have closed the circle between the rough "geodesic stability" and the properly defined algebraic notion of K-stability. Moreover, the infinite moment map picture makes quite clear that it is natural to expect that such stability condition can precisely capture the existence problem. As we will discuss in the next section, this is precisely the case in our Fano situation.

\section{Equivalence between existence of K\"ahler-Einstein and K-stability}

The goal of this last section is to describe some of the ideas in the proof of equivalence between existence of KE metrics on Fano manifolds and K-stability, as well as to shortly describe compact moduli spaces of such manifolds and some explicit examples.

\subsection{Berman's result}

In \cite{berman-Inv} it is proved that K-polystability is a necessary condition for the existence of a KE metrics on smooth (or even mildly singular) Fanos. Previous results in this direction are given by works of Tian, Donaldson, and Stoppa.

We do not discuss in detail Berman's proof, but we simply emphasize the main points. The argument is based on a \emph{crucial formula} for the Donaldson-Futaki invariant:

\begin{equation}\label{CF}
DF(X,\mathcal{X})=\lim_{t\rightarrow \infty}\frac{d}{dt}\mathcal{D}_\omega(\varphi_t)+ \epsilon.
\end{equation}

Let us explain its meaning.  Here $\mathcal{X}\rightarrow \Delta_\tau\subseteq \C^*$ is a test configuration for a Fano variety $X\cong X_1$. Starting from a K\"ahler metric $\omega_\phi=\omega+\sqrt{-1} \partial\overline\partial \phi$ on $X_1$  one can construct a weak geodesic ray in the Mabuchi metric ({\itshape i.e.} a weak solution of the homogeneus Monge-Ampère equation, compare Remark \ref{rmk:donaldson-semmes-chma}) {with boundary datum $\phi$. Let us denote with  $\phi_t:=\rho_\tau^*\psi_\tau$ such geodesic ray emanating from $\phi$, where $t=-\log|\tau|^2$ and $\rho_\tau$ is the $\C^*$-action which identifies the  general fibres. 
	
	$\mathcal{D}_\omega$ denotes the {\em Ding functional} \cite{ding}. The  Ding functional has  the advantage that it can be defined for metric of less regularity with respect to the Mabuchi energy. It has the property that its  critical points are again the KE metrics as for the Mabuchi functional.  It takes the explicit form
	$$ \mathcal{D}_\omega (\varphi) = - \frac{1}{(n+1)!} \sum_{j=0}^{n} \dashint_X \varphi \omega_{\varphi}^j \wedge \omega^{n-j} -  \log \left( \dashint_X \exp (f-\varphi)\omega^n \right),$$
where $f$ is the Ricci potential of $\omega$.

	The Berman's formula thus states that, similarly to what we discussed for the Mabuchi's energy, the Donaldson-Futaki invariant is the slope at infinity in the space of K\"ahler potentials of the Ding functional plus an error term denoted by $\epsilon$ which has the crucial property of being \emph{non-negative}. Actually the term  $\epsilon$ can be explicitly given in terms of purely algebro-geometric data of the test-configuration and its resolutions. 
	
	Thanks to the this formula \eqref{CF}, the idea of proof is now very clear. If we start the geodesic ray $\phi_t$ from the KE potential $\phi_0=\phi^{KE}$, using that the KE equation is the Euler-Lagrange equation of the Ding functional, we have $\frac{d}{dt}\mathcal{D}_\omega(\phi_t)_{|t=0}=0$ (more precisely, Berman proved that is non-negative, the issue being related to regularity properties of the ray). Moreover, since also the Ding functional is convex along the ray, we have that its slope at infinity $\lim_{t \rightarrow \infty}\frac{d}{dt}\mathcal{D}_\omega(\phi_t)$ is still non negative. Hence, by the formula and the positivity of the error term, we find that the Donaldson-Futaki invariant of an arbitrary test configuration is non-negative ({\itshape i.e.} $X$ is K-semistable). It is then possible to analyze in further details what happens in the the case of DF invariant equal to zero, to conclude that the error term vanishes and $X_1$ is biholomorphic to $X_0$ (the point is that one can show that the Ding function has to be affine, and the test configuration generated by an holomorphic vector field).

	\begin{remark}
		\begin{itemize}
			\item The Ding functional can be defined also for mildly singular Fano varieties (more precisely klt Fanos, see later). Thus Berman's formula gives K-polystability also for such very natural class of singular varieties. As we will see, this is the class of singularities which need to be considered for moduli spaces compactification of smooth KE Fano manifolds.
			\item The formula for $\epsilon$ was used to define the notion of \emph{Ding stability}, a notion that is {\itshape a priori} stronger but then equivalent to K-stability. This notion has been relevant for some of the newest developments in the area, for example in the works of K. Fujita ({\itshape e.g.} \cite{fujita-arxiv1508, fujita-arxiv1602}).
		\end{itemize}
	\end{remark}

	\subsection{K-polystability implies existence of KE metrics: Donaldson's cone angle continuity path}
	
	The goal of the next sections is to give a ”map” of the main steps in the proof of the existence of KE metrics on K-polystable Fano manifolds. We hope that such description could be of some help for readers of the seminal papers \cite{chen-donaldson-sun}.
	
	The first idea in the proof is to consider a "specialization" of the Aubin's continuity path 
	$\mathrm{Ric}({\omega_t})=t\omega_t+(1-t)\omega$, see \eqref{eq:ke-fano-continuity}, with $\omega \in 2\pi c_1(X)$, by replacing the smooth background K\"ahler form $\omega$ with a current of integration $\delta_D$ along a smooth divisor $D$ given as zero of a smooth section of the line bundle $K_X^{-1}$. That is, we consider the following path, called \emph{Donaldson's cone angle continuity path}:
	\begin{equation} \label{DP}
	\mathrm{Ric}({\omega_\beta})=\beta \omega +2 \pi(1-\beta)\delta_D,
	\end{equation}
	for $\beta \in (0,1]$.
	
	Note that the metric $\omega_\beta$ is exactly KE with Einstein constant equal to $\beta$ on $X\setminus D$, but the price to pay is that is not longer smooth near $D$: it has so-called conical singularities.  Roughly, this means that near the divisor $D$, if we choose local coordinates $(z_j)_j$ so that $D$ is given by $z_n=0$, the metric looks like (is uniformly equivalent to) the model flat cone (transverse to $D$) metric: 
	$$ \omega_\beta \sim \frac{\sqrt{-1}}{|z|^{2(1-\beta)}}dz_n\wedge d\bar{z}_n+ \sqrt{-1} \sum_{j=1}^{n-1}d z_j\wedge d\bar{z}_j.$$

	Since it is not always possible to find a \emph{smooth} anti-canonical sections, one need to consider plurianti-canonical sections, that is sections of powers $K_X^{-\lambda}$ of the anti-canonical bundle. Then the generic sections are smooth by Bertini's theorem . The value of the Einstein constant in the corresponding conical continuity path is given by $r(\beta)=1-\lambda(1-\beta)$, and it is positive as soon as $\beta> 1-\frac{1}{\lambda}$.\\
	
	The idea is now simply to prove that the set of cone angles $\beta$s, for which we can solve equation \ref{DP} is \emph{non-empty}, \emph{open} and \emph{closed} under the assumption of $K$-polystability of $X$. Then, letting $\beta \rightarrow 1$, the cone angles can  "open-up" to a genuinely smooth KE metric on X (of course there is plenty of things to be proven to make such statement precise!). As a historical comment, let us point out that a similar strategy to deform cone angle Einstein metrics (with cone singularities along smooth \emph{real} curves) to smooth Einstein metric was proposed by Thurston as a way to attack the $3$D Poincar\`e conjecture \cite{thurston}.
	
	The fact that cone angle metrics (with no assumption on the curvature) always exist is  easy: take $\omega=\omega_{0} + \varepsilon  \sqrt{-1} \partial\overline\partial |s_D|^{2\beta}_{h_0}$ with $\omega_0$ smooth background metric (the curvature of the Chern connection of some Hermitian metric $h_0$ on the plurianti-canonical bundle) and $\varepsilon \ll 1$. Of course they do not satisfy our equation \eqref{DP}. Thus non-emptiness of the set of solutions for small cone angles follows by taking $\lambda>0$ and apply the orbifold version of the negative KE problem for $\beta=\frac{1}{m}< 1-\frac{1}{\lambda}$, see \cite[Section 6.2]{demailly-kollar}, \cite[Part III, Section 5]{chen-donaldson-sun}, using the above conical $\omega$ as background. If $\lambda=1$ one could instead use a result of Berman \cite{berman-Adv} which says that for sufficiently small $\beta$s, a $\log$ version of the alpha invariant criterion ensures existence of conical  KE metrics along $D$ of positive Einstein constant.

	Even if not needed in the proof, we would like to make some  remarks about the {\itshape behaviour} of the conical KE metrics for \emph{small} cone angles.
	
	\begin{remark}
		For  $\lambda>1$ the conical (negative) KE metrics converge as $\beta \rightarrow 0$ to complete cuspidal KE metrics (with hyperbolic cuspidal  behaviour transverse to $D$) as proven in \cite{guenancia}.

		For  $\lambda=1$, an interesting well-known  conjecture states  that the KE cone metrics converge  to the complete Calabi-Yau metrics on $X\setminus D$ constructed by Tian and Yau in \cite{tian-yau-1, tian-yau-2}.
		
		If the anti-canonical bundle admits roots, we can consider also $\lambda< 1$. However, in such case small angle metrics do not exist: it is a ($\log$)-K-unstable regime, as proved in \cite{li-sun}. Heuristically  it is clear: we would expect convergence to complete metrics with strictly positive Ricci tensor, and this contradicts Myers' theorem.
	\end{remark}

	Openness is definitely more sophisticated. The idea is to use some \emph{implicit function theorem}, combined with the fact that there are no holomorphic vector fields tangential to $D$ \cite[Section 4.4]{donaldson-Essays} (which would be related to infinitesimal isometries as we have discussed in section \ref{sec:matsushima}, and hence would give obstructions to the linearization of the problem). However, using naive (weighted) H\"older spaces will not work. The actual function spaces to be consider are given by a variation of usual H\"older spaces where one considers norms of only certain second order derivatives \cite[Section 4.3]{donaldson-Essays} (essentially the one which appear in the complex Hessian). In any case, in \cite[Theorem 2]{donaldson-Essays} it is proved that if we assume existence of  conical KE metrics, we can move the cone angle a bit and still find conical KE metrics. 
	
	\begin{remark} Detailed study of conical KE metrics (generalized energy functionals, alpha invariant, asymptotic analysis, etc\dots) can be found in the works of Jeffres, Mazzeo, and Rubinstein \cite{jeffres-mazzeo-rubinstein}. 
	\end{remark}
	
	From now on, we are taking for granted that we have found conical KE metric for small cone angles, and that the KE condition is open. Thus we need to understand what happens to conical KE spaces $(X,\omega_{\beta_i})$ as $\beta_i\nearrow \underline{\beta}$. The crucial point is now to establish a \emph{weak geometric  compactness result}, which will ensure that, after eventually passing to a subsequence, we can find a limit space $X_\infty$ with sufficiently good property (via partial regularity results). In particular, if $X_\infty$ can be proved to be an \emph{algebraic} space we can try to use the $K$-stability hypothesis to show that bad behaviours cannot happens and that $X_\infty$ is still biholomorphic to $X$ and the limit metric structure is a cone angle metric on $X$ along $D$ with cone angle $\underline{\beta}$. By the continuity argument we could then show that $\beta$ can go to one, and in this case the metric converge to a smooth KE metric.
	
	The right geometric notion of weak limits we need to consider is given by the so-called \emph{Gromov-Hausdorff topology}. This is a very weak topology on the class of isomorphism class of compact metric spaces, which has the advantage to make precise what one means for convergence of a smooth Riemannian manifold to a (potentially) singular space. Thus, in the next section we will recall some of the basic definitions and properties of such notion.
	
	The next sections can be understood  as a geometric way to do {\itshape a priori} estimates and study regularity for our KE equations.
	
	\subsection{Gromov-Hausdorff basics ("weak limits")}\label{sec:GH}
	Let $(X,d_X)$ and $(Y,d_Y)$ be two compact metric spaces. Define the following \emph{Gromov-Hausdorff} (GH) (pseudo-)distance:
	$$d_{GH}(X,Y)= \inf_{Z\hookleftarrow_{\text{isom}}X,Y} \inf\{\varepsilon \geq 0 \,|\, \mathcal{N}_\varepsilon^Z(X)\supseteq Y, \, \mathcal{N}_\varepsilon^Z(Y)\supseteq X\},$$
	where $\mathcal{N}_\varepsilon^Z(-)$ denotes the $\varepsilon$-neighborhood of a set in $Z$.
	It can be easily shown that $d_{GH}(X,Y)=0$ implies that $X$ and $Y$ are isometric. Thus, the space of compact metric spaces modulo isometries is itself a metric space. Such "universal" space can be used to study degeneration of Riemannian manifolds to singular spaces. A usual way to bound the GH distance to study convergence in such "weak" GH topology is by finding functions $f:X \rightarrow Y$ which are $\varepsilon$-dense and $\varepsilon$-isometries. 
	
	A crucial question is to find criteria of convergence in such GH topology (\emph{pre-compactness}), {\itshape i.e.} given a sequence $(X_i)$ of compact metric spaces, find conditions which guarantee that such sequence subconverges to a compact metric space $X_\infty$.
	
	The fundamental result is the Gromov's pre-compactness theorem. A set $\mathcal{X}$ of compact metric spaces is \emph{uniformally totally bounded} (UTB) if:
	\begin{itemize}
		\item $\exists D>0$ such that $\mathrm{diam}(X)<D$ for all $X \in \mathcal{X}$;
		\item $\forall \varepsilon >0$, $\exists N=N(\varepsilon)$ such that for all $X \in \mathcal{X}$, $\exists S_X$ $\varepsilon$-dense set of cardinality at most $N$.
	\end{itemize}
	\begin{theorem}[Gromov's pre-compactness]
		If $\mathcal{X}$ is UTB, then $\mathcal{X}$ is pre-compact.
	\end{theorem}
	The proof is simple (for more details and for more references on GH topology see \cite{burago-burago-ivanov}). The idea is to uniformly discretize the spaces in the family, apply a diagonalization argument, and take a metric completion to get a compact limit space. More precisely, let $\mathcal{X} \in X_n \supseteq S_n:= S_{n,1} \cup S_{n,\frac{1}{2}}\cup \dots = \{x_{i,n}\}_{i=1}^{\infty}$ where  $S_{n,\frac{1}{j}}$ is a finite  $\frac{1}{j}$-dense set with at most  $N(\frac{1}{j})$ points independent of $n$. Since $d_{ij}^n:=d_{X_n}(x_{i,n}, x_{j,n}) < D$, by diagonalization we can find a subsequence $n_k$ such that $d_{ij}^{n_k} \rightarrow d_{ij}$. Take $S$ a countable set and equip it with the pseudo-distance $d_{ij}$. Then its metric completion after identification between points of zero distance $X_{\infty}=\overline{S/d}$ is a compact metric space which is the GH limit of the sequence $X_{n_k}$.
	
	Thus, given families of  Riemannian manifolds, it becomes interesting to find conditions which will implies that the family is UTB in the GH distance. The idea is that if we have a uniform control of the volume of small balls, we may hope to control uniformly  the number of them we need to cover a manifolds. It is in such consideration that the Ricci curvature enters the game! Recall that in normal coordinates centered at $p$, the Ricci tensor exactly measure the second order deviation of the volume form from the flat Euclidean measure:
	$$dV_g=\left(1-\frac{1}{6}\sum_{i,j} \mathrm{Ric}_{ij}(p)x_ix_j+O(x^3) \right)d\underline{x}.$$
	A global result on the behaviour of volume under Ricci lower bounds is given by the following crucial result. Let $H_c$ be the  $n$-dimensional  simply-connected space of constant sectional curvature equal to $c$. Then:
	\begin{theorem}[Bishop-Gromov monotonicity] If $\mathrm{Ric}(g)\geq(n-1)cg$, for $c \in \R$, then the function $\frac{\mathrm{Vol}(B_g(p,r))}{\mathrm{Vol}(B_{H_c}(r))}$ is non-increasing. 
	\end{theorem}
	See \cite{burago-burago-ivanov} for its proof. Its immediate consequence is:
	\begin{theorem}[Riemannian Gromov's precompactness] The set of $n$-dimensional Riemannian manifolds with $\mathrm{diam}(M,g)\leq D$ and $\mathrm{Ric}(g)\geq c(n-1)g$ is GH pre-compact.
	\end{theorem}
	\begin{proof}
		It is sufficient to check the UTB property. Take $\{x_i\}$ a maximally $\varepsilon$-separated set, {\itshape i.e.} $\{x_i\}$ is $\varepsilon$-dense and $B(x_i,\frac{\varepsilon}{2})$ are disjoint. Then, by monotonicity:
		$$|\{x_i\}|\leq \frac{\mathrm{Vol}(M,g)}{\mathrm{Vol}(B_g(x_i, \frac{\varepsilon}{2}))} \leq  \frac{\mathrm{Vol}(B_{H_c}(D))}{\mathrm{Vol}(B_{H_c}( \frac{\varepsilon}{2}))}.$$
		Since the last term in the above expression depends only on $c$, $D$, and $\varepsilon$, we are done.
	\end{proof}
	
	If in the above theorem $c>0$, then the diameter is automatically bounded thanks to Myers' theorem. In particular, let us note that the set of $n$-dimensional KE Fano manifolds is GH pre-compact, {\itshape i.e.} given any sequence of them we can extract a subsequence converging  to a compact metric space. 
	
	\begin{remark}
		If the diameter bound does not hold, one still have pre-compactness in the \emph{pointed GH sense}:  by choosing a sequence of points $p_i \in M_i$, for all $r>0$, closed $r$-balls centered at $p_i$ sub-GH-converge to closed $r$-balls  and center $p_\infty$ in a metric space $X_\infty$ (in general non-compact).  In particular, this notion of convergence is important  in the analysis of \emph{metric bubbles} and \emph{tangent cones} ---both arising from "unbounded" rescalings (zoomings)--- to study singularities of limits of \emph{compact} Einstein spaces.
	\end{remark}
	
	In our problem, we need to consider limits of \emph{conical} KE metrics. It is still possible to show that one can take limit in the GH topology also in this case (for example, but it is not trivial, it is possible to show that the cone angle can be "smoothed out" to genuine Riemannian metrics with lower bounds on the Ricci curvature and with control on their diameter \cite[part I]{chen-donaldson-sun}). Thus we can assume that our sequence in the Donaldson's continuity path $(X, \omega_{\beta_i})$ subconverge in GH to a compact metric space $X_{\infty}$. This is the "weak limit". Now we need to do some "regularity"!

	\subsection{Cheeger-Colding Theory  ("smooth regularity")}
	In general, Cheeger-Colding(-Tian, for the K\"ahler case) theory, studies regularity of the GH limits $X_\infty$ of sequences of Riemannian manifolds with Ricci bounded below. See survey \cite{cheeger-Fermi}. In our situation we can (essentially) restrict to the Einstein case with volume \emph{non-collapsing} hypothesis $\mathrm{Vol}(B(p,1))>C>0$ (observe that if Ricci is positive, this is a simple consequence of Bishop-Gromov monotonicity). We focus on the absolute case, {\itshape i.e.} no conical manifolds. This case is interesting in itself ({\itshape e.g.}  moduli compactifications in the KE Fano case), and still central in the more technical case of conical KE.
	
	From a historical perspective, such theory generalizes results of Andersen, Bando, Kazue, Nakajima, and Tian of the ends of the eighties for the real four dimensional case, based on Ulhenbeck's $\varepsilon$-regularity techniques. In particular, under the above assumptions, they were able to show \emph{orbifolds compactness}: {\itshape i.e.} $X_\infty$ is a smooth Einstein space away from finitely many points which can be locally modeled on $\R^4/\Gamma_p$ (where $\Gamma_p \subset SO(4)$ is a finite subgroup acting freely on the $3$-sphere, so that the metric pull-back to a smooth tensor (\emph{orbifold smooth metric}).
	
	In Cheeger-Colding theory "almost rigidity" theorems are used to define  \emph{stratification} of the limit space $X_\infty$. A prototypical example of rigidity theorem in smooth Riemannian geometry is given by the Cheeger-Gromoll splitting theorem: a complete Riemannian manifold $(M,g)$ with non-negative Ricci containing a \emph{line} ({\itshape i.e.} an infinity length geodesic $\gamma:(-\infty,\infty) \rightarrow M$, whose subsegments are always minimizing) must be isometric to the split product $M\simeq N\times \R$. The idea of the proof consists in considering \emph{Busemann functions} $b_{\gamma^{\pm}}(p):=\lim_{t\rightarrow \pm \infty} d(p, \gamma^{\pm}(t))-t$. Using Laplacian comparison for the distance function under the Ricci curvature bound hypothesis, one can show that $b_{\gamma^{\pm}}$ are sub-harmonic. Since $b_{\gamma^{+}}+b_{\gamma^{-}} \geq 0$ taking value zero exactly on the image of $\gamma$, by the maximum principle the function $b_{\gamma^{+}}=-b_{\gamma^{-}}$  is harmonic. Finally, by Bochner's formula the unit length vector field $\nabla b_{\gamma^{+}}$ must be parallel, and thus its flow gives the desired splitting.
	
	By performing an analysis of "approximate Busemann functions" via integral estimates, Cheeger-Colding theory  provides the following "quantitative" version of the above theorem, known as \emph{almost splitting theorem}. Let
	\begin{itemize}
		\item $\mathrm{Ric}(g)\geq -(n-1)\delta g $ (think $\delta \ll 1$);
		\item $p,q^+,q^- \in M$, with $d(p,q^\pm)\geq L \gg 1$ and $d(p,q^+)+d(p,q^-)-d(q^+,q^-)\leq \varepsilon \ll 1$.
	\end{itemize}
	Then $\exists$ a  ball $B_R \subseteq \R \times X$, with $X$ length space, such that $d_{GH}(B(p,R),B_R) \leq \phi(\delta,L^{-1},\varepsilon |  R)$, {\itshape i.e.} for fixed $R$,   $\phi \rightarrow 0$, as $\delta,  L^{-1}, \varepsilon,   \rightarrow 0$.
	
	As a corollary one sees that, by rescaling, pointed GH limits of sequences $(M_i,g_i,p_i)$ with Ricci bounded below must split if they contain a line, since  $\mathrm{Ric}({\lambda_i g_i}) \rightarrow0$ for $\lambda_i\rightarrow \infty$. 
	
	More generally, this result, combined with "volume cones are metric cones", is the basis of the following structural picture for GH limits $(M_i,g_i)\rightarrow X_\infty$ of manifolds with Ricci bounded below, bounded diameter and volume non-collapsing. 
	
	Take $p \in X_\infty$ and define a (non-necessarily unique!) \emph{metric tangent cone} at $p$ to be $C_p(X_\infty):=\lim_{pGH} (X_\infty, p, \lambda_i d_\infty)$ as $\lambda_i\rightarrow \infty$. Then $C_p(X_\infty)= \R^k\times C(Y)$ with $C(Y)$ a \emph{metric cone} ({\itshape i.e.} $d((r_1,y_1),(r_2,y_2))=\sqrt{r_1^2+r_2^2-2r_1r_2 \cos(d_Y(y_1,y_2))} $ of diameter less than $\pi$ and $n-k$ Hausdorff dimension. Note that the \emph{link} $Y$ may  be singular. Define the \emph{regular set} $$\mathcal{R}:=\{p\in X_\infty \,| \, \exists C_p(X_\infty)\cong_{\text{isom}} \R^n\},$$  which do not need to be open, and the \emph{singular strata} for $k \in \mathbb{N}$, $$\mathcal{S}:= X_\infty \setminus \mathcal{R}  \supseteq \mathcal{S}_k:=\{p \in X_\infty \,| \, \mbox{no tangent cones at $p$ splits } \R^{k+1}\}.$$
	
	Summarizing some of the main results of Cheeger-Colding theory:
	\begin{theorem} Let $\mathcal{S}_0\subseteq \dots \subseteq  \mathcal{S}_{n-2}=\mathcal{S}\subseteq X_\infty$ with Hausdorff dimension  $\mbox{dim}_{\mathcal{H}} \mathcal{S}_k\leq k$. If Ricci has  two side bounds then $\mathcal{R}$ is open (actually a $\mathcal{C}^{1,\alpha}$ manifold,   $\mathcal{C}^{\infty}$ if the metrics are Einstein). If $g_i$ are KE then $S=S_{n-4}$ (Cheeger-Colding-Tian).
		
	\end{theorem}
	
	Let us go back to discuss in more details the K\"ahler  situation.
	
	\subsection{Donaldson-Sun Theory  ("algebraic  regularity")}
	Let $X_\infty$ be a limit of smooth complex $n$-dimensional KE Fano manifolds $(M_i,g_i)$. The conical KE case is more technical, but part of the ideas can be already seen in this "absolute" situation. By  Cheeger-Colding regularity we know that $\mbox{dim}_\mathcal{H} X_\infty=2n$, its singular set agrees with $\mathcal{S}_{2n-4}$ stratum, the regular set is open and on it we can find a (incomplete) smooth KE metric $g_\infty$. Moreover, for all compact subset $K\subset\subset\mathcal{R}$ there exists diffeomorphisms onto their images $\psi_i:K\rightarrow M_i$ such that $\psi_i^*g_i \rightarrow g_\infty$ in $\mathcal{C}^{1,\alpha}$, and similarly for the complex structures $J_i$. 
	
	Not too surprisingly, the metric cones are now complex cones: $C_p(X_\infty)\cong_{isom}\C^k\times C(Y)$, with $C(Y)$ a Calabi-Yau cone, {\itshape i.e.} $\omega_{C(Y)}=\frac{\sqrt{-1}}{2} \partial \bar \partial r_0^2$ where $r_0$ is the distance from the apex of the cone, and $Ric(\omega_{CY})=0$. The link $Y$ is a potentially singular (Sasaki)-Einstein space with positive scalar curvature.
	
	The main results of Donaldson and Sun \cite{donaldson-sun-1} says the following:
	
	\begin{theorem}\label{AGL}
		The GH limit $X_\infty$ is naturally   \emph{homeomorphic} to $W\subseteq \C\mathbb{P}^n$ a \emph{singular Fano variety} with $\mathcal{S}(X_\infty)=\mathrm{Sing}(W)$.
	\end{theorem}  
	
	The rough (and very  imprecise!) idea consists in proving a uniform Kodaira embedding using the pick sections method "near the singularities" ("on metric tangent cones"). The main technical result used in their proof is  known as \emph{partial $\mathcal{C}^0$-estimates}. Let $b_k^{KE}(p)=\sum_i |s_i(p)|_{h^k_{KE}}^2$ be the Bergman Kernel using holomorphic sections $L^2$-orthonormal with respect to KE metrics. Then
	\begin{proposition}\label{partial} $\exists k_0=k_0(n, V), b=b(n,V)$ so that $b_{k_0}^{KE}\geq b^2 >0$ for all  KE $X^n$ of volume $V$.
	\end{proposition}
	
	Then the identification of the GH limit with a complex variety follows by comparing the GH limit $X_\infty$ with the "flat limit" obtaining as limit of the image $T_i^k(X_i)$ via Tian's $L^2$ orthonormal embedding $T_i^k$, for a fixed bounded $k$.
	
	To bound the Bergman's kernel, we need to construct a pick section at any  $p$ ({\itshape i.e.} $s \in H^0(K_{X_i}^{-k_0})$ non vanishing at $p$), also, and crucially, near the singularity formations, with some uniform control. The idea would be to trivialize the canonical on $U_i$ near $p$ with a local section $\sigma_0=1$. By cutting-off $\sigma:=\chi \sigma_0$ we have a global smooth section picked at $p$ and holomorphic near $p$ and on the complement of $U_i$ (where is identically zero). Next we want to solve the $\bar \partial$-equation $\bar \partial \tau= \bar \partial \sigma$ with estimates (of H\"ormander's $L^2$-type), so that $s:=\sigma-\tau$ will be fully holomorphic and picked. We need to find a nice $U_i$.
	The crucial idea is that, instead  of working on $X_i$, we can work near the tip of a \emph{metric tangent cone}, and then transplant the region on $X_i$ via the diffeomorphism $\phi_i$ provided by Cheeger-Colding theory. Working on such  a region it is possible to construct a $\sigma_0$ on the tangent cone (as outlined above), whose $\bar \partial$ is small in $L^2$ (via a good cut of function). Define $\tau:= \bar \partial^\ast \Delta_{\bar\partial^*}^{-1} \bar \partial \phi_i^\ast \sigma_0$. Here $\Delta_{\bar\partial^*}$ is the Laplacian on $(1,0)$ forms valued the $k$-plurianticanonical bundle. It is uniformly invertible for $k$ big enough, since by Bochner's formula $\Delta_{\bar\partial^*}=\nabla^\ast \nabla+\frac{1}{k}Ric(\omega)+1$, we have  $(\Delta_{\bar\partial^*}(s),s)_{L^2_{k\omega}} \geq \frac{1}{2} |s|^2_{L^2_{k\omega}}$. Let $\sigma:=\phi_i^\ast \sigma_0-\tau$. Then it is holomorphic. Moreover $|\tau|_{2}^2\leq 2|\bar \partial \phi_i^\ast \sigma_0|^2_2$ which is small. By Moser's iteration its $\mathcal{C}^1$-norm is also small. Thus $|\sigma(\phi_i^{-1}(p))|\geq C(k_0)>0$, for a $k_0$ bounded by the geometry independent of $i$. Since also $\nabla \sigma$ is bounded uniformly, $T_i$ are Lipschitz and  converge to some continuous map $T_\infty$ which identifies the GH limit $X_{\infty}$ with the flat limit. 
	
	\begin{remark}
		In reality,  to get Theorem \ref{AGL} from Proposition \ref{partial}  is necessary to raise further (in a bounded way) the $k_0$ constant (in order to really separate points in the limit). Of course there are many other difficulties we haven't discussed in the oversimplified sketch given above.
	\end{remark} 
	\begin{remark}
		\begin{itemize}
			\item For the conical case, the limit is a Fano variety $W$ equipped with a (singular) limit  divisors $D_\infty$ where, generically, the singular limit KE metric is going to have some cone angle  singularities in the transverse directions to smooth point of $D_\infty$.
			
			\item It can be proved that the singularities of the limit are Kawamata log terminal (klt), that is $K_{\hat{X}}=\phi^* K_{X}+\sum a_i E_i$, with $a_i > -1$ for any divisorial resolutions of the singularities (since this is  essentially equivalent to say that near the singularity there is a (root) of a holomorphic volume   form $\Omega$ such that $\int_U \Omega\wedge \bar\Omega< +\infty$). In dimension two these singularities coincide with quotient singularities, thus recovering the orbifold theory. However, in higher dimension this is not the case, {\itshape e.g.} the ODP singularity $\sum_{i=0}^nx^2=0$ is not of quotient type as long as $n\geq3$. Moreover, the understanding of the metric near such singularities is much more subtle, but there are very  important  recent results in such direction \cite{donaldson-sun-2}.
			.
			
			\item It can be shown that the singular limit KE metrics are \emph{weak KE} in the pluripotential theory sense \cite{egz-JAMS}.
		\end{itemize}
	\end{remark}

	\subsection{Summing-up: the idea of the proof}
	Let us now go back to the Donaldson continuity path and say that $X_\infty$ is the GH limit of a sequence of increasing conical KE $(M_,\omega_{\beta_i})$, with $\beta_i\nearrow \underline{\beta}$ solving \eqref{DP}. Then, by the above regularity theories, we can assume that $X_\infty\cong (W, D_\infty, \omega_{\underline{\beta}})$ for a weak (log) KE Fano pair. We are now again in the domain of algebraic geometry! The main trick is now to use such $ (W, D_\infty, \omega_{\underline{\beta}})$ to construct a test configuration so that we can use the K-stability hypothesis. 
	
	The first step consists in extending the Donaldson-Futaki invariant to the pair setting: $(\mathcal{X},\mathcal{D})\rightarrow \C$, with $DF((X,(1-\beta)D),(\mathcal{X}, (1-\beta)D) ):=DF(X, \mathcal{X})+(1-\beta) \tilde{F}(D)$ in a natural way:  if $(X,D,\omega_\beta)$ is a singular space with a weak conical metric, then for a test configuration induced by holomorphic vector fields tangential to $D$, we have $DF(\mathcal{X}, (1-\beta)D)=F_D(v)$, for the natural extension of the Futaki type invariant \cite{li-sun}, as in Proposition \ref{DF=F}. In particular, if $\omega_\beta$ is conical KE, $DF(\mathcal{X}, (1-\beta)D)=0$ for such test configurations. Note the linearity in $\beta$ of the conical ($\log$) version of the Futaki invariant.
	
	If $(W, D_\infty)\cong(M,D)$ we would have done by uniqueness of the KE metric. If not, then by Donaldson-Sun we know that $(W, D_\infty)$ is realized as a flat limit of $(M,D)$ with some "universal" projective space $\C\mathbb{P}^{n_{k}}$, for fixed uniform $k$. In particular $(W, D_\infty) \in \overline{\mathrm{PGL}({n_{k}}).(M,D)}$. Since  $ (W, D_\infty, \omega_{\underline{\beta}})$ is KE, by an extension of Matsushima's reductivity, we have that $\mathrm{Aut}(W,D_\infty)$ is reductive. Hence, by the standard Luna's slice theorem for reductive groups, we can assume that there exists a $\C^* \hookrightarrow \mathrm{Aut}(W,D_\infty) \subseteq  \mathrm{PGL}({n_{k}})$ so that $(W,D)=\lim_{t\rightarrow 0} t.(M,D)$, {\itshape i.e.} $(W,D)$ is the central fibre for a test configuration! In particular, for such test configuration the log Futaki invariant is zero for parameter $\underline{\beta}$ since the central fibre has a KE metric. By linearity of the Futaki invariant, this  implies that the "absolute" Futaki invariant for $X$ has to be negative. But this contradicts our K-stability hypothesis. Thus only case one is indeed possible, and we can repeat the argument until $\underline{\beta}=1$, and finally construct a smooth KE metric on $X$! Be aware that the fact that cone angle metrics open up to usual \emph{smooth} KE  metrics  is far from being a triviality, even if it is very intuitive.
	
	This concludes the sketch of the arguments for proving Theorem \ref{CDS}.
	
	\begin{remark}
		\begin{itemize}
			\item There are now different proofs of the  equivalence between KE metric and K-stability: via the original Aubin's path \cite{datar-gabor}, via Ricci flow \cite{chen-sun-wang}, via Calculus of Variations \cite{berman-boucksom-jonsson, demailly-Bourbaki}. 
			\item Checking K-stability to construct new KE metric is still very hard (if not impossible at the present state-of-the art!) due to the too many test configurations which {\itshape a priori} needed to be checked. Understanding better this issue is definitely a crucial aspect of future investigations in the field.
		\end{itemize} 
	\end{remark}

	\subsection{KE moduli spaces and explicit examples}
	
	Let $$\overline{\mathcal{E M}}^{GH}:= \overline{\{(X^n,\omega) \, | \, \omega \mbox{ is KE Fano} \}/{\mbox{biholo-isom}}}^{GH}$$ be the \emph{compactification of the moduli space of KE Fano manifolds} obtained by adding the singular weak KE spaces coming from GH-degenerations. In some sense, these spaces can be thought as higher dimensional generalizations of the Deligne-Mumford moduli compactification in the curve case, but now for the \emph{positive} KE case. Thus, it should be not too surprisingly that it can be shown that such GH compactifications admit a natural structure of complex variety. Sometime they are  also known as  K-moduli space $\mathcal{K M}$, for obvious reasons. Can we find explicit examples of such $\overline{\mathcal{E M}}^{GH}$?
	This is important since testing K-stability, as we said, is still very hard: if we can explicitly described such moduli spaces in concrete situations, then we can understand exactly which Fano varieties in \emph{given families} are indeed KE/K-polystable.
	
	The idea to study such problem is via a "moduli continuity method": that is, given a family $\mathcal{X} \rightarrow{H}$ of Fano varieties, we would like to study the set of parameters for which the fibre variety $X_t$ admits a KE metrics. If we can prove that there exists at least one fibre which is KE, that the KE condition is essentially open (in ideal situations) and that abstract GH limits are actually naturally embedded in our starting family, it is possible to conclude (via "stability comparison") that $\overline{\mathcal{E M}}^{GH}_d=\mathcal{H}_d/\!\!/G$, where $d$ is some numerical parameter (as the volume/degree) and $G$ is a group acting on the parameter space giving a classical GIT quotient. In particular, K-stability would be equivalent to  GIT for such family, and GIT stability can be in principle checked by hand. This strategy has been used for fully understanding the two dimensional case in \cite{odaka-spotti-sun}, extending results of \cite{tian-Invent90, mabuchi-mukai}. 
	
	For example one can study \emph{cubic surfaces} in $\C\mathbb{P}^3$. Their defining coefficient form a parameter space $\mathcal{H}_3=\C \mathbb{P}^{19}$ on which the group $\mathrm{SL}(4,\C)$ acts naturally by reparameterization. The Fermat cubic ${x^3+y^3+z^3+t^3=0}$ is know to have alpha invariant bigger than $2/3$ and hence is KE. By implicit function theorem one can see that, among smooth cubic, the KE condition is open since they have only discrete automorphisms (actually, by \cite{tian-Invent90} we also know that smooth ones are all KE).  In any case let $X_\infty \cong W$ an abstract GH limit, which has only isolated orbifold singularities by the general regularity theory of degenerate limit recalled above. Then, by Bishop-Gromov inequality, we can compare the metric density at singular points with the total volume. This gives that $3=c_1^2(W)< 12 \sharp (\Gamma_p)^{-1}$. In particular, $\Gamma_p\cong  \Z_2, \Z_3$. It can be shown that in such case only the $SU(2)$ representations survives (related to the notion of $\mathbb{Q}$-Gorenstein smoothability, which is a consequence of Donaldson-Sun theory). This means that $W$ has only  $A_k$ singularities of the form $x^2+y^2=z^{k+1}$ with $k\leq2$. It is a known result that two dimensional Fano varieties of degree three  with such singularities have still very ample anti-canonical bundle, which thus provides embeddings $W \subseteq \C \mathbb{P}^{3}$. By Berman results we know that $W$ is K-polystable (since singular KE). It is then easy to see in such case that K-polystability implies GIT-polystabilty for the essentially unique linearization of the $SL(4,\C)$ action. Thus, with some further work,  $\overline{\mathcal{E M}}^{GH}_3\cong \mathcal{H}_3//SL(4,\C)$. The study of such GIT problem goes back to Hilbert. In particular, it is known that cubics with at worse $A_1$-singularities are stable, and only (the orbit of) the toric $xyz=t^3$  (with $3A_2$ singularities), is strictly polystable. Thus this classifies also all the cubics which are KE.
	
\begin{remark}
Very recently, new techniques ({\itshape e.g.} \cite{li-arxiv1511, liu-arxiv1605}) based on better understanding of the metric tangent cones in higher dimension and their relations with algebraic geometry,  made possible to study in concrete situations such moduli spaces in higher dimensions too ({\itshape e.g.} \cite{spotti-sun-arxiv1705, liu-xu-arxiv1706}).
\end{remark}


\begin{thebibliography}{9999}

\bibitem[Aub76]{aubin-CRAS}
Th. Aubin, Équations du type Monge-Ampère sur les variétés kähleriennes compactes, {\em C. R. Acad. Sci. Paris Sér. A-B} \textbf{283} (1976), no. 3, Aiii, A119--A121.

\bibitem[Aub78]{aubin-BSM}
Th. Aubin, Équations du type Monge-Ampère sur les variétés kählériennes compactes, {\em Bull. Sci. Math. (2)} \textbf{102} (1978), no. 1, 63--95.

\bibitem[Aub84]{aubin-JFA}
Th. Aubin, Réduction du cas positif de l'équation de Monge-Ampère sur les variétés kählériennes compactes à la démonstration d'une inégalité, {\em J. Funct. Anal.} \textbf{57} (1984), no. 2, 143--153.

\bibitem[Aub98]{aubin-book}
Th. Aubin, {\em Some nonlinear problems in Riemannian geometry}, Springer Monographs in Mathematics, Springer-Verlag, Berlin, 1998.

\bibitem[Bal06]{ballmann}
W. Ballmann, {\em Lectures on Kähler manifolds}, ESI Lectures in Mathematics and Physics, European Mathematical Society (EMS), Zürich, 2006.

\bibitem[BM87]{bando-mabuchi}
S. Bando, T. Mabuchi, Uniqueness of Einstein Kähler metrics modulo connected group actions, in {\em Algebraic geometry, Sendai, 1985}, 11--40, Adv. Stud. Pure Math., \textbf{10}, North-Holland, Amsterdam, 1987.

\bibitem[Ber13]{berman-Adv}
R. J.  Berman, A thermodynamical formalism for Monge-Ampère equations, Moser-Trudinger inequalities and Kähler-Einstein metrics, {\em Adv. Math}.\textbf{248} (2013), 1254--1297. 
 
\bibitem[Ber16]{berman-Inv}
R. J. Berman, K-polystability of $\mathbb Q$-Fano varieties admitting Kähler-Einstein metrics, {\em Invent. Math.} \textbf{203} (2016), no. 3, 973--1025.

\bibitem[BB17]{berman-berndtsson}
R. J. Berman, B. Berndtsson, Convexity of the K-energy on the space of Kähler metrics and uniqueness of extremal metrics, {\em J. Amer. Math. Soc.} \textbf{30} (2017), no. 4, 1165--1196.

\bibitem[BBJ15]{berman-boucksom-jonsson}
R. Berman, S. Boucksom, M. Jonsson, A variational approach to the Yau-Tian-Donaldson conjecture, \texttt{arXiv:1509.04561}.

\bibitem[Bes87]{besse}
A. L. Besse, {\em Einstein manifolds}, Reprint of the 1987 edition, Classics in Mathematics. Springer-Verlag, Berlin, 2008.

\bibitem[Bla56]{blanchard}
A. Blanchard, Sur les variétés analytiques complexes, {\em Ann. Sci. Ecole Norm. Sup. (3)} \textbf{73} (1956), 157--202.

\bibitem[Boc46]{bochner}
S. Bochner, Vector fields and Ricci curvature, {\em Bull. Amer. Math. Soc.} \textbf{52} (1946), 776--797.

\bibitem[Bog78]{bogomolov}
F. A. Bogomolov, Holomorphic tensors and vector bundles on projective manifolds, {\em Izv. Akad. Nauk SSSR Ser. Mat.} \textbf{42} (1978), no. 6, 1227--1287, 1439. English translation in {\em Math. USSR-Izv.} \textbf{13} (1979), no. 3, 499--555.

\bibitem[BBI01]{burago-burago-ivanov}
D. Burago, Y. Burago, S. Ivanov, {\em A course in metric geometry}, Graduate Studies in Mathematics, \textbf{33}, American Mathematical Society, Providence, RI, 2001.

\bibitem[CC95]{caffarelli-cabre}
L. A. Caffarelli, X. Cabré, {\em Fully nonlinear elliptic equations}, American Mathematical Society Colloquium Publications, \textbf{43}, American Mathematical Society, Providence, RI, 1995.

\bibitem[Cal54]{calabi-ICM}
E. Calabi, The space of K\"ahler metrics, {\em Proc. Internat. Congress Math. Amsterdam} \textbf{2} (1954), pp. 206--207.

\bibitem[Cal57]{calabi-57}
E. Calabi, On K\"ahler manifolds with vanishing canonical class, in {\em Algebraic geometry and topology. A symposium in honor of S. Lefschetz}, pp. 78--89, Princeton University Press, Princeton, N. J., 1957.

\bibitem[CdS01]{cannasdasilva}
A. Cannas da Silva, {\em Lectures on symplectic geometry}, Lecture Notes in Mathematics, \textbf{1764}, Springer-Verlag, Berlin, 2001.

\bibitem[CL73]{carrell-lieberman}
J. B. Carrell, D. I. Lieberman, Holomorphic Vector fields and Kaehler Manifolds, {\em Invent. Math.} \textbf{21} (1973), 303--309.

\bibitem[Cat99]{catlin}
D. Catlin, The Bergman kernel and a theorem of Tian, in {\em Analysis and geometry in several complex variables (Katata, 1997)}, 1--23, Trends Math., Birkhäuser Boston, Boston, MA, 1999.

\bibitem[Che01]{cheeger-Fermi}
J. Cheeger, {\em Degeneration of Riemannian metrics under Ricci curvature bounds}, Lezioni Fermiane, Scuola Normale Superiore, Pisa, 2001.

\bibitem[Che08]{cheltsov}
I. A. Chel'tsov, Log canonical thresholds of del Pezzo surfaces, {\em Geom. Funct. Anal.} \textbf{18} (2008), no. 4, 1118--1144.

\bibitem[CS08]{cheltsov-shramov}
I. A. Chel'tsov, K. A. Shramov, Log-canonical thresholds for nonsingular Fano threefolds. With an appendix by J.-P. Demailly, {\em Uspekhi Mat. Nauk} \textbf{63} (2008), no. 5(383), 73--180; translation in {\em Russian Math. Surveys} \textbf{63} (2008), no. 5, 859--958.

\bibitem[Che00]{chen-JDG00}
X.x. Chen, The space of K\"ahler metrics, {\em J. Differential Geom.} \textbf{56} (2000), no. 2, 189--234.

\bibitem[CDS14]{chen-donaldson-sun-IMRN}
X.x. Chen, S. Donaldson, S. Sun, Kähler-Einstein metrics and stability, {\em Int. Math. Res. Not. IMRN} \textbf{2014} (2014), no. 8, 2119--2125.

\bibitem[CDS15]{chen-donaldson-sun}
X.x. Chen, S. Donaldson, S. Sun, K\"ahler-Einstein metrics on Fano manifolds. I: Approximation of metrics with cone singularities, II: Limits with cone angle less than $2\pi$, III: Limits as cone angle approaches $2\pi$ and completion of the main proof, {\em J. Amer. Math. Soc.} \textbf{28} (2015), no. 1, 183--197, 199--234, 235--278.

\bibitem[CSW15]{chen-sun-wang}
X.x. Chen, S. Sun, B. Wang, Kähler-Ricci flow, Kähler-Einstein metric, and K-stability, \texttt{arXiv:1508.04397}.

\bibitem[CTW17]{chu-tosatti-weinkove}
J. Chu, V. Tosatti, B. Weinkove, On the $C^{1,1}$ Regularity of Geodesics in the Space of Kähler Metrics, {\em Ann. PDE} \textbf{3} (2017), no. 2, 3--15.

\bibitem[DL12]{darvas-lempert}
T. Darvas, L. Lempert, Weak geodesics in the space of Kähler metrics, {\em Math. Res. Lett.} \textbf{19} (2012), no. 5, 1127--1135.

\bibitem[DS16]{datar-gabor}
V. Datar, G. Székelyhidi, Kähler-Einstein metrics along the smooth continuity method, {\em Geom. Funct. Anal.} \textbf{26} (2016), no. 4, 975--1010.

\bibitem[Deb01]{debarre}
O. Debarre, {\em Higher-dimensional algebraic geometry}, Universitext, Springer-Verlag, New York, 2001.

\bibitem[DGMS75]{deligne-griffiths-morgan-sullivan}
P. Deligne, Ph. Griffiths, J. Morgan, D.  Sullivan, Real homotopy theory of K\"ahler manifolds, {\em Invent. Math.} \textbf{29} (1975), no.~3, 245--274.

\bibitem[Dem12]{demailly-agbook}
J.-P. Demailly, {\em Complex Analytic and Differential Geometry}, \url{http://www-fourier.ujf-grenoble.fr/~demailly/manuscripts/agbook.pdf}, 2012.

\bibitem[Dem17]{demailly-Bourbaki}
J.-P. Demailly, Variational approach for complex Monge-Ampère equations and geometric applications, Séminaire Bourbaki, Exposé 1112, {\em Astérisque} \textbf{390} (2017), 245--275.

\bibitem[DK01]{demailly-kollar}
J.-P. Demailly, J. Kollár, Semi-continuity of complex singularity exponents and Kähler-Einstein metrics on Fano orbifolds, {\em Ann. Sci. École Norm. Sup. (4)} \textbf{34} (2001), no. 4, 525--556.

\bibitem[Din88]{ding}
W. Y. Ding, Remarks on the existence problem of positive Kähler-Einstein metrics, {\em Math. Ann.} \textbf{282} (1988), no. 3, 463--471.

\bibitem[Dol53]{dolbeault}
P. Dolbeault, Sur la cohomologie des variétés analytiques complexes., {\em C. R. Acad. Sci. Paris} \textbf{236} (1953). 175--177.
 
\bibitem[Don87]{donaldson-Duke87}
S. K. Donaldson, Infinite determinants, stable bundles and curvature, {\em Duke Math. J.} \textbf{54} (1987), no. 1, 231--247.

\bibitem[Don97]{donaldson-Fields}
S.~K. Donaldson, Remarks on gauge theory, complex geometry and 4-manifold topology, in {\em Fields Medallists' lectures}, 384--403, World Sci. Ser. 20th Century Math., \textbf{5}, World Sci. Publ., River Edge, NJ, 1997. 

\bibitem[Don99]{donaldson-AMST}
S. K. Donaldson, Symmetric spaces, Kähler geometry and Hamiltonian dynamics, in {\em Northern California Symplectic Geometry Seminar}, 13--33, Amer. Math. Soc. Transl. Ser. 2, \textbf{196}, Adv. Math. Sci., \textbf{45}, Amer. Math. Soc., Providence, RI, 1999.

\bibitem[Don01]{donaldson-JDG01}
S. K. Donaldson, Scalar curvature and projective embeddings. I, {\em J. Differential Geom.} \textbf{59} (2001), no. 3, 479--522.

\bibitem[Don02a]{donaldson-JSG}
S. K.Donaldson, Holomorphic discs and the complex Monge-Ampère equation, {\em J. Symplectic Geom.} \textbf{1} (2002), no. 2, 171--196.

\bibitem[Don02b]{donaldson-JDG}
S. K. Donaldson, Scalar curvature and stability of toric varieties, {\em J. Differential Geom.} \textbf{62} (2002), no. 2, 289--349.

\bibitem[Don05]{donaldson-JDG05}
S. K. Donaldson, Lower bounds on the Calabi functional, {\em J. Differential Geom.} \textbf{70} (2005), no. 3, 453--472.


\bibitem[Don12]{donaldson-Essays}
S. K. Donaldson, Kähler metrics with cone singularities along a divisor, in {\em Essays in mathematics and its applications}, 49--79, Springer, Heidelberg, 2012.

\bibitem[Don15]{donaldson-Zariski}
S. K. Donaldson, Algebraic families of constant scalar curvature Kähler metrics, in {\em Surveys in differential geometry 2014. Regularity and evolution of nonlinear equations}, 111--137, Surv. Differ. Geom., \textbf{19}, Int. Press, Somerville, MA, 2015.

\bibitem[DS14]{donaldson-sun-1}
S. K. Donaldson, S. Sun, Gromov-Hausdorff limits of Kähler manifolds and algebraic geometry, {\em Acta Math.} \textbf{213} (2014), no. 1, 63--106.

\bibitem[DS17]{donaldson-sun-2}
S. K. Donaldson, S. Sun, Gromov-Hausdorff limits of Kähler manifolds and algebraic geometry, II, {\em J. Differential Geom.} \textbf{107} (2017), no. 2, 327--371.

\bibitem[Eva82]{evans}
L. C. Evans, Classical solutions of fully nonlinear, convex, second-order elliptic equations, {\em Comm. Pure Appl. Math.} \textbf{35} (1982), 333--363.

\bibitem[EGZ09]{egz-JAMS}
Ph. Eyssidieux, V. Guedj, A. Zeriahi, Singular Kähler-Einstein metrics, {\em J. Amer. Math. Soc.} \textbf{22} (2009), no. 3, 607--639.

\bibitem[Fig17]{figalli-book}
A. Figalli, {\em The Monge-Ampère equation and its applications}, Zurich Lectures in Advanced Mathematics. European Mathematical Society (EMS), Zürich, 2017.

\bibitem[Fuj90]{fujiki}
A. Fujiki, Moduli space of polarized algebraic manifolds and K\"ahler metrics [translation of S\^{u}gaku \textbf{42} (1990), no.~3, 231--243], Sugaku Expositions \textbf{5} (1992), no.~2, 173--191.

\bibitem[Fuj15]{fujita-arxiv1508}
K. Fujita, Optimal bounds for the volumes of Kähler-Einstein Fano manifolds, \texttt{arXiv:1508.04578}.

\bibitem[Fuj16a]{fujita-arxiv1602}
K. Fujita, A valuative criterion for uniform K-stability of $\mathbb Q$-Fano varieties, \texttt{arXiv:1602.00901}.

\bibitem[Fuj16b]{fujita-arxiv}
K. Fujita, K-stability of Fano manifolds with not small alpha invariants, \texttt{arXiv:1606.08261}.

\bibitem[Fut83]{futaki}
A. Futaki, An obstruction to the existence of Einstein Kähler metrics, {\em Invent. Math.} \textbf{73} (1983), no. 3, 437--443.

\bibitem[Gau15]{gauduchon-book}
P. Gauduchon, {\em Calabi's extremal K\"ahler metrics: An elementary introduction}, \url{http://germanio.math.unifi.it/wp-content/uploads/2015/03/dercalabi.pdf}.

\bibitem[GT98]{gilbarg-trudinger}
D. Gilbarg, N. S. Trudinger, {\em Elliptic partial differential equations of second order}, Reprint of the 1998 edition, Classics in Mathematics, Springer-Verlag, Berlin, 2001.

\bibitem[GH78]{griffiths-harris}
Ph. Griffiths, J. Harris, {\em Principles of algebraic geometry}, Reprint of the 1978 original, Wiley Classics Library. John Wiley \& Sons, Inc., New York, 1994.

\bibitem[LNM2038]{guedj-book}
{\em Complex Monge-Ampère equations and geodesics in the space of Kähler metrics}, Edited by Vincent Guedj, Lecture Notes in Mathematics, \textbf{2038}, Springer, Heidelberg, 2012.

\bibitem[GZ17]{guedj-zeriahi}
V. Guedj, A. Zeriahi, {\em Degenerate complex Monge-Ampère equations}, EMS Tracts in Mathematics, \textbf{26}, European Mathematical Society (EMS), Zürich, 2017.

\bibitem[Gue15]{guenancia}
H. Guenancia, Kähler-Einstein metrics: from cones to cusps, \texttt{arXiv:1504.01947}.

\bibitem[Huy05]{huybrechts}
D. Huybrechts, {\em Complex geometry. An introduction}, Universitext, Springer-Verlag, Berlin, 2005.

\bibitem[JMR16]{jeffres-mazzeo-rubinstein}
T. Jeffres, R. Mazzeo, Y. A. Rubinstein, Kähler-Einstein metrics with edge singularities, {\em Ann. of Math. (2)} \textbf{183} (2016), no. 1, 95--176.

\bibitem[Kob74]{kobayashi-Ann}
S. Kobayashi, On compact Kähler manifolds with positive definite Ricci tensor, {\em Ann. of Math. (2)} \textbf{74} (1961), 570--574.

\bibitem[KN2]{kobayashi-nomizu-2}
S. Kobayashi, K. Nomizu, {\em Foundations of differential geometry. Vol. II}, Reprint of the 1969 original, Wiley Classics Library, A Wiley-Interscience Publication. John Wiley \& Sons, Inc., New York, 1996.

\bibitem[Kod54]{kodaira-Embedding}
K. Kodaira, On Kähler varieties of restricted type (an intrinsic characterization of algebraic varieties), {\em Ann. Math. (2)} \textbf{60} (1954), 28--48.

\bibitem[Kod86]{kodaira-book}
K. Kodaira, {\em Complex manifolds and deformation of complex structures}, Translated from the 1981 Japanese original by Kazuo Akao, Reprint of the 1986 English edition, Classics in Mathematics, Springer-Verlag, Berlin, 2005.

\bibitem[Kol96]{kollar}
J. Kollár, {\em Rational curves on algebraic varieties}, Ergebnisse der Mathematik und ihrer Grenzgebiete, \textbf{3}, Folge, A Series of Modern Surveys in Mathematics, \textbf{32}, Springer-Verlag, Berlin, 1996.

\bibitem[KMM92]{kmm}
J. Kollár, Y. Miyaoka, S. Mori, Rational Connectedness and Boundedness of Fano Manifolds, {\em J. Diff. Geom.} \textbf{36} (1992), 765--769.

\bibitem[KM98]{kollar-mori}
J. Kollár, S. Mori, {\em Birational geometry of algebraic varieties}, With the collaboration of C. H. Clemens and A. Corti, Translated from the 1998 Japanese original, Cambridge Tracts in Mathematics, \textbf{134}, Cambridge University Press, Cambridge, 1998.

\bibitem[Kry82]{krylov}
N. V. Krylov. Boundedly inhomogeneous elliptic and parabolic equations, {\em Izv. Akad. Nauk SSSR Ser. Mat.} \textbf{46} (1982), 487--523.

\bibitem[LBS94]{lebrun-simanca}
C. R. LeBrun, S. Simanca, Extremal K\"ahler metrics and complex deformation theory, {\em Geom. Funct. Analysis} \textbf{4} (1994), no. 3, 298--336.

\bibitem[LV13]{lempert-vivas}
L. Lempert, L. Vivas, Geodesics in the space of Kähler metrics, {\em Duke Math. J.} \textbf{162} (2013), no. 7, 1369--1381.

\bibitem[Li11]{li-AdvMath}
C. Li, Greatest lower bounds on Ricci curvature for toric Fano manifolds, {\em Adv. Math.} \textbf{226} (2011), no. 6, 4921--4932.

\bibitem[Li15a]{li-arxiv1511}
C. Li, Minimizing normalized volumes of valuations, \texttt{arXiv:1511.08164}.

\bibitem[Li15b]{li-CCM}
C. Li, Remarks on logarithmic K-stability, {\em Commun. Contemp. Math.} \textbf{17} (2015), no. 2, 1450020, 17 pp..

\bibitem[LS14]{li-sun}
C. Li, S. Sun, Conical Kähler-Einstein metrics revisited, {\em Comm. Math. Phys.} \textbf{331} (2014), no. 3, 927--973.

\bibitem[LX14]{li-xu}
C. Li, C. Xu, Special test configuration and K-stability of Fano varieties, {\em Ann. of Math. (2)} \textbf{180} (2014), no. 1, 197--232.

\bibitem[LY87]{li-yau}
J. Li, S.-T. Yau, Hermitian-Yang-Mills connection on non-Kähler manifolds, in {\em Mathematical aspects of string theory (San Diego, Calif., 1986)}, 560--573, Adv. Ser. Math. Phys., \textbf{1}, World Sci. Publishing, Singapore, 1987.

\bibitem[Lic58]{lichnerowicz}
A. Lichnerowicz, {\em Géométrie des groupes de transformation}, Travaux et Recherches Mathématiques \textbf{3}, Dunod (1958).

\bibitem[Liu16]{liu-arxiv1605}
Y. Liu, The volume of singular K\"hler-Einstein Fano varieties, \texttt{arXiv:1605.01034}.

\bibitem[LX17]{liu-xu-arxiv1706}
Y. Liu, C. Xu, K-stability of cubic threefolds, \texttt{arXiv:1706.01933}.

\bibitem[Lu00]{lu}
Z. Lu, On the lower order terms of the asymptotic expansion of Tian-Yau-Zelditch, {\em Amer. J. Math.} \textbf{122} (2000), no. 2, 235--273.

\bibitem[LT95]{lubke-teleman}
M. Lübke, A. Teleman, {\em The Kobayashi-Hitchin correspondence}, World Scientific Publishing Co., Inc., River Edge, NJ, 1995.

\bibitem[Mab85]{mabuchi}
T. Mabuchi, A Functional Integrating Futaki's Invariant, {\em Proc. Japan Acad. Ser. A} \textbf{61} (1985), no. 4, 119--120.

\bibitem[Mab87]{mabuchi}
T. Mabuchi, Some symplectic geometry on compact Kähler manifolds. I, {\em Osaka J. Math.} \textbf{24} (1987), no. 2, 227--252.

\bibitem[Mab08]{mabuchi-arXiv}
T. Mabuchi, K-stability of constant scalar curvature polarization, \texttt{arXiv:0812.4093}.

\bibitem[MM93]{mabuchi-mukai}
T. Mabuchi, S. Mukai, Stability and Einstein-Kähler metric of a quartic del Pezzo surface, in {\em Einstein metrics and Yang-Mills connections (Sanda, 1990)}, 133--160, Lecture Notes in Pure and Appl. Math., \textbf{145}, Dekker, New York, 1993.

\bibitem[MS06]{martelli-sparks}
D. Martelli, J. Sparks, Toric geometry, Sasaki-Einstein manifolds and a new infinite class of AdS/CFT duals, {\em Comm. Math. Phys.} \textbf{262} (2006), no. 1, 51--89.

\bibitem[Mat57]{matsushima}
Y. Matsushima, Sur la structure du groupe d'homéomorphismes analytiques d'une certaine variété k\"ahlérienne, {\em Nagoya Math. J.} \textbf{11} (1957), 145--150.

\bibitem[Miy77]{miyaoka}
Y. Miyaoka, On the Chern numbers of surfaces of general type, {\em Invent. Math.} \textbf{42} (1977), 225--237.

\bibitem[Mor79]{mori}
S. Mori, Projective manifolds with ample tangent bundles, {\em Ann. of Math.} \textbf{110} (1979), 593--606.

\bibitem[Mor07]{moroianu}
A. Moroianu, {\em Lectures on Kähler geometry}, London Mathematical Society Student Texts, \textbf{69}, Cambridge University Press, Cambridge, 2007.

\bibitem[MK71]{kodaira-morrow}
J. Morrow, K. Kodaira, {\em Complex manifolds}, Reprint of the 1971 edition with errata, AMS Chelsea Publishing, Providence, RI, 2006.

\bibitem[Mos65]{moser}
J. Moser, On the volume elements on a manifold, {\em Trans. Amer. Math. Soc.} \textbf{120} (1965), 286--294.

\bibitem[Mum77]{mumford}
D. Mumford, Stability of projective varieties, {\em Enseignement Math. (2)} \textbf{23} (1977), no. 1-2, 39--110.

\bibitem[Mum79]{mumford-Fake}
D. Mumford, An algebraic surface with $K$ ample, $(K^2)=9$, $p_g=q=0$, {\em Amer. J. Math.} \textbf{101} (1979), no. 1, 233--244.

\bibitem[MFK94]{mumford-fogarty-kirwan}
D. Mumford, J. Fogarty, F. Kirwan, {\em Geometric invariant theory}, Third edition, Ergebnisse der Mathematik und ihrer Grenzgebiete (2), \textbf{34}, Springer-Verlag, Berlin, 1994.

\bibitem[MS39]{myers-steenrod}
S. B. Myers, N. E. Steenrod, The group of isometries of a Riemannian manifold, {\em Ann. of Math. (2)} \textbf{40} (1939), no. 2, 400--416.

\bibitem[NN57]{newlander-nirenberg}
A. Newlander, L. Nirenberg, Complex analytic coordinates in almost complex manifolds, {\em Ann. of Math. (2)} \textbf{65} (1957), 391--404.

\bibitem[New78]{Newsted_Tata_mod}
P. E. Newstead, {\em Introduction to moduli problems and orbit spaces}, Tata Institute of Fundamental Research Lectures on Mathematics and Physics, \textbf{51}, Tata Institute of Fundamental Research, Bombay; by the Narosa Publishing House, New Delhi, 1978.

\bibitem[OSS16]{odaka-spotti-sun}
Y. Odaka, C. Spotti, S. Sun, Compact moduli spaces of del Pezzo surfaces and Kähler-Einstein metrics, {\em J. Differential Geom.} \textbf{102} (2016), no. 1, 127--172.

\bibitem[Pag79]{page}
D. Page, A Compact Rotating Gravitational Instanton, {\em Phys. Lett. B} \textbf{79B} (1979), 235--238.

\bibitem[PT09]{paul-tian}
S. T. Paul, G. Tian, CM stability and the generalized Futaki invariant II, {\em Astérisque} \textbf{328 (2009)} (2010), 339--354.

\bibitem[PSS07]{pss}
D. H. Phong, N. Sesum, J. Sturm, Multiplier ideal sheaves and the Kähler-Ricci flow, {\em Comm. Anal. Geom.} \textbf{15} (2007), no. 3, 613--632.

\bibitem[RT07]{ross-thomas}
J. Ross, R. Thomas, A study of the Hilbert-Mumford criterion for the stability of projective varieties, {\em J. Algebraic Geom.} \textbf{16} (2007), no. 2, 201--	255.

\bibitem[Rua98]{ruan}
W.-D. Ruan, Canonical coordinates and Bergmann metrics, {\em Comm. Anal. Geom.} \textbf{6} (1998), no. 3, 589--631.

\bibitem[Sem92]{semmes}
S. Semmes, Complex Monge-Ampère and symplectic manifolds, {\em Amer. J. Math.} \textbf{114} (1992), no. 3, 495--550.

\bibitem[Spo17]{Spotti_Cortona_moduli}
C. Spotti, Kähler-Einstein Metrics on $\mathbb Q$-Smoothable Fano Varieties, Their Moduli and Some Applications, in {\em Complex and Symplectic Geometry}, 211--229, Springer INdAM Series 21, Springer, 2017.

\bibitem[SS17]{spotti-sun-arxiv1705}
C. Spotti, S. Sun, Explicit Gromov-Hausdorff compactifications of moduli spaces of Kähler-Einstein Fano manifolds, \texttt{arXiv:1705.00377}.

\bibitem[Sto09]{stoppa-AdvMath}
J. Stoppa, K-stability of constant scalar curvature Kähler manifolds, {\em Adv. Math.} \textbf{221} (2009), no. 4, 1397--1408.

\bibitem[Szé11]{gabor-CM}
G. Székelyhidi, Greatest lower bounds on the Ricci curvature of Fano manifolds, {\em Compos. Math.} \textbf{147} (2011), no. 1, 319--331.

\bibitem[Szé14]{gabor-book}
G. Székelyhidi, {\em An introduction to extremal Kähler metrics}, Graduate Studies in Mathematics, \textbf{152}, American Mathematical Society, Providence, RI, 2014.

\bibitem[Tho06]{Thomas:intro GIT}
R. P. Thomas, Notes on GIT and symplectic reduction for bundles and varieties, {\em Surveys in differential geometry. Vol. X}, 221--273, {\em Surv. Differ. Geom.}, \textbf{10}, Int. Press, Somerville, MA, 2006.

\bibitem[Thu82]{thurston}
W. P. Thurston, Three-dimensional manifolds, Kleinian groups and hyperbolic geometry, {\em Bull. Amer. Math. Soc. (N.S.)} \textbf{6} (1982), no. 3, 357--381.

\bibitem[Tian87]{tian-Inv87}
G. Tian, On Kähler-Einstein metrics on certain Kähler manifolds with $C_1(M)>0$, {\em Invent. Math.} \textbf{89} (1987), no. 2, 225--246.

\bibitem[Tia90a]{tian-Invent90}
G. Tian, On Calabi's conjecture for complex surfaces with positive first Chern class, {\em Invent. Math.} \textbf{101} (1990), no. 1, 101--172.

\bibitem[Tia90b]{tian-JDG90}
G. Tian, On a set of polarized Kähler metrics on algebraic manifolds, {\em J. Differential Geom.} \textbf{32} (1990), no. 1, 99--130.

\bibitem[Tia92]{tian-IJM}
G. Tian, On stability of the tangent bundles of Fano varieties, {\em Internat. J. Math.} \textbf{3} (1992), no. 3, 401--413.

\bibitem[Tia94]{tian-CAG}
G. Tian, The K-energy on hypersurfaces and stability, {\em Comm. Anal. Geom.} \textbf{2} (1994), no. 2, 239--265.

\bibitem[Tia97]{tian-Inv}
G. Tian, Kähler-Einstein metrics with positive scalar curvature, {\em Invent. Math.} \textbf{130} (1997), no. 1, 1--37.

\bibitem[Tia00]{tian-book}
G. Tian, {\em Canonical metrics in Kähler geometry}, Notes taken by Meike Akveld, Lectures in Mathematics ETH Zürich, Birkhäuser Verlag, Basel, 2000.


\bibitem[TY90]{tian-yau-1}
G. Tian, S.-T. Yau, Complete Kähler manifolds with zero Ricci curvature. I, {\em J. Amer. Math. Soc.} \textbf{3} (1990), no. 3, 579--609.

\bibitem[TY91]{tian-yau-2}
G. Tian, S.-T. Yau, Complete Kähler manifolds with zero Ricci curvature. II, {\em Invent. Math.} \textbf{106} (1991), no. 1, 27--60. 

\bibitem[Tos17]{tosatti-Expos}
V. Tosatti, Uniqueness of $\mathbb{CP}^n$, {\em Expo. Math.} \textbf{35} (2017), no. 1, 1--12.

\bibitem[UY86]{uhlenbeck-yau}
K. Uhlenbeck, S.-T. Yau, On the existence of Hermitian-Yang-Mills connections in stable vector bundles, {\em Frontiers of the mathematical sciences: 1985 (New York, 1985)}, Comm. Pure Appl. Math. \textbf{39} (1986), no. S, suppl., S257-–S293.

\bibitem[VdV66]{vandeven}
A. Van de Ven, On the Chern numbers of certain complex and almost complex manifolds, {\em Proc. Nat. Acad. Sci. U.S.A.} \textbf{55} (1966), 1624--1627.

\bibitem[Voi07]{voisin-1}
C. Voisin, {\em Hodge theory and complex algebraic geometry. I}, Translated from the French by Leila Schneps, Reprint of the 2002 English edition, Cambridge Studies in Advanced Mathematics, \textbf{76}, Cambridge University Press, Cambridge, 2007.

\bibitem[WZ04]{wang-zhou}
X.-J. Wang, X. Zhu, Kähler-Ricci solitons on toric manifolds with positive first Chern class, {\em Adv. Math.} \textbf{188} (2004), no. 1, 87--103.

\bibitem[Wel08]{wells-book}
R. O. Wells, {\em Differential analysis on complex manifolds}, Third edition, With a new appendix by Oscar Garcia-Prada, Graduate Texts in Mathematics, \textbf{65}, Springer, New York, 2008.

\bibitem[Yau77]{yau-PNAS}
S.-T. Yau, Calabi's conjecture and some new results in algebraic geometry, {\em Proc. Nat. Acad. Sci. U.S.A.} \textbf{74} (1977), no. 5, 1798--1799.

\bibitem[Yau78]{yau-CPAM}
S.-T. Yau, On the Ricci curvature of a compact Kähler manifold and the complex Monge-Ampère equation. I, {\em Comm. Pure Appl. Math.} \textbf{31} (1978), no. 3, 339--411.

\bibitem[Zel98]{zelditch}
S. Zelditch, Szegő kernels and a theorem of Tian, {\em Internat. Math. Res. Notices} \textbf{1998} (1998), no. 6, 317--331.

\end{thebibliography}
\end{document}